 \documentclass[11pt]{amsart}
\usepackage{amssymb}
\usepackage{amscd}
\usepackage{amsmath}
\usepackage{amsmath}
\usepackage[all]{xy}
\usepackage{mathrsfs}
\usepackage{graphicx}
 \usepackage{color}

\theoremstyle{plain}
\newtheorem*{theorem*}{Theorem}
\newtheorem{theorem}{Theorem}[section]
\newtheorem{lemma}[theorem]{Lemma}

\newtheorem{proposition}[theorem]{Proposition}

\newtheorem{example}[theorem]{Example}

\newtheorem{definition}[theorem]{Definition}
\newtheorem{remark}[theorem]{Remark}

\font\russ=wncyr10  1
\def\sha{\hbox{\russ\char88}}

\DeclareMathOperator{\Det}{Det}
\DeclareMathOperator{\Ext}{Ext}
\DeclareMathOperator{\Gal}{Gal}
\DeclareMathOperator{\Fr}{Fr}

\DeclareMathOperator{\Hom}{Hom}

\DeclareMathOperator{\Spec}{Spec}
\DeclareMathOperator{\N}{N}

\DeclareMathOperator{\im}{im}

\newcommand{\CC}{\mathbb{C}}
\newcommand{\bc}{\mathbb{C}}

\newcommand{\GG}{\mathbb{G}}

\newcommand{\QQ}{\mathbb{Q}}

\newcommand{\RR}{\mathbb{R}}
\newcommand{\br}{\mathbb{R}}
\newcommand{\ZZ}{\mathbb{Z}}
\newcommand{\Z}{\mathbb{Z}}
\newcommand{\wrp}{\mathfrak{p}}

\newcommand{\Fit}{{\rm{Fit}}}

\newcommand{\co}{\mathcal{O}}

\newcommand{\OO}{\mathcal{O}}

\newcommand{\frp}{\mathfrak{p}}

\newcommand{\cok}{\text{cok}}

\newcommand{\bz}{\mathbb{Z}}

\newcommand{\bq}{\mathbb{Q}}
\newcommand{\sm}{\sigma}
\def\bigcapp{\raise1ex\hbox{\rotatebox{180}{$\biguplus$}}}
 \def\bigcappd{\raise1ex\hbox{\rotatebox{180}{$\displaystyle\biguplus$}}}

\begin{document}

\title[On higher special elements]{On higher special elements\\
 of $p$-adic representations}


\author{David Burns, Takamichi Sano and Kwok-Wing Tsoi}

\begin{abstract} As a natural generalization of the notion of `higher rank Euler system', we develop a theory of `higher special elements' in the exterior power biduals of the Galois cohomology of $p$-adic representations. We show, in particular, that such elements encode detailed information about the structure of Galois cohomology groups and are related by families of congruences involving natural height pairings on cohomology. As a first concrete application of the approach, we use it to refine, and extend, a variety of existing results and conjectures concerning the values of derivatives of Dirichlet $L$-series. 
\end{abstract}

\address{King's College London,
Department of Mathematics,
London WC2R 2LS,
U.K.}
\email{david.burns@kcl.ac.uk}


\address{Osaka City University,
Department of Mathematics,
3-3-138 Sugimoto\\Sumiyoshi-ku\\Osaka\\558-8585,
Japan}
\email{sano@sci.osaka-cu.ac.jp}

\address{King's College London,
Department of Mathematics,
London WC2R 2LS,
U.K.}
\email{kwok-wing.tsoi@kcl.ac.uk}

\thanks{Preliminary version of August 2018}


\maketitle

\tableofcontents
\section{Introduction}\label{intro}

\subsection{Background}In recent years there has been increasing interest in the study of elements that lie in the higher exterior powers (or higher exterior power biduals) of the cohomology of $p$-adic Galois representations that are defined over a number field $K$.

In the best case, one hopes that, for a given representation $T$, a collection of such elements that are defined over a family of abelian extensions of $K$ constitutes a `higher rank Euler system', is explicitly related to special values of an $L$-series associated to $T$ and leads to structural information about the Selmer modules of $T$.

However, this theory is still very much in its infancy. In particular, there are few concrete examples and, since the algebra of higher exterior powers can be rather complicated, there is little understanding of precisely what integrality properties elements in exterior power biduals can be expected to have or what information regarding the structure of Galois cohomology groups or Selmer modules that they can be expected to encode.

With this latter problem in mind, in the first part of this article we shall describe a detailed algebraic theory of `special elements' in the exterior power biduals of the Galois cohomology of $p$-adic representations.

The degree of generality in which we can develop this theory allows subsequent applications in a wide range of natural arithmetic settings, including to the  compactly supported $p$-adic cohomology groups, the finite support cohomology groups in the sense of Bloch and Kato and the cohomology groups of the Nekov\'a\v r-Selmer complexes that arise from very general $p$-adic representations.

\subsection{The general theory} To give a little more detail of our algebraic approach we assume to be given a Dedekind domain $R$ with field of fractions $F$, an $R$-order $\mathfrak{A}$ in a finite dimensional separable commutative $F$-algebra $A$ and an extension field $E$ of $F$.

Then we shall first prove several technical results in homological algebra concerning the theory of `admissible complexes' of $\mathfrak{A}$-modules that was introduced by Barrett and the first author in \cite{barretb} and then subsequently used by Macias Castillo and the first author in \cite{omac} (but see Remark \ref{terminology} for details of terminology differences). These results will play a key role in several of our later arguments and, although we shall not pursue this point any further here, they can also be used to obtain improvements in the theory of `organising matrices' that is developed \cite{omac}.

We assume now to be given a complex of $\mathfrak{A}$-modules $C$ that is both admissible and acyclic outside degrees one and two together with an isomorphism of $E\otimes_FA$-modules $\lambda: E\otimes_F H^1(C)\cong E\otimes_F H^2(C)$ and a generator $\mathcal{L}$ of the $\mathfrak{A}$-submodule of $E\otimes_RA$ that is canonically determined by $\lambda$ and the determinant (over $\mathfrak{A}$) of $C$.

Then, for each non-negative integer $a$, and each ordered $a$-tuple of elements $x_\bullet$ of $H^2(C)$ we shall define a canonical `higher special element' $\eta = \eta_{(C,\lambda, \mathcal{L},x_\bullet)}$ that belongs to the $a$-th exterior power
 of the $A$-module generated by $H^1(C)$.

In the three main algebraic results of this article (Theorems \ref{char els}, \ref{str-bidual} and \ref{mrstheorem}) we shall then investigate the key integrality properties of these higher special elements and the relations that hold between them as the data $(C,\lambda, \mathcal{L},x_\bullet)$ varies.

In particular, in Theorem \ref{char els} we show that higher special elements lead to a computation of the higher Fitting ideals of the $\mathfrak{A}$-module $H^2(C)$ and can also be used to generate annihilators of important subquotients of $H^2(C)$.

Modulo a slight difference of approach, and under certain strong hypotheses on the tuple $x_\bullet$, this sort of result was already proved in \cite{bks1}. However, the key feature of Theorem \ref{char els}, and the reason that its proof is much more complicated than that given in loc. cit., is that it imposes no special hypotheses on $x_\bullet$ and, as we shall later see, this is critical for the purposes of arithmetic applications.

Under the assumption that $\mathfrak{A}$ is Gorenstein, we shall then in Theorem \ref{str-bidual} construct a canonical perfect $\mathfrak{A}$-bilinear pairing between the quotient of ${\bigcap}^a_\mathfrak{A} H^1(C)$ by the submodule generated by any suitable linear multiple of $\eta$ and the quotient of $\mathfrak{A}$ by a canonical ideal determined by (the same linear multiple of) $\eta$. In specific cases, we find that this result gives much information about the positioning of higher special elements in exterior power biduals.

Finally, in Theorem \ref{mrstheorem}, we show that under suitable conditions, higher special elements for differing data sets can be related by congruence relations involving a canonical algebraic height pairing between the cohomology groups of $C$.

This result provides an axiomatic formulation of the `refined class number formula for $\mathbb{G}_m$' that was independently formulated by Mazur and Rubin in \cite{MR2} and by the second author in \cite{sano} and allows us to translate the insight obtained via this approach in the context of the multiplicative group to other important arithmetic settings. 

\subsection{Arithmetic applications} Turning now to consider arithmetic applications, we remark that one of the original motivating factors behind our proving the above algebraic results is that they lead to the formulation of certain precise refinements of the Birch and Swinnerton-Dyer conjecture concerning leading terms of the Artin-Hasse-Weil $L$-series that arise from an abelian variety over $K$ and the complex characters of the Galois group of a given finite abelian extension of $K$. This conjectural formalism is discussed in a subsequent article \cite{bmc2} of Macias Castillo and the first author and leads to a range of new, and very explicit, predictions. These predictions do not require any restrictive hypotheses on either the abelian variety or the abelian extension and, in particular, incorporate and extend earlier conjectures that are due, amongst others, to Mazur and Tate, to Gross and to Darmon.

However, as a first concrete illustration of the interest of our approach, in the second part of this article we shall apply it to the compactly supported $p$-adic cohomology of the representations $T = \ZZ_p(r)$ for varying integers $r$.

In this setting, we find that our theory of higher special elements extends the theory of `generalized Stark elements of arbitrary rank and weight' that is developed by Kurihara and the first and second authors in \cite{bks2-2}. In particular, therefore, when $r=0$ these elements recover the `Rubin-Stark elements' that underpin the theory of abelian Stark conjectures.

Perhaps surprisingly, even in this classical setting, our approach leads to refinements and generalisations of a wide range of existing results and, for the convenience of the reader, we have given a list of these improvements at the beginning of \S\ref{tate section}.

However, rather than discuss these applications in any further detail here, we prefer just to give two very simple, and very concrete, examples of the sort of new results that we are able to prove in this way.

To do this we fix a finite abelian extension $F/k$ of totally real fields. We set $G := \Gal(F/k)$, fix a finite set of places $S$ of $k$ containing  the sets $S_\infty(k)$ of archimedean places and $S_{\rm ram}(F/k)$ of places that ramify in $F$ and an auxiliary finite set of places $\Sigma$ of $k$ that is disjoint from $S$ and such that the multiplicative subgroup $\mathcal{O}_{F,S,\Sigma}^\times$ of $F$ comprising all elements that are integral at all primes outside $S$ and congruent to $1$ modulo all places above $\Sigma$ is torsion-free.

For the first result we assume that $k = \QQ$ and $S = S_\infty(\QQ)\cup S_{\rm ram}(F/\QQ)$ and write $f$ for the conductor of $F$ and $c_{F,\Sigma}$ for the image of the cyclotomic element ${\rm Norm}_{\QQ(\zeta_f)/F}(1- \zeta_f)$ under the action of $\prod_{\ell \in \Sigma}(1-\sigma_\ell^{-1}\cdot \ell)$, where $\zeta_f:=e^{2\pi\sqrt{-1}/f}$ and $\sigma_\ell$ is the Frobenius automorphism of $\ell$. We also write ${\rm Sel}^\Sigma_S(F)$ for the Selmer group with respect to $S$ and $\Sigma$ of $\mathbb{G}_m$ over $F$ (the definition of which is recalled at the beginning of \S\ref{recall selmer}).

For each finitely generated $G$-module $M$ we write $M_{\rm tor}$ for its torsion subgroup and, for each non-negative integer $a$, we denote its $a$-th Fitting ideal by ${\rm Fit}^a_G(M)$.

Then the following result is derived in \S\ref{proof of A} as an easy application of Theorem \ref{str-bidual}.
\medskip

\noindent{\bf Theorem A.}\,\,{\em There exists a canonical perfect $G$-invariant pairing of finite groups }
\[ (\mathcal{O}_{F,S,\Sigma}^\times/\langle c_{F,\Sigma}\rangle)_{\,{\rm tor}}\times
( \ZZ[G]/{\rm  Fit}^1_G({\rm Sel}_S^\Sigma(F)))_{\,{\rm tor}} \to \QQ/\ZZ.\]
\medskip

The observation that such pairings exist is, as far as we are aware, new (although, in the special case of Theorem A, there are direct links to previous work of Kurihara and the first two authors in \cite{bks1} - see Remark \ref{remark bks}) and can be seen to encode detailed information about the structure of the quotient of the group of algebraic units by the group of cyclotomic units. In this regard we recall that in \cite[p. 260]{lang} Lang explicitly refers to the structure of the latter quotient group as a `mystery'.

For example, in the special case that $f$ is an odd prime power the module $\mathcal{O}_{F,S,\Sigma}^\times/\langle c_{F,\Sigma}\rangle$ is finite and the pairing in Theorem A induces a canonical isomorphism 
\[ \mathcal{O}_{F,S,\Sigma}^\times/\langle c_{F,\Sigma}\rangle \cong \Hom_\ZZ(\ZZ[G]/{\rm  Fit}^0_G({\rm Cl}^\Sigma_S(F)),\QQ/\ZZ),\]
where ${\rm Cl}^\Sigma_S(F)$ is the ray class group of $\mathcal{O}_{F,S}$ modulo the product of all places above $\Sigma$ and the Pontryagin dual is endowed with the natural (rather than contragredient) action of $G$.

This isomorphism addresses the problem raised by Washington in \cite[Rem. after Th. 8.2]{washington} of explicitly relating the $G$-structures of the modules $\mathcal{O}_{F,S,\Sigma}^\times/\langle c_{F,\Sigma}\rangle$ and ${\rm Cl}^\Sigma_S(F)$ in this case and also immediately implies an equality
 ${\rm Fit}^0_G(\mathcal{O}_{F,S,\Sigma}^\times/\langle c_{F,\Sigma}\rangle) = {\rm  Fit}^0_G({\rm Cl}^\Sigma_S(F))$ that refines the main result of Cornacchia and Greither in \cite{cg}.

To give a second concrete example we fix a prime $p$, assume $S$ contains all $p$-adic places of $k$ and write $\theta^p_{F/k,S}(1)$ for the unique element of the augmentation ideal of $\CC_p[G]$ with the property that for all non-trivial homomorphisms $\rho: G \to \CC_p^\times$ one has
\begin{equation}\label{interpolation} \rho_*(\theta^p_{F/k,S}(1)) = L_{p,S}(1,\rho),\end{equation}
where $\rho_*$ is the ring homomorphism $\CC_p[G] \to \CC_p$ induced by $\rho$ and $L_{p,S}(s,\rho)$ the $S$-truncated Deligne-Ribet $p$-adic $L$-series of $\rho$. We also write ${\rm Cl}(F)_p$ for the Sylow $p$-subgroup  of the ideal class group of $F$.

We write $e_{\bf 1}$ for the idempotent $|G|^{-1}\sum_{g \in G}g$ of $\QQ[G]$ and note that, for any $G$-module $M$, the ring $\ZZ[G](1-e_{\bf 1})$ acts naturally on the quotient $M/H^0(G,M)$.

Then the following result is proved in \S\ref{proof of B} by combining several of the technical results that we obtain concerning admissible complexes with previous work of Macias Castillo and the first author in \cite{omac} and \cite{dbmc2}.

This result constitutes a rather striking analogue for the values at $s=1$ of $p$-adic $L$-functions of Brumer's classical conjecture that, in the setting of abelian CM extensions of totally real fields, the natural Stickelberger element constructed from values at $s=0$ of Artin $L$-functions should annihilate ideal class groups.
\medskip

\noindent{}{\bf Theorem B.}\,\, {\em If $F$ is totally real and either $\mu_p(F)$ vanishes or $p$ does not divide $[F:k]$, then $\theta^p_{F/k,S}(1)$ belongs to $\ZZ_p[G](1-e_{\bf 1})$ and annihilates ${\rm Cl}(F)_p/H^0(G,{\rm Cl}(F)_p)$.}
\medskip

If $F$ is abelian over $\QQ$, then $\mu_p(F)$ is known to vanish and so Theorem B is unconditional. For this reason, Theorem B is both a generalization and strong refinement of results proved (in the setting of cyclic extensions of $\QQ$) by Oriat in \cite{oriat}.

In addition, our approach leads naturally to the formulation of a precise conjectural extension of Theorem B to abelian extensions of arbitrary number fields that, in particular, specialises to give a strongly refined version of the central conjecture formulated by Solomon in \cite{sol3} (see Remark \ref{general B}).

\subsection{Structure of the article} In a little more detail, the main contents of this article is as follows. In \S\ref{the complexes} we prove certain preliminary results concerning the general categories of complexes to which our constructions and results can be applied and describe several ways in which they occur naturally in arithmetic. In \S\ref{hse sect} we introduce the key notion of `higher special element' and prove the main algebraic results concerning their integrality properties (we note that, as observed earlier, several of these algebraic results extend results from the literature and may have independent interest). As a preliminary to arithmetic applications, in \S\ref{compact-applications} we prove some necessary results concerning compactly supported \'etale cohomology. 
 Then, in \S\ref{tate section}, we give a first illustration of our results by applying them in the setting of Tate motives and deriving refinements and/or generalisations of a range of existing results.

Finally, we would like to note that much of the work presented here grew out of the King's College London PhD Thesis \cite{thesis} of the third author.

\section{Categories of complexes}\label{the complexes}

In this section we prove certain useful preliminary results concerning the category of `admissible' complexes that were introduced by Barrett and the first author in \cite{barretb}.

We shall assume that all rings are commutative and, when the context is clear, we shall often abbreviate functors of the form `$ \otimes_R$' to `$\cdot$'.


For a noetherian ring $R$ we write $D(R)$ for the derived category of $R$-modules and $D^{\rm p}(R)$ for the full triangulated subcategory of $D(R)$ comprising complexes that are perfect.

For an abelian group $N$ we write $N_{\rm tor}$ for its torsion submodule and
set $N_{\rm tf}: = N/N_{\rm tor}$, which we regard as embedded in the associated space $\QQ\cdot N$.

\subsection{Admissible complexes} We fix a
Dedekind domain $R$ of characteristic $0$ with field of fractions
$F$ and an $R$-order $\mathfrak{A}$ that spans a separable $F$-algebra $A := F\otimes_R\mathfrak{A}$.

\subsubsection{}\label{admissible definition}We define the category $D^{\rm a}(\mathfrak{A})$ of `admissible perfect complexes of $\mathfrak{A}$-modules' to be the full subcategory of $D(\mathfrak{A})$ comprising complexes $C = (C^i)_{i \in \ZZ}$ that satisfy the following four conditions:
%
%
\begin{itemize}
\item[(ad$_1$)] $C$ is an object of $D^{{\rm p}}(\mathfrak{A})$;
\item[(ad$_2$)] the Euler characteristic of $A\otimes_{\mathfrak{A}}C$ in the Grothendieck group $K_0(A)$
vanishes;
\item[(ad$_3$)] $C$ is acyclic outside degrees one, two and three;
\item[(ad$_4$)] $H^1(C)$ is $R$-torsion-free.
\end{itemize}

We shall say that an object of $D^{\rm a}(\mathfrak{A})$ that is acyclic outside degrees one and two is said to be a `strictly admissible perfect complex of $\mathfrak{A}$-modules' and write $D^{{\rm s}}(\mathfrak{A})$ denotes the full subcategory of $D^{\rm a}(\mathfrak{A})$ comprising such  complexes.

\begin{remark}\label{semisimplicity}{\em Since the $F$-algebra $A$ is a finite product of fields, the assumptions (ad$_2$) and (ad$_3$) combine to imply there is an isomorphism of $A$-modules
\[ A\otimes_\mathfrak{A}H^2(C) \cong H^2(A\otimes_\mathfrak{A}C) \cong
H^1(A\otimes_\mathfrak{A}C)\oplus H^3(A\otimes_\mathfrak{A}C)\cong A\otimes_\mathfrak{A}(H^1(C)\oplus H^3(C)).\]
}\end{remark}

\begin{remark}\label{representative}{\em Let $\wrp$ be a prime ideal of $R$ and $A'_\wrp$ a local component of the semi-local algebra $A_\wrp$. Then, for each complex $C$ in $D^{\rm a}(\mathfrak{A})$, respectively in $D^{\rm s}(\mathfrak{A})$, the complex $C_\wrp' := \mathfrak{A}_\wrp'\otimes_\mathfrak{A}C$ belongs to
 $D^{\rm a}(\mathfrak{A}'_\wrp)$, respectively $D^{\rm s}(\mathfrak{A}'_\wrp)$. In particular, since each finitely generated torsion-free $A_\frp'$-module has finite projective dimension if and only if it is free, a standard argument of homological algebra combines with the assumptions (ad$_1$), (ad$_3$) and (ad$_4$) to imply that for each $C$ in $D^{\rm a}(\mathfrak{A})$ the complex $C_\wrp'$ is isomorphic in $D(\mathfrak{A}'_\wrp)$ to a complex of $\mathfrak{A}'_\wrp$-modules
\begin{equation}\label{exp rep com} P^1 \xrightarrow{d^1} P^2 \xrightarrow{d^2} P^3\end{equation}
in which each module is both finitely generated and free and the first term is placed in degree one. If $C$ belongs to $D^{\rm s}(\mathfrak{A})$, then one can in addition take the module $P^3$ to be zero. }
\end{remark}

\begin{remark}\label{terminology}{\em The categories $D^{{\rm a}}(\mathfrak{A})$ and $D^{{\rm s}}(\mathfrak{A})$ also play a key role in the theory of organising matrices that is developed by Macias Castillo and the first author in \cite{omac}. However, we caution the reader that there is a slight difference of terminology  in that objects of the categories $D^{{\rm a}}(\mathfrak{A})$ and $D^{{\rm s}}(\mathfrak{A})$ defined above are in loc. cit. respectively referred to as `weakly admissible' and `admissible' complexes.}\end{remark}

\subsubsection{}In this section we record two useful ways in which admissible complexes give rise to new families of admissible complexes.

\begin{lemma}\label{cone construct} We assume to be given the following data:
\begin{itemize}
\item[(a)] an object $C$ of $D^{{\rm a}}(\mathfrak{A})$, and
\item[(b)] a finitely generated projective $\mathfrak{A}$-module $P$ and a homomorphism of $\mathfrak{A}$-modules
\[ \theta^i: P \to H^i(C)\]%
for $i=1$ and $i=2$ where $\theta^1$ is both injective and such that ${\rm cok}(\theta^1)$ is $R$-torsion-free.
\end{itemize}
Then there is an exact triangle in $D(\mathfrak{A})$ of the form

\begin{equation}\label{stated form} P[-1]\oplus P[-2] \xrightarrow{\theta} C \to D \to P[0]\oplus P[-1]\end{equation}
in which $\theta$ is the unique morphism with $H^i(\theta) = \theta^i$ for $i = 1,2$ and $D$ belongs to $D^{{\rm a}}(\mathfrak{A})$.  \end{lemma}

\begin{proof} Since $P$ is projective the natural passage to cohomology map
\[ \Hom_{D^{\rm p}(\mathfrak{A})}(P[-1]\oplus P[-2] ,C) \to \Hom_{\mathfrak{A}}(P,H^1(C))\oplus \Hom_{\mathfrak{A}}(P,H^2(C))\]
is bijective and hence there exists a unique morphism $\theta$ with the stated properties.

Choose $D$ to be any complex that lies in an exact triangle (\ref{stated form}). Then it is clear that $D$ belongs to $D^{\rm p}(\mathfrak{A})$ and that the Euler characteristic of $A\otimes_{\mathfrak{A}}D$ in $K_0(A)$ vanishes (since this is true for both $C$ and $P[-1]\oplus P[-2]$). In addition, there long exact cohomology sequence of (\ref{stated form}) has the form
\begin{equation}\label{les sigma} P \xrightarrow{\theta^1} H^1(C) \to H^1(D) \to P \xrightarrow{\theta^2} H^2(C) \to H^2(D) \to 0 \to H^3(C) \to H^3(D)\to 0\end{equation}
and so implies immediately that $D$ is acyclic outside degrees one, two and three and combines with the assumption that $\theta^1$ is injective and such that ${\rm cok}(\theta^1)$ is $R$-torsion-free to imply that $H^1(D)$ is also $R$-torsion-free. \end{proof}


In the next result we assume to be given a homomorphism of $R$-orders $\mathfrak{B} \to \mathfrak{A}$ and use it to regard each $\mathfrak{A}$-module $M$ as a $\mathfrak{B}$-module. We then regard the linear dual $\Hom_\mathfrak{B}(M,\mathfrak{B})$ of any such $M$ as an $\mathfrak{A}$-module by the rule $a(f)(m) := f(a(m))$ for each $a$ in $\mathfrak{A}$, $f$ in $\Hom_\mathfrak{B}(M,\mathfrak{B})$ and $m$ in $M$.

We say that $\mathfrak{A}$ is `everywhere locally Gorenstein relative to $\mathfrak{B}$' if for each prime ideal $\mathfrak{p}$ of $R$ the $R_\mathfrak{p}\otimes_R\mathfrak{A}$-module $R_\mathfrak{p}\otimes_R\Hom_\mathfrak{B}(\mathfrak{A},\mathfrak{B})$ is free of rank one.

\begin{lemma}\label{dual preserve} Let $\mathfrak{B} \to \mathfrak{A}$ be a homomorphism of $R$-orders that satisfies all of the following three conditions.
\begin{itemize}
\item[(a)] $\mathfrak{A}$ is a projective $\mathfrak{B}$-module.
\item[(b)] $\mathfrak{A}$ is everywhere locally Gorenstein relative to $\mathfrak{B}$.
\item[(c)] $\mathfrak{B}$ is Gorenstein.
\end{itemize}

Then the functors $C \mapsto {\rm R}\Hom_{\mathfrak{B}}(C,\mathfrak{B}[-4])$ and $C \mapsto {\rm R}\Hom_{\mathfrak{B}}(C,\mathfrak{B}[-3])$ respectively preserve the categories $D^{\rm a}(\mathfrak{A})$ and $D^{\rm s}(\mathfrak{A})$.
\end{lemma}

\begin{proof} For any $\mathfrak{A}$-module $Q$ we set $Q^* := \Hom_\mathfrak{B}(Q,\mathfrak{B})$.

Then the assumption (b) implies that for any finitely generated projective $\mathfrak{A}$-module $Q$ the $\mathfrak{A}$-module $Q^*$ is finitely generated and projective. This implies that the  category $D^{\rm p}(\mathfrak{A})$ is preserved by the functor $C \mapsto C^* := {\rm R}\Hom_{\mathfrak{B}}(C,\mathfrak{B})$.

To compute the cohomology groups of $C^*$ we claim first that for any finitely generated $\mathfrak{A}$-modules $N$ and for all integers $j > 1$ the group ${\rm Ext}^j_\mathfrak{B}(N,\mathfrak{B})$ vanishes. To prove this we note that any exact sequence of $\mathfrak{A}$-modules of the form
\[ 0 \to N' \to \mathfrak{A}^d \to N\to 0,\]
induces, as a consequence of assumption (a), an isomorphism of the form ${\rm Ext}^j_\mathfrak{B}(N,\mathfrak{B}) \cong {\rm Ext}^{j-1}_\mathfrak{B}(N',\mathfrak{B})$. Thus, since each such module $N'$ is torsion-free over $R$, the vanishing of ${\rm Ext}^j_\mathfrak{B}(N,\mathfrak{B})$ can be reduced, by a downward induction on $j$, to the vanishing of ${\rm Ext}^1_\mathfrak{B}(N,\mathfrak{B})$ for every finitely generated $\mathfrak{B}$-module $N$ that is torsion-free over $R$. It is then enough to note that the latter condition is equivalent to condition (c) (cf. \cite[\S 37]{CR} and \cite[Prop. 6.1]{DKR}).

Then, since the groups ${\rm Ext}^j_\mathfrak{B}(H^i(C),\mathfrak{B})$ vanish for all integers $j > 1$ and all integers $i$, the universal coefficient spectral sequence implies that in each degree $i$ there is a canonical short exact sequence
\begin{equation}\label{ucss seq} 0 \to {\rm Ext}^1_\mathfrak{B}(H^{1-i}(C)_{\rm tor},\mathfrak{B}) \to H^i(C^*) \to H^{-i}(C)^*\to 0.\end{equation}
Here we have also used the fact that the tautological exact sequence
\[ 0 \to H^{1-i}(C)_{\rm tor} \to H^{1-i}(C) \to H^{1-i}(C)_{\rm tf} \to 0\]
combines with assumption (c) to induce an isomorphism
\[ {\rm Ext}^1_\mathfrak{B}(H^{1-i}(C),\mathfrak{B})\cong {\rm Ext}^1_\mathfrak{B}(H^{1-i}(C)_{\rm tor},\mathfrak{B}).\]

By using the sequence (\ref{ucss seq}) to compute the cohomology groups of $C^*[-3]$ and $C^*[-4]$ one finds that the functors $C \mapsto C^*[-4]$ and $C \mapsto C^*[-3]$ respectively preserve the categories $D^{\rm a}(\mathfrak{A})$ and $D^{\rm s}(\mathfrak{A})$, as claimed.\end{proof}

\begin{example}\label{dual exams} {\em There are several natural cases in which the assumptions of Lemma \ref{dual preserve} are satisfied. \

\noindent{}(i) Let $\mathfrak{B} \to \mathfrak{A}$ be the tautological homomorphism $R \to \mathfrak{A}$. Then the assumptions (a) and (c) are automatically satisfied and assumption (b) is satisfied if, for example, $\mathfrak{A}$ is a direct factor of the group ring $R[\Gamma]$ for a finite abelian group $\Gamma$.

\noindent{}(ii) If $\mathfrak{B} \to \mathfrak{A}$ is the identity map $\mathfrak{A} \to \mathfrak{A}$, then the conditions (a), (b) and (c) are all satisfied if and only if $\mathfrak{A}$ is Gorenstein.}
 \end{example}

%
%

\subsection{Arithmetic examples}\label{arithmetic exams} 
To describe arithmetic examples of the classes of complexes defined above we set $G_{L/K} := \Gal(L/K)$ for any Galois extension of fields $L/K$. We also fix an algebraic closure $K^c$ of $K$ and set $G_K := G_{K^c/K}$.

For any finite group $\Gamma$ and any $\ZZ_p[\Gamma]$-module $N$ we write $N^\vee$ for the Pontryagin dual $\Hom_{\ZZ_p}(N,\QQ_p/\ZZ_p)$ which we endow with the usual contragredient action of $\ZZ_p[\Gamma]$.

\subsubsection{The $p$-adic representations}

Let $k$ be a number field. We fix a finite set of places $S$ of $k$ that contains all archimedean places and all $p$-adic places of $k$.

Then in this section we assume to be given a pair $(\mathfrak{A},T)$ comprising  a $\ZZ_p$-order $\mathfrak{A}$ that spans a separable $\QQ_p$-algebra $A$ and a continuous $\ZZ_p[G_k]$-module $T$ that is unramified outside $S$ and is endowed with a commuting action of $\mathfrak{A}$ with respect to which it is a projective module.

We endow the Kummer dual $T^*(1) := \Hom_{\ZZ_p}(T,\ZZ_p(1))$ with the following
commuting actions of $\mathfrak{A}$ and $G_{k}$: for each $a
\in \mathfrak{A}$, $\sigma \in G_{k}$, $f \in T^*(1)$ and $t
\in T$ one has $a(f)(t) = f(a(t))$ and $\sigma(f)(t) =
\sigma(f(\sigma^{-1}(t))$.

%

\begin{remark}\label{base change remark}{\em Let $T$ be a finitely generated free $\ZZ_p$-module that is endowed with a continuous action of $G_k$ unramified outside $S$. Then there are two natural ways in which $T$ gives rise to data $(\mathfrak{A},T')$ of the above sort.

\noindent{}(i) Fix a finite abelian extension of number fields $F/k$ that is unramified outside $S$ and set $\mathfrak{A} := \ZZ_p[G_{F/k}]$. Then  the tensor product $T_F := \mathfrak{A}\otimes_{\ZZ_p}T$
 becomes a module over $\mathfrak{A}\times \ZZ_p[G_k]$ in the following way: the action of $\mathfrak{A}$ is induced by letting $G_{F/k}$ act via multiplication on the left and each
$u\in G_k$ acts as $x\otimes_{\ZZ_p}t \mapsto x\bar{u}^{-1}\otimes_{\ZZ_p}u(t)$ for $x \in\mathfrak{A}$ and $t \in T$ where $\bar{u}$ denotes the image of $u$ in $G_{F/k}$. Then, with respect to this action $T_F$ is both unramified outside $S$ and a free $\mathfrak{A}$-module.

\noindent{}(ii) If $T$ is endowed with an action of any order $\mathfrak{A}$ that is both regular and such that $V$ spans a free $A$-module, then $T$ is automatically a free $\mathfrak{A}$-module. }
\end{remark}

\subsubsection{Compactly supported \'etale cohomology}\label{csc exam} In the sequel we write $\mathcal{O}_{k,S}$ for the subring of $k$ comprising elements that are integral at all places outside $S$ and $\kappa_v$ for the residue field of a non-archimedean place $v$ of $k$.

In \S\ref{compact-applications} we recall the definition of the compactly supported \'etale cohomology complex
\[ C(T) := R\Gamma_c(\mathcal{O}_{k,S},T)\]
of the $p$-adic representation $T$ fixed above and prove that it has the properties recorded in the following result.

This result involves the (Bloch-Kato) Selmer group ${\rm Sel}(T^*(1))$ and Tate-Shafarevich group $\sha(T)$, the definitions of which will be recalled in \S\ref{sha recall}.

\begin{proposition}\label{admiss constructions} For any data $(\mathfrak{A},T)$ as above, the following claims are valid.
\begin{itemize}
\item[(i)] $C(T)$ belongs to $D^{\rm a}(\mathfrak{A})$.
\item[(ii)] Assume that the diagonal localisation homomorphism
\[ H^1(\mathcal{O}_{k,S},T)_{\rm tor} \to \bigoplus_{v \in S}H^1(k_v,T)_{\rm tor}\]
is injective. Then for any homomorphism
\[ \phi: \bigoplus_{v \in S_\infty}H^0(k_v,T)
\to \bigoplus_{v \in S_p}H^1_f(k_v,T)\]
of $\mathfrak{A}$-modules the data $(C(T),\phi)$ gives rise via the construction in Lemma \ref{cone construct} to a canonical object $C_{\phi}(T)$ of $D^{\rm a}(\mathfrak{A})$ with the property that ${\rm Sel}(T^*(1))^\vee$, and hence also $\sha(T)$,  is isomorphic to a subquotient of $H^2(C_{\phi}(T))$.
\item[(iii)] Let $\Sigma$ be any finite non-empty set of places of $k$ that is disjoint from $S$. Then there exists a natural morphism in $D(\mathfrak{A})$ from $\bigoplus_{v \in \Sigma}R\Gamma(\kappa_v,T(-1))[-2]$ to $C(T)$, respectively to $C_\phi(T)$ for any morphism $\phi$ as in claim (ii). The mapping cone $C_\Sigma(T)$, respectively $C_{\phi,\Sigma}(T)$, of this morphism belongs to $D^{\rm s}(\mathfrak{A})$ and the modules ${\rm Sel}(T^*(1))^\vee$ and $\sha(T)$ are isomorphic to subquotients of $H^2(C_{\phi,\Sigma}(T))$.
\end{itemize}
\end{proposition}

\begin{remark}{\em In the case that $T = \ZZ_p(r)_F$ and $\mathfrak{A} = \ZZ_p[G_{F/k}]$ for any integer $r$ and any finite Galois extension $F$ of $k$ the construction in Proposition \ref{admiss constructions}(iii) plays a key role in the article of Kurihara and the first and second authors in \cite{bks2-2} and this case is considered in greater detail in \S\ref{tate section}. In addition, in \cite{bmc2} it is shown that, if $T$ is the $p$-adic Tate module of a critical motive, then the construction in Proposition \ref{admiss constructions}(iii) plays an important role in motivating certain refined conjectures of Birch and Swinnerton-Dyer type for Artin-Hasse-Weil $L$-series.}\end{remark}

\begin{remark}{\em The displayed localisation map in Proposition \ref{admiss constructions}(ii) is injective, for example, whenever the space $H^0(k_v,\QQ_p\otimes_{\ZZ_p} T)$ vanishes for each $v$ in $S\setminus S_\infty$.}\end{remark}

\subsubsection{Finite Support Cohomology}\label{fc exam} Let $A$ be an abelian variety defined over $k$ and write $A^t$ for its dual. Fix a finite abelian extension $F/k$ that is unramified outside $S$ and for each intermediate field $K$ of $F/k$ write $S_p^K$, $S^K_{\rm
r}$ and $S^K_{\rm b}$ for the finite sets of places of $K$ which are $p$-adic, which lie above a place of $k$ that ramifies in $F/k$ and at which $A_{/K}$ has bad reduction respectively. For each $v$ in $S^K_{\rm b}$ write $c_v(A_{/K})$ for the Tamagawa factor at $v$ that occurs in the Birch-Swinnerton-Dyer conjecture for $A_{/K}$. We write $\sha(A_F)$ and ${\rm Sel}_p(A_{/F})$ for the (classical) Tate-Shafarevic and $p$-primary Selmer groups of $A$ over $F$ and
\[ C_f(T_{F}) := R\Gamma_f(k,T_{F})\]
for the (global) finite support cohomology of Bloch and Kato.

The following result is due to Macias Castillo, Wuthrich and the first author (and is discussed more fully in \cite[\S2.2.3]{omac}).

\begin{proposition}\label{admiss constructions2} Let $K$ be the intermediate field of $F/k$ corresponding to the $p$-Sylow subgroup of $G_{F/k}$.

\begin{itemize}
\item[(i)] $C_f(T_{F})$ belongs to $D^{\rm a}(\ZZ_p[G_{F/k}])$ if each of the following hypotheses is satisfied:
\begin{itemize}
\item[(a)] $p$ is odd and does not divide $
|A(K)_{\rm tor}|, |A^t(K)_{\rm tor}|, c_w(A_{/K})$ for any $w\in S^K_{\rm b}$ and $|A(\kappa_v)|$ for any $v \in S^K_{\rm r}$;
\item[(b)]  $A_{/K}$ has good reduction at all places in $S_p^K$ and at any place in $S_p^K \cap S_{\rm r}^K$ the reduction is also ordinary;
\item[(c)] $S^K_{\rm r}\cap S_{\rm b}^K = \emptyset$.
\end{itemize}
\item[(ii)] If $\sha(A_F)$ is finite, then $H^1(C_f(T_{F}))$ and $H^2(C_f(T_{F}))$ are canonically isomorphic to $\ZZ_p\otimes_\ZZ A^t(F)$ and ${\rm Sel}_p(A_{/F})^\vee$ respectively.
\end{itemize}
\end{proposition}


\subsubsection{Nekov\'a\v r-Selmer Complexes}\label{cm exam} To consider critical motives beyond those associated to the restricted  class of abelian varieties discussed in Proposition \ref{admiss constructions2} it is natural to use the theory of Nekov\'a\v r-Selmer Complexes.

To discuss this we write $V$ for the $p$-adic realisation of a critical motive $M$ that is defined over $k=\QQ$ and has good ordinary reduction at $p$. We recall that in this case Perrin-Riou \cite{pr} has shown that there exists a unique $G_{\QQ_p}$-stable $\QQ_p$-subspace $V^0$ of $V$ with the property that the composite homomorphism $D_{\rm dR}(\QQ_p,V^0) \to D_{\rm dR}(\QQ_p,V)\twoheadrightarrow
t_p(V)$ induces an identification of $D_{\rm dR}(\QQ_p,V^0)$ with $t_p(V)$.

In particular, if we fix a full $G_\QQ$-stable $\ZZ_p$-sublattice $T$ of $V$, then we obtain a full
$G_{\QQ_p}$-stable $\ZZ_p$-sublattice of $V^0$ by setting $T^0:=T\cap V^0$.

We also fix a finite totally real abelian extension $F$ of $\QQ$ and a finite set of places $S$ of $\QQ$ that contains $\infty, p$, all primes that ramify in $F/\QQ$ and all at which $M$ has bad reduction.

In this setting we use the induced representations $T_F$ and $T_F^0$ that are constructed as in Remark \ref{base change remark}(i) and we endow each of the modules $V_F := \QQ_p\otimes_{\ZZ_p}T_F$, $V^0_F := \QQ_p\otimes_{\ZZ_p}T^0_F$, $W_F := V_F/T_F$ and $W^*(1)_F := V_F^*(1)/T_F^*(1) \cong \Hom_{\ZZ_p}(T_F,(\QQ_p/\ZZ_p)(1))$ with the actions of $G_{F/\QQ}$ and $G_\QQ$ that are
induced from the respective actions on $T_F$ and $T^*(1)_F$.

We then write $C(T_F,T_F^0) := {\rm SC}(\ZZ_S,T_F,T_F^0)$ for the Nekov\'a\v r-Selmer complex constructed by Fukaya and Kato in \cite[\S4.1.2.]{fukaya-kato}.

The relevance of our theory to such complexes is explained by the next result. This result is proved by the argument of Barrett and the first author  in \cite[Prop. 4.1]{barretb} and relies heavily on computations made in \cite{fukaya-kato}.

\begin{proposition}\label{admiss constructons3} Set $G:= G_{F/\QQ}$ and assume the spaces $H^0(\bq_p, V_F/V_F^0)$, $H^0(\bq_p,(V_F^0)^*(1))$ and $H^0(\bq_\ell,V_F)$ for each prime $\ell \notin S$ all vanish.

\begin{itemize}
\item[(i)] If $H^0(\bq,W_F)$ and $H^0(\bq,W^*(1)_F)$ vanish, then $C(T_F,T_F^0)$ is an object of $D^{\rm s}(\bz_p[G])$.
\item[(ii)] The $\ZZ_p[G]$-modules $\sha(T_F)$ and ${\rm Sel}(T^*(1)_F)^\vee$ are respectively isomorphic to subquotients of $H^2(C(T_F,T_F^0))_{\rm tor}$ and $H^2(C(T_F,T_F^0))$.
\item[(iii)]  There is a canonical identification $\bq_p\otimes_{\bz_p}H^2(C(T_F,T_F^0)) = \bq_p\otimes_{\bz_p}{\rm Sel}(T^*(1)_F)^\vee$.
\end{itemize}
\end{proposition}

\begin{remark}{\em The spaces $H^0(\bq_p, V_F/V_F^0)$, $H^0(\bq_p,(V_F^0)^*(1))$ and $H^0(\bq_\ell,V_F)$ will vanish if, for example, $M$ is pure with weight not equal to either $0$ or $-2$. }
\end{remark}

\begin{remark}{\em The statement of \cite[Prop. 4.1(i)]{barretb} is incorrect as stated since, in the context of loc. cit., the condition (ad$_4$)  need not be satisfied unless one also assumes the vanishing of $H^0(\bq,W_F)$. The first author would like to apologise for this error. }\end{remark}

\subsection{Controlling the cohomology of highest degree}\label{control section} In this section we prove that every admissible complex gives rise to natural families of strictly admissible complexes.

To do this we assume to be given an object $C$ of $D^{\rm a}(\mathfrak{A})$ and we write $e_0 = e_{C,0}$ for the sum of all primitive idempotents $e$ of $A$ which are such that the module $e(F\cdot H^3(C))$ vanishes.

We then also write $\mathfrak{A}_0 = \mathfrak{A}_{C,0}$ for the order $\mathfrak{A}e_0$ and $C_0$ for the object $\mathfrak{A}_0\otimes^{\mathbb{L}}_{\mathfrak{A}}C$ of $D(\mathfrak{A}_0)$. We set $A_0: = Ae_0$.

We can now state the main result to be proved in this section.

\begin{proposition}\label{reduction} Let $C$ be an object of $D^{\rm a}(\mathfrak{A})$.

Then, for each prime ideal $\wrp$ of $R$, there exists a set of generators $\mathcal{F}_\wrp$ of the $\mathfrak{A}_\wrp$-module $
{\rm Fit}_\mathfrak{A}^0(H^3(C))\cdot \mathfrak{A}_{0,\wrp}$ that are each invertible in $A_{0,\wrp}$ and are such that for each $x$ in $\mathcal{F}_\wrp$ there is a complex $C_{x}$ in $D^{\rm s}(\mathfrak{A}_{0,\wrp})$ with all of the following properties:
\begin{itemize}
\item[(i)] $H^1(C_{x}) = H^1(C_{0,\wrp})$. 
\item[(ii)] $H^2(C_{x})$ contains $H^2(C_{0,\wrp})$ as a submodule of finite index and the quotient module $H^2(C_{x})/H^2(C_{0,\wrp})$ is annihilated by $x$.
\item[(iii)] $e_0\cdot{\rm Det}_{\mathfrak{A}}(C)_\wrp = x\cdot{\rm Det}_{\mathfrak{A}_{0,\wrp}}(C_{x})$.
%
\end{itemize}
\end{proposition}

\begin{proof} If necessary, we can replace $R$ by $R_\wrp$, $\mathfrak{A}$ by a local component $A_\wrp'$ of $A_\wrp$ and $C$ by $\mathfrak{A}_\wrp'\otimes_\mathfrak{A}C$ to assume in the sequel that $\mathfrak{A}$ is local.

Then, as $C$ belongs to $D^{\rm a}(\mathfrak{A})$, the observation in Remark \ref{representative} implies that $C$ is represented by a complex of finitely generated free $\mathfrak{A}$-modules of the form (\ref{exp rep com}).

Having chosen such a representative, for each $i \in \{1,2,3\}$ we write $s_i$ for the rank of the $\mathfrak{A}$-module $P^i$ and fix a basis  $\{x^i_j\}_{1\le j\le s_i}$ of this module. We assume, as we may, that $s_2 \ge s_3$ and we choose an $s_3\times s_3$ minor $M := (m_{ij})_{1\le i,j\le s_3}$ of the matrix of $d^2$ with respect to these bases.

Then for each integer $j$ with $1\le j\le s_3$ the element $\sum_{i=1}^{i=s_3}m_{ji}\cdot x^3_i$ belongs to $\im(d^2)$ and we fix a pre-image $b_j$ in $P^2$  of it under $d^2$. We also fix a multiple $n$ of $|H^3(C)|$ and an element $c_j$ of $P^2$ with $d^2(c_j) = n\cdot x_j^{3}$.

We write $\phi = \phi_{\{b_\bullet\},\{c_\bullet\}}$ for the homomorphism $P^3\to P^2$ of $\mathfrak{A}$-modules that sends each element $x_j^3$ to $b_j + c_j$ and consider the following commutative diagram

\begin{equation}\label{triangle diag}\begin{CD}
P^1 @> d^1 >> P^2 @> d^2>> P^3\\
@V({\rm id},0)VV @\vert \\
P^1\oplus P^3 @> (d^1,\phi) >> P^2.
\end{CD}\end{equation}

We write $C_\phi$ for the complex given by the lower row of the diagram, with the first term placed in degree one. Then the diagram constitutes a morphism $\theta: C \to C_\phi$ of complexes of $\mathfrak{A}$-modules, the mapping cone of which is the upper complex in the following morphism of complexes

\begin{equation}\label{triangle diag2}\begin{CD} P^1 @> ({\rm id},0,-d^1) >> (P^1\oplus P^3)\oplus P^2 @> (((d^1,\phi),0)+ (0,{\rm id}),-d^2) >> P^2\oplus P^3\\
 @. @V (0,{\rm id},0) VV @VV (d^2,{\rm id}) V \\
 @. P^3 @> d^2\circ \phi>> P^3.\end{CD}\end{equation}
Here the first term in the upper complex is placed in degree zero and an explicit check shows that the two vertical arrows constitute a quasi-isomorphism of complexes.

Writing ${\rm cone}(\theta)'$ for the lower complex in (\ref{triangle diag2}), one therefore obtains an exact triangle
\[ C \to C_\phi \to {\rm cone}(\theta)' \to C[1]\]
in $D(\mathfrak{A})$ and hence an associated long exact sequence of cohomology
\begin{equation*}\label{les}0 \to H^1(C) \to H^1(C_\phi) \to \ker(d^2\circ \phi) \to H^2(C) \to H^2(C_\phi) \to {\rm cok}(d^2\circ \phi) \to  H^3(C)\to 0.\end{equation*}
%


Now, with respect to the given basis $\{x^3_j\}_{1\le j\le s_3}$ of $P^3$, the matrix of $d^2\circ \phi$ is  $M + n\cdot I$, with $I$ the $s_3\times s_3$ identity matrix. In particular by choosing $n$ large enough we can ensure both that $d^2\circ \phi$ is injective, and hence that ${\rm cok}(d^2\circ \phi)$ is finite and ${\rm cone}(\theta)'$ is isomorphic in $D(\mathfrak{A})$ to ${\rm cok}(d^2\circ \phi)[-2]$, and that
\[ {\rm det}(M + n\cdot I) \equiv {\rm det}(M)\,\, \text{ modulo }\, \wrp\cdot |H^3(C)|\cdot {\rm Fit}_{\mathfrak{A}}^0(H^3(C)).\]

This last observation shows, in particular, that as we vary the choice of minor $M$, the elements ${\rm det}(d^2\circ \phi) =
{\rm det}(M + n\cdot I)$ are invertible in $A$ and constitute a set $\mathcal{F}$ of generators of ${\rm Fit}^0_{\mathfrak{A}}(H^3(C))$.

For each minor $M$ we then set $x := {\rm det}(d^2\circ \phi)$ and $C_x := C_\phi$ and $Q_x := {\rm cok}(d^2\circ \phi)$.

With these choices, claim (i) is clear. To prove claim (ii) we note the above construction gives an exact triangle
\[C\to C_x \to Q_x[-2] \to C[1]\]
in $D^{\rm p}(\mathfrak{A})$ and hence gives rise to an equality of invertible $\mathfrak{A}$-modules
\begin{equation}\label{C and C_x} {\rm Det}_{\mathfrak{A}}(C_x) = {\rm Det}_{\mathfrak{A}}(C)\cdot {\rm Det}_{\mathfrak{A}}(Q_x[-2]).
\end{equation}

It is thus enough to note that the exact sequence of $\mathfrak{A}$-modules
\begin{equation}\label{resolution}0 \to P^3\xrightarrow{d^2\circ\phi} P^3 \to Q_x\to 0\end{equation}
implies that (the ungraded part of) ${\rm Det}_{\mathfrak{A}}(Q_x[-2])$ is generated over $\mathfrak{A}$ by the inverse of the invertible element
${\rm det}(d^2\circ\phi) =: x$. Now (iii) follows from this and the equality (\ref{C and C_x}).
\end{proof}

The next result describes the cohomology groups of the complex $\mathfrak{A}_0\otimes^{\mathbb{L}}_{\mathfrak{A}}C$.

\begin{lemma}\label{spec seq exact} Let $C$ be an object of $D^{\rm a}(\mathfrak{A})$ and set $C_0:= \mathfrak{A}_0\otimes^{\mathbb{L}}_{\mathfrak{A}}C$.

\begin{itemize}
\item[(i)] $C_0$ belongs to $D^{\rm a}(\mathfrak{A}_0)$ and $H^3(C_0)$ identifies with the finite module $\mathfrak{A}_0\otimes_{\mathfrak{A}}H^3(C)$.
\item[(ii)] Assume $\mathfrak{A} = R[\Gamma]$ for a finite abelian group $\Gamma$ and for any $\mathfrak{A}_0$-submodule $I$ of $A$ and $R[\Gamma]$-module $M$ endow the tensor product $I\otimes_RM$ with the action of $\mathfrak{A}_0\times R[\Gamma]$ under
which $a_0$ in $\mathfrak{A}_0$  acts as $i\otimes_Rm \mapsto
a_0i\otimes_Rm$ and $\gamma$ in $\Gamma$ as $i\otimes_Rm \mapsto
i\gamma^{-1}\otimes_R\gamma (m)$. Then $H^1(C_0)$ identifies with $H^0(\Gamma,\mathfrak{A}_0\otimes_R H^1(C))$ and there is a natural exact sequence of $\mathfrak{A}_0$-modules of the form
\begin{equation*}\label{exactseq} H^1(\Gamma,\mathfrak{A}_0\otimes_R H^1(C)) \to
H^2(C_0) \to
 H^0(\Gamma,\mathfrak{A}_0\otimes_RH^2(C)) \to
 H^2(\Gamma,\mathfrak{A}_0\otimes_R H^1(C)).
 \end{equation*}
\end{itemize}
\end{lemma}

\begin{proof} Since $C$ is admissible it is easy to see $C_0$ belongs to $D^{\rm p}(\mathfrak{A}_0)$, is acyclic outside degrees one, two and three and that its top degree of cohomology $H^3(C_0)$ identifies with $\mathfrak{A}_0\otimes_{\mathfrak{A}}H^3(C)$.

In addition, $H^1(C_0)$ is torsion-free over $R$ since if for each prime ideal $\wrp$ of $R$ we fix a representative of $C_\wrp$ of the form (\ref{exp rep com}), then $H^1(C_0)_\wrp$ is a submodule of the torsion-free module $\mathfrak{A}_{0,\wrp}\otimes_\mathfrak{A}P$. This shows that $C_0$ belongs to $D^{\rm a}(\mathfrak{A}_0)$ and so proves claim (i).

To prove claim (ii) we note that, as $\mathfrak{A}_0$ is $R$-free, for any projective $\mathfrak{A}$-module
 $Q$ the $\Gamma$-module $\mathfrak{A}_0\otimes_RQ$ is cohomologically-trivial  and so the map
 $a\otimes_Rq \mapsto \sum_{\gamma \in \Gamma}\gamma(a\otimes_Rq)$ induces an isomorphism of $\mathfrak{A}_0$-modules
 $\mathfrak{A}_0\otimes_\mathfrak{A}Q = H_0(\Gamma,\mathfrak{A}_0\otimes_RQ)\cong
H^0(\Gamma,\mathfrak{A}_0\otimes_RQ)$. Such isomorphisms give rise to a
convergent cohomological spectral sequence of the form
$H^b(\Gamma,\mathfrak{A}_0\otimes_R H^a(C)) \Rightarrow
H^{b+a}(C_0)$.

Since $H^a(C)$ vanishes for $a < 1$ this spectral sequence induces a canonical isomorphism
$\mathfrak{A}_0$-modules $H^1(C_0) \cong H^0(\Gamma,\mathfrak{A}_0\otimes_R H^1(C))$ and an exact sequence of low degree terms (which coincides with displayed sequence in the statement above). \end{proof}



\section{Higher special elements}\label{hse sect}

As in earlier sections, we fix a Dedekind domain $R$ that has characteristic $0$ and field of fractions $F$ and an $R$-order $\mathfrak{A}$ in a finite dimensional separable $F$-algebra $A$. We also fix an extension field $E$ of $F$ and set $A_E := E\otimes_F A$.

Taking advantage of the observations made in \S\ref{control section}, in the sequel we shall assume to be given a strictly admissible complex of $\mathfrak{A}$-modules $C$ and an isomorphism of $A_E$-modules
\[ \lambda: E\otimes_RH^1(C) \xrightarrow{\sim} E\otimes_RH^2(C).\]

 \subsection{Definitions}\label{hse def sect} The given isomorphism $\lambda$ induces a composite isomorphism of $A_E$-modules
\begin{align}\label{passage to chomo iso} \vartheta_{\lambda}: {\rm Det}_{A_E}(E\otimes_R C) \cong &{\rm Det}_{A_E}(E\otimes_R H^1(C))^{-1} \otimes {\rm Det}_{A_E}(E\otimes_R H^2(C))\\
\cong &{\rm Det}_{A_E}(E\otimes_R H^2(C))^{-1} \otimes {\rm Det}_{A_E}(E\otimes_R H^2(C))\notag\\
\cong & (A_E,0),\notag\end{align}
where the first isomorphism is the canonical `passage to cohomology' isomorphism, the second is ${\rm Det}_{A_E}(\lambda)^{-1}\otimes 1$ and the third is induced by the standard evaluation pairing on ${\rm Det}_{A_E}(E\otimes_R H^2(C))$.

\begin{definition}\label{char elt def} {\em An element $\mathcal{L} =\mathcal{L}_{(C,\lambda)}$ of $A_E^\times$ will be said to be a `characteristic element' for the pair $(C, \lambda)$ if it satisfies
\begin{equation}\label{etnc equality} \vartheta_\lambda({\rm Det}_\mathfrak{A}(C)) = (\mathfrak {A}\cdot \mathcal{L},0).\end{equation}}
\end{definition}

\begin{remark}{\em \
There exists an element $\mathcal{L}$ of $A_E^\times$ with the property (\ref{etnc equality}) if and only if the Euler characteristic of $C$ in $K_0(\mathfrak{A})$ vanishes. The condition (${\rm ad}_2$) implies that this condition is automatically satisfied if $K_0(\mathfrak{A})$ is torsion-free as is the case, for example, if $R$ is local.


}\end{remark}

For each non-negative integer $a$ we also define idempotents of $A$ by setting
\[ e_a=e_{C,a} := \sum_e e \,\, \text{  and  }\,\, e_{(a)}=e_{C,(a)} := \sum _{a'\ge a}e_{C,a'}\]
where the first sum runs over all primitive idempotents $e$ of $A$ with the property that the (free) $Ae$-module $e(F\cdot H^2(C))$ has rank $a$.

\begin{definition}\label{hse def}{\em Let $\mathcal{L}$ be a characteristic element for the pair $(C,\lambda)$. Then, for any ordered subset $\mathcal{X}$ of $H^2(C)_{\rm tf}$ of cardinality $a$, the {\em higher special element} associated to the data $(C,\lambda,\mathcal{L},\mathcal{X})$ is the unique element $\eta = \eta_{(C,\lambda,\mathcal{L},\mathcal{X})}$ of ${\bigwedge}_{A_E}^a(E\otimes_RH^1(C))$ that satisfies}
\begin{equation}\label{eta def} ({\bigwedge}_{A_E}^a\lambda)(\eta) = e_{C,a}\cdot\mathcal{L}^{-1}\cdot {\wedge}_{x \in \mathcal{X}}x.\end{equation}
\end{definition}

\begin{remark}{\em \
 The use of $\mathcal{L}^{-1}$ (rather than $\mathcal{L}$) in the definition of $\eta$ is forced by the equality (\ref{etnc equality}) and the fact that we have assumed $C$, and hence $E\otimes_R C$, to be acyclic outside degrees one and two (rather than zero and one).}\end{remark}

\begin{lemma}\label{rationality} For any data $(C,\lambda,\mathcal{L},\mathcal{X})$ as in Definition \ref{hse def} the element $\eta_{(C,\lambda,\mathcal{L},\mathcal{X})}$ belongs to $e_{C,|\mathcal{X}|}(F\cdot{\bigwedge}_{\mathfrak{A}}^{|\mathcal{X}|} H^1(C))$.
\end{lemma}

\begin{proof} Setting $\eta:= \eta_{(C,\lambda,\mathcal{L},\mathcal{X})}$, $a := |\mathcal{X}|$ and $e_a := e_{C,a}$, it is enough to prove for each primitive idempotent $e$ of $A$ that the element $e(\eta)$ belongs to $ee_{a}({\bigwedge}_{A}^a H^1(C^\bullet))$.

If $ee_a =0$, then this containment is clear since the defining equality (\ref{eta def}) implies that $e(\eta) = 0$.

If $ee_a \not= 0$, then the rank over $Ae$ of $e(F\cdot H^2(C))$, and hence also of $e(F\cdot H^1(C))$,  is $a$ and so
\[ e\cdot {\rm Det}_{A_E}(E\otimes_ R H^i(C))  = ({\bigwedge}^a_{A_E}e(E\otimes_R H^i(C)),a)\]
for both $i=1$ and $i=2$. Given this, our choice of $\mathcal{L}$ implies that
\[ ({\bigwedge}_{A_E}^a\lambda)(e_a(F\cdot {\bigwedge}^a_{\mathfrak{A}}H^1(C))) = e_a\cdot \mathcal{L}^{-1}\cdot  (F\cdot {\bigwedge}^a_{\mathfrak{A}} H^2(C))\]
and so the required containment is true since ${\wedge}_{x \in \mathcal{X}}x$ belongs to ${\bigwedge}^a_{\mathfrak{A}} H^2(C)$.
\end{proof}


We will find that the theory of special elements has an especially rich structure when one considers subsets $\mathcal{X}$ of $H^2(C)$ with the following property.

\begin{definition}\label{sep def}{\em  A finite subset $\mathcal{X}$ of an $\mathfrak{A}$-module $X$ will be said to be {\em separable} if the $\mathfrak{A}$-module $\langle\mathcal{X}\rangle$ that it generates is both free of rank $|\mathcal{X}|$ and a direct summand of $X$.}\end{definition}

\begin{remark}{\em If $\langle\mathcal{X}\rangle$ is a projective $\mathfrak{A}$-module that is a direct summand of $H^2(C)$, then for each prime ideal $\wrp$ of $R$ and each local factor $\mathfrak{A}'_\wrp$ of $\mathfrak{A}_\wrp$ the image of $\mathcal{X}$ in the $\mathfrak{A}'_\wrp$-module generated by $H^2(C)$ is separable. In this way, the term `free' in the definition of separable can be replaced by `projective'.} \end{remark}

\begin{remark}\label{sep=sat}\em If $\langle \mathcal{X}\rangle$ is a free $\mathfrak{A}$-module of rank $|\mathcal{X}|$, then $\mathcal{X}$ is separable if and only if the tautological short exact sequence $0 \to \langle \mathcal{X}\rangle \to X \to X/\langle \mathcal{X}\rangle \to 0$ splits. If $\mathfrak{A}$ is Gorenstein, then the latter condition is automatically satisfied if $\langle \mathcal{X}\rangle$ is $R$-saturated in $X$ since then the group ${\rm Ext}^1_\mathfrak{A}(X/\langle\mathcal{X}\rangle,\mathfrak{A})$ vanishes (cf. the proof of Lemma \ref{dual preserve}). \end{remark}

\begin{example}{\em Assume that $R$ is local, with maximal ideal $\mathfrak{m}$, and that $\mathfrak{A} = R[G]$ for a finite abelian group $G$. Write $G = P\times H$ with $P$ the Sylow $p$-subgroup. Then it can be shown that a finite subset $\mathcal{X}$ of an $\mathfrak{A}$-lattice $X$ is separable if and only if the images in $H^0(P,X)/\mathfrak{m}$ of the elements $h\cdot\sum_{g \in P}g(x)$ for $h \in H$ and $x \in X$ are linearly independent over $R/\mathfrak{m}$. In particular, if $G=P$, then a singleton subset $\{x\}$ of an $R[G]$-lattice $X$ is separable if and only if the element $\sum_{g\in G}g(x)$ is not contained in $\mathfrak{m}\cdot H^0(G,X)$.}
\end{example}

\subsection{Higher Fitting ideals and annihilators}\label{separable} In this section we fix data $(C,\lambda,\mathcal{L},\mathcal{X})$ as in Definition \ref{hse def}. For a finitely generated $\mathfrak{A}$-module $X$, we set $X^*:=\Hom_\mathfrak{A}(X,\mathfrak{A})$.

To study integrality properties of the special element $\eta = \eta_{(C,\lambda,\mathcal{L},\mathcal{X})}$ we note at the outset that Lemma \ref{rationality} implies the set
\[ I(\eta) := \{ (\wedge_{i=1}^{i=|\mathcal{X}|}\varphi_i)(\eta)\! :\! \varphi_i\in H^1(C)^* \,\text{ for }\, 1\le i\le |\mathcal{X}|\}\]
is an $\mathfrak{A}$-submodule of $A$.

\subsubsection{}In this section we prove the following result.

\begin{theorem}\label{char els} Fix data $(C,\lambda,\mathcal{L},\mathcal{X})$ as in Definition \ref{hse def} and set $\eta := \eta_{(C,\lambda,\mathcal{L},\mathcal{X})}$, $a := |\mathcal{X}|$ and $\mathfrak{A}' := \mathfrak{A}e_{C,(a)}$.

Then for any element $y$ of the ideal ${\rm Ann}_\mathfrak{A}({\rm Ext}^2_{\mathfrak{A}}(H^2(C),\mathfrak{A}))\cdot {\rm Ann}_{\mathfrak{A}}
({\rm Ext}^2_{\mathfrak{A}'}(\mathfrak{A}'\otimes_{\mathfrak{A}}H^2(C),\mathfrak{A}'))$ the following claims are valid.

\begin{itemize}
\item[(i)] If $H^2(C)'$ is any subquotient of the $\mathfrak{A}$-module $H^2(C)$ that has rank at least $a$ at each simple component of $Ae_{C,(a)}$, then for any element $x$ of $\mathfrak{A} \cap \mathfrak{A}'$ one has
\[ xy^a \cdot I(\eta) \subseteq {\rm Fit}^a_\mathfrak{A}(H^2(C)) \cap {\rm Ann}_\mathfrak{A}(H^2(C)'_{\rm tor}).\]
\item[(ii)] If $\mathcal{X}$ is separable, then $e_{C,(a)} = 1$ and both
\[y^a\cdot I(\eta) \subseteq {\rm Fit}^a_\mathfrak{A}(H^2(C))\mbox{ and }{\rm Fit}^a_\mathfrak{A}(H^2(C)) \subseteq I(\eta).\]
\end{itemize}
\end{theorem}

\begin{remark}\label{get rid of x and y}{\em In certain cases, the elements $x$ and $y$ can be either omitted from, or made explicit in, the statement of Theorem \ref{char els}. For example, if $e = e_{C,(a)}\in \mathfrak{A}$, then the element $x$ can be taken to be $e$ and so (since $\eta = e\cdot \eta)$ can be omitted from the statement. Regarding the choice of $y$ note, for example, that if $\mathfrak{B}$ is an $R$-order such that ${\rm Ext}^1_{\mathfrak{B}}(N,\mathfrak{B})$ vanishes for all finitely generated torsion-free $\mathfrak{B}$-modules $N$ (see Example \ref{dual exams}(ii)), then ${\rm Ext}^2_{\mathfrak{B}}(X',\mathfrak{B})$ vanishes for all finitely generated $\mathfrak{B}$-modules $X'$.}
\end{remark}

\begin{remark}\label{explicit verson of char els} {\em Under the conditions discussed in Remark \ref{get rid of x and y} (for example when $\mathcal{X}$ is separable and $\mathfrak{A}$ is Gorenstein), Theorem \ref{char els} implies that $I(\eta_{(C,\lambda,\mathcal{L},\mathcal{X})}) = {\rm Fit}^a_\mathfrak{A}(H^2(C))$ with $a = |\mathcal{X}|$ and so the special element completely determines the higher Fitting ideal. In this case Theorem \ref{char els} is a direct generalization of \cite[Th. 7.5]{bks1}.}
\end{remark}

\begin{remark}{\em The occurrence of the subquotient $H^2(C)'$ in Theorem \ref{char els}(i) is motivated, in part, by the constructions in Propositions \ref{admiss constructions}(iii), \ref{admiss constructons3}(iii) and \ref{reduction}(ii).}\end{remark}

\begin{remark}\label{explicit cong}{\em If $\mathfrak{A} = R[\Gamma]$ for a finite abelian group $\Gamma$, then the containments in Theorem \ref{char els} imply families of congruence relations between the different components of elements of the form $(\wedge_{i=1}^{i=a}\varphi_i)(\eta)$. To see this, we fix an algebraic closure $F^c$ of $F$ and for each $\rho$ in $\Gamma^* := \Hom(\Gamma,F^{c,\times})$ define an idempotent $e_\rho := |\Gamma|^{-1}\sum_{\gamma\in \Gamma}\rho(\gamma^{-1})\gamma$ in $F^c[\Gamma]$. For each $x$ in $F^c[\Gamma]$ we then define elements $x_\rho$ in $F^c$ via the equality
\[ x = \sum_{\rho \in \Gamma^*}x_\rho e_\rho = |\Gamma|^{-1}\sum_{\gamma \in \Gamma}\gamma\sum_{\rho \in \Gamma^*}\rho(\gamma^{-1})x_\rho.\]
Then, writing $F_\rho$ for the field generated over $F$ by $\{\rho(\gamma): \gamma \in \Gamma\}$ and $\mathcal{O}_\rho$ for its valuation ring, an element $x$ belongs to $R[\Gamma]$ if and only if one has $x_\rho\in \mathcal{O}_\rho$ for all $\rho$ in $\Gamma^*$, $\omega(x_\rho) = x_{\omega\circ\rho}$ for all $\rho \in \Gamma^*$ and $\omega\in G_{F_\rho/F}$ and $\sum_{\rho \in \Gamma^*}\rho(\gamma^{-1})x_\rho\equiv 0\,\, ({\rm modulo} \,\,|\Gamma|\cdot R)$ for all $\gamma \in \Gamma$.}\end{remark}


The proof of Theorem \ref{char els} will occupy the rest of \S\ref{separable}.

\subsubsection{}We first prove an important reduction result.

\begin{proposition}\label{first reduction} It is enough to prove Theorem \ref{char els} (i) in the case that $H^2(C)'$ is a quotient of $H^2(C)$ and the image of $\mathcal{X}$ in $e_{(a)}(F\cdot H^2(C)')$ generates a free $\mathfrak{A}e_{(a)}$-module of rank $a$.
 \end{proposition}

\begin{proof} The first assertion is clear since if $H^2(C)'$ is a submodule of a quotient $Q$ of $H^2(C)$, then $Q$ must have rank at least $a$ at each simple component of $A{e_{(a)}}$ and $H^2(C)'_{\rm tor}$ is a submodule of $Q_{\rm tor}$ so that ${\rm Ann}_\mathfrak{A}(Q_{\rm tor})\subseteq {\rm Ann}_\mathfrak{A}(H^2(C)'_{\rm tor})$.

To prove the second assertion we assume that $Q := H^2(C)'$ is a quotient of $H^2(C)$, label the elements of $\mathcal{X}$ as $\{x_i\}_{1\le i\le a}$ and for each index $i$ write $\overline{x_i}$ for the image of $x_i$ in $Q$.

We can also fix a subset $\mathcal{Y} := \{y_j\}_{1\le j\le a}$ of $Q$ that spans a free $A$-module of rank $a$. Then Lemma \ref{tech1} below implies that for any large enough integer $N$ the set $\mathcal{Y}_N:= \{\overline{x_j}+ p^N\cdot y_j\}_{1\le j\le a}$ spans a free $Ae_{(a)}$-module of rank $a$.

We fix such an $N$ and write $\eta_N$ for the unique element that satisfies
\[ ({\bigwedge}_{A_E}^a\lambda)(\eta_N) = e_a\cdot\mathcal{L}^{-1}\cdot {\wedge}_{y \in \mathcal{Y}_N}y.\]
By explicitly comparing the terms ${\wedge}_{y \in \mathcal{Y}_N}y$ and ${\wedge}_{x\in \mathcal{X}}x$ one finds that
\[ \eta_N \equiv \eta \,\,{\rm modulo} \,\, p^N({\bigwedge}_{A_E}^a\lambda)^{-1}(\mathcal{L}^{-1}\cdot \Xi)\]
with $\Xi$ the $\mathfrak{A}$-submodule of $Q$ generated by $\mathcal{X}\cup \mathcal{Y}$, and so it enough for us to prove that the $\mathfrak{A}$-lattice
\[ I(\Xi) := \{ (\wedge_{i=1}^{i=a}\varphi_i)({\bigwedge}_{A_E}^a\lambda)^{-1}(\mathcal{L}^{-1}\cdot \xi)\! : \xi \in \Xi\,\,\text{and}\,\, (\varphi_i)_{i}\in \Hom_\mathfrak{A}(H^1(C),\mathfrak{A})^{a}\}\]
is such that for any large enough choice of $N$ one has
\begin{equation}\label{needed} p^N\cdot I(\Xi) \subseteq {\rm Fit}^a_\mathfrak{A}(H^2(C)) \cap {\rm Ann}_\mathfrak{A}(Q_{\rm tor}).\end{equation}

To show this we note that the definition of $e_a$ combines with our choice of $\mathcal{Y}$ to imply that the natural projection and inclusion maps
\[ e_a(F\cdot ({\bigwedge}_{\mathfrak{A}}^aH^2(C))) \to e_a(F\cdot ({\bigwedge}_{\mathfrak{A}}^a Q))\,\,\,\text{ and }\,\,\, e_a(F\cdot ({\bigwedge}_{\mathfrak{A}}^a \Xi)) \to e_a(F\cdot ({\bigwedge}_{\mathfrak{A}}^a Q)\]
are both bijective. This implies that $({\bigwedge}_{A_E}^a\lambda)^{-1}(\mathcal{L}^{-1}\cdot \Xi)$ is a full $\mathfrak{A}$-sublattice of the space $e_a(F\cdot ({\bigwedge}_{\mathfrak{A}}^aH^1(C)))$ and hence that $I(\Xi)$ is a full sublattice of $e_aA$.

Since ${\rm Ann}_\mathfrak{A}(Q_{\rm tor})$ has finite index in $\mathfrak{A}$, to prove that (\ref{needed}) is valid for any large enough $N$, it is thus enough to note that ${\rm Fit}^a_\mathfrak{A}(H^2(C))\cap e_aA$ has finite index in $e_a\mathfrak{A}$. This is true because $e_a(F\cdot H^2(C))$ is a free $A$-module of rank $a$ and hence that $e_a(F\cdot  {\rm Fit}^a_\mathfrak{A}(H^2(C))) = {\rm Fit}^a_{Ae_a}(e_a(F\cdot H^2(C))) = Ae_a$. \end{proof}

In the following result we use the set $\mathcal{Y}_N$ defined above.

\begin{lemma}\label{tech1} For any large enough natural number $N$ the $Ae_{(a)}$-module $Y_N$ spanned by $\mathcal{Y}_N$ is free of rank $a$.\end{lemma}

\begin{proof} The algebra $Ae_{(a)}$ decomposes as a finite product of fields (of characteristic zero) and, after replacing $Ae_{(a)}$ by any such field $E$, it is enough to prove that for any large enough natural number $N$ the $E$-module $E\otimes_{\mathfrak{A}}Y_N$ is free of rank $a$.

We write $\pi$ for the natural projection $Q \to E\otimes_{\mathfrak{A}}Y_N$. Then the set $\pi(\mathcal{Y})$ gives an $E$-basis of the space $E\otimes_\mathfrak{A}Q = E\otimes_\mathfrak{A}H^2(C)$ and we write $M_N$ and $M$ for the transition matrices from $\pi(\mathcal{Y})$ to the sets $\pi(\mathcal{Y}_N)$ and $\pi(\mathcal{X})$. It is then enough to prove that for any large enough $N$ the matrix $M_N$ is invertible.

However, since $M_N = M + p^N\cdot{\rm Id}_a$, this is true since for any large enough choice of $N$ the integer $-p^N$ cannot be an eigenvalue of $M$ (over any algebraic closure of $E$). \end{proof}

\subsubsection{}It is enough to prove Theorem \ref{char els} (i) after replacing $\mathfrak{A}$ by its localisation $\mathfrak{A}_\wrp$ at each prime ideal $\wrp$ of $R$ and then the semi-local ring $\mathfrak{A}_\wrp$ by each of its local components. In the sequel we shall therefore assume that $\mathfrak{A}$ is local and hence that every finitely generated projective $\mathfrak{A}$-module is free.

In addition, we always assume $H^2(C)'$ and $\mathcal{X}$ are chosen as in Proposition \ref{first reduction}. We also write $e_a$ and $e_{(a)}$ in place of $e_{C,a}$ and $e_{C,(a)}$ and then set $\mathfrak{A}_a := \mathfrak{A}e_a$, $\mathfrak{A}_{(a)} := \mathfrak{A}e_{(a)}$, $A_a := Ae_a$ and $A_{(a)} := Ae_{(a)}$.

We consider the object $C_{(a)} := \mathfrak{A}_{(a)}\otimes^\mathbb{L}_\mathfrak{A} C$ of $D^{\rm p}(\mathfrak{A}_{(a)})$.

We note that the definition of $e_{(a)}$ ensures $F\cdot H^2(C_{(a)})$ contains a free $A_{(a)}$-module of rank $a$ and then Remark \ref{semisimplicity} implies the same is true of the $A_{(a)}$-module $F\cdot H^1(C_{(a)})$. We can thus fix a subset $\mathcal{X}' = \{x'_j\}_{1\le j \le a}$ of $H^1(C_{(a)})$ that is linearly independent over $A_{(a)}$.

In particular, since the definition of $e_a$ ensures that the sets $e_a\mathcal{X}'$ and $e_a\mathcal{X}$ are respectively bases of the $A_{a,E}$-modules $e_a(E\cdot H^1(C_{(a)}))$ and $e_a(E\cdot H^2(C_{(a)}))$ we can define $M(\lambda)$ to be the matrix in ${\rm GL}_a(A_{a,E})$ that represents $e_a\lambda$ with respect to these bases.

The next key step is to prove the following result concerning this matrix. In this result we write $X$ and $X'$ for the $\mathfrak{A}_{(a)}$-modules that are generated by $\mathcal{X}$ and $\mathcal{X}'$ respectively.

\begin{lemma}\label{inclusion lemma} Set $\mathfrak{A}' := \mathfrak{A}_{(a)}$ and $C' := C_{(a)}$. Then there exists a natural homomorphism of $\mathfrak{A}'$-modules
\[ \kappa: {\rm Ext}^1_{\mathfrak{A}'}\!(H^1(C'),\!\mathfrak{A}') \to {\rm Ext}^3_{\mathfrak{A}'}\!(H^2(C'),\mathfrak{A}')\]
for which one has
\[ {\rm Ann}_{\mathfrak{A}'}({\rm Ext}^2_{\mathfrak{A}'}(H^2(C'),\mathfrak{A}'))^a\cdot {\rm Fit}^0_{\mathfrak{A}'}(\ker(\kappa))\cdot ({\rm det}(M(\lambda))\mathcal{L})^{-1} \subseteq {\rm Fit}^0_{\mathfrak{A}'}(H^2(C')/X).\]
\end{lemma}

\begin{proof} Since the $\mathfrak{A}'$-modules $X$ and $X'$ that are generated by $\mathcal{X}$ and $\mathcal{X}'$ are both free of rank $a$, the inclusions $\iota^2:X\subseteq H^2(C')$ and $\iota^1:X'\subseteq H^1(C')$ give rise to an exact triangle in $D^{\rm p}(\mathfrak{A}')$ of the form
\begin{equation}\label{key triangle} X[-2] \oplus X'[-1] \xrightarrow{\iota} C' \to D \to X[-1] \oplus X'[0]\end{equation}
in which $H^1(\iota) = \iota^1$, $H^2(\iota) = \iota^2$ and the complex $D$ is acyclic outside degrees one and two and has cohomology groups in these degrees that respectively identify with the quotients $H^1(C')/X'$ and $H^2(C')/X$.


Since $\lambda$  induces an isomorphism of $A_E$-modules between $e_a(E\cdot X') = e_a(E\cdot H^1(C))$ and $e_a(E\cdot H^2(C)) = e_a(E\cdot X)$ we can fix a commutative diagram of $A_E$-modules
\begin{equation}\label{key diagram} \begin{CD}
0 @> >> E\cdot X' @> \subseteq >> E\cdot H^1(C') @> >> E\cdot H^1(D) @> >> 0\\
@. @V \lambda_1 VV @V \lambda_2 VV @V \lambda_3 VV\\
0 @> >> E\cdot X @> \subseteq >> E\cdot H^2(C') @> >> E\cdot H^2(D) @> >> 0\end{CD}\end{equation}
where the maps $\lambda_1$, $\lambda_2$ and $\lambda_3$ are bijective and satisfy
\begin{equation}\label{restrictions agree} e_a\lambda_1 = e_a\lambda_2 = e_a \lambda.\end{equation}

This commutative diagram combines with the triangle (\ref{key triangle}) to give an equality of lattices
\[ \vartheta_{\lambda_3}({\rm Det}_{\mathfrak{A}'}(D)) = \vartheta_{\lambda_1}({\rm Det}_{\mathfrak{A}'}(X[-2]\oplus X'[-1]))^{-1}\cdot \vartheta_{\lambda_2}({\rm Det}_{\mathfrak{A}'}(C'))\]

In particular, upon multiplying this equality by $e_a$, and taking account of both of the equalities (\ref{etnc equality}) and (\ref{restrictions agree}), one deduces that
\begin{align}\label{useful step eq} e_a\vartheta_{\lambda_3}({\rm Det}_{\mathfrak{A}'}(D)) = &e_a\vartheta_{\lambda_1}({\rm Det}_{\mathfrak{A}'}(X[-2]\oplus X'[-1]))^{-1}\cdot e_a\vartheta_{\lambda_2}({\rm Det}_{\mathfrak{A}'}(C'))\\
= &{\rm det}(M(\lambda))\cdot \vartheta_{e_{a}\lambda_2}(e_a\cdot {\rm Det}_{\mathfrak{A}}(C))\notag\\
= &{\rm det}(M(\lambda))\cdot e_a\vartheta_{\lambda}({\rm Det}_{\mathfrak{A}}(C))\notag\\
= & \mathfrak{A} \cdot {\rm det}(M(\lambda))\mathcal{L}\notag\end{align}
where $M(\lambda)$ is the matrix defined above.

To investigate the first expression in (\ref{useful step eq}) we note that Remark \ref{representative} allows us to fix a representative of $C$ of the form
\begin{equation}\label{rep2} P^1\to P^2,\end{equation}
where $P^1$ and $P^2$ are finitely generated free $\mathfrak{A}$-modules, $P^1$ is placed in degree one and the differential is $\delta$. In this case the complex $C'$ is represented by the complex $e_{(a)} P^1\to e_{(a)} P^2$ with differential $e_{(a)}\delta$ and so we can take $D$ to be a complex
\begin{equation}\label{D rep} X' \xrightarrow{\iota '\oplus\, 0} e_{(a)}P^1\oplus X \xrightarrow{(e_{(a)}\delta,\,\iota)} e_{(a)}P^2\end{equation}
with $X'$ placed in degree zero.

After applying the result of Proposition \ref{almost there}(ii) below to this complex we obtain the claimed inclusion as a direct consequence of (\ref{useful step eq}). \end{proof}

\begin{proposition}\label{almost there} Let $\mathfrak{B}$ be a commutative local $R$-algebra that spans a finite dimensional separable $F$-algebra $B$. Let $D$ be a complex of finitely generated free $\mathfrak{B}$-modules of the form
\[ D^0\xrightarrow{\delta^0} D^1 \xrightarrow{\delta^1} D^2\]
where $D^0$ is placed in degree zero. Assume $D$ is acyclic outside degrees one and two and that there is an isomorphism of $B_E$-modules of the form $\mu: E\cdot H^1(D) \cong E\cdot H^2(D)$.

Then the following claims are valid.

\begin{itemize}
\item[(i)] There exists a natural homomorphism of abelian groups of the form
\[ \kappa: {\rm Ext}^1_{\mathfrak{B}}\!(H^1(D),\!\mathfrak{B}) \to {\rm Ext}^3_{\mathfrak{B}}\!(H^2(D),\mathfrak{B}).\]

\item[(ii)] With $e$ denoting the sum of primitive idempotents of $B$ that annihilate $F\cdot H^2(D)$, one has
\[ {\rm Ann}_{\mathfrak{B}}({\rm Ext}^2_{\mathfrak{B}}(H^2(D),\mathfrak{B}))^{{\rm rk}(D^0)}\cdot {\rm Fit}^0_{\mathfrak{B}}(\ker(\kappa))\cdot e \cdot \vartheta_{\mu}({\rm Det}_{\mathfrak{B}}(D))^{-1} \subseteq {\rm Fit}^0_{\mathfrak{B}}(H^2(D)),\]
where ${\rm rk}(D^0)$ denotes the rank of the $\mathfrak{B}$-module $D^0$.

\item[(iii)] Assume that the canonical homomorphism $R \to \mathfrak{B}$ satisfies the conditions of Lemma \ref{dual preserve} and that the modules $H^1(D)$ and $H^2(D)$ are finite. Then $e=1$ and one has
\[{\rm Fit}^0_{\mathfrak{B}}( \Hom_R(H^1(D),F/R))\cdot \vartheta_{0}({\rm Det}_{\mathfrak{B}}(D))^{-1} = {\rm Fit}^0_{\mathfrak{B}}(H^2(D)),\]
where the Pontryagin dual is endowed with the natural action of $\mathfrak{B}$ and $0$ denotes the unique automorphism of the zero space.

\end{itemize}
\end{proposition}

\begin{remark}
{\em
The proof of Proposition \ref{almost there} given below shows that the homomorphism $\kappa$ in claim (i) is constructed as the composition
$$\Ext^1_\mathfrak{B}(H^1(D), \mathfrak{B})  \to \Ext^1_\mathfrak{B}(Z^1(D), \mathfrak{B}) \to \Ext^3_\mathfrak{B}(H^2(D),\mathfrak{B}),$$
where the first map is induced by the natural map $Z^1(D)=\ker (\delta^1) \to H^1(D)$ and the second by the `Yoneda product' with the Yoneda extension class
$$0 \to Z^1(D) \to D^1 \to D^2 \to H^2(D)\to 0$$
in $\Ext^2_\mathfrak{B}(H^2(D),Z^1(D))$.
}
\end{remark}

\begin{proof}[Proof of Proposition \ref{almost there}] The differential $\delta^0$ is injective and, since the groups $e(F\cdot H^2(D))$ and hence $e(F\cdot H^1(D))$ vanish, there exists a direct sum decomposition $e(F\cdot D^1) = V^1_1 \oplus V^1_2$ so that the maps $e(F\otimes_R\delta^0)$ and $e(E\otimes_R\delta^1)$ give isomorphisms $e(F\cdot D^0) \cong V^1_1$ and $V^1_2 \cong e(F\cdot D^2)$ respectively. We can therefore fix an isomorphism of $B_E$-modules
\begin{equation}\label{helpful iso} \iota: E\cdot (D^0\oplus D^2) \to E\cdot D^1\end{equation}
which restricts to give the scalar extension of the isomorphism $e(F\cdot D^0) \oplus e(F\cdot D^2) \cong e(F\cdot D^1)= V_1^1 \oplus V^1_2$ given by $(e(F\otimes_R\delta^0),e(F\otimes_R\delta^1)^{-1})$.

For any such isomorphism one has an equality
\begin{equation}\label{initial step} e \cdot \vartheta_{\mu}({\rm Det}_{\mathfrak{B}}(D)) = \mathfrak{B}e\cdot {\rm det} (M(\iota))\end{equation}
where $M(\iota)$ is the matrix of $\iota$ with respect to any choice of $\mathfrak{B}$-bases of $D^0\oplus D^2$ and $D^1$.

To compute the term ${\rm det}(M(\iota))$ more explicitly we apply the functor $\Hom_{\mathfrak{B}}(-,\mathfrak{B})$ to the tautological exact sequences
\[\begin{cases}0 \to D^0 \xrightarrow{\delta^0} Z^1(D) \to H^1(D)\to 0\\
 0 \to Z^1(D) \to D^1 \to B^2(D) \to 0\\
 0 \to B^2(D) \to D^2 \to H^2(D) \to 0\end{cases}\]

In particular, since the groups ${\rm Ext}^j_{\mathfrak{B}}(D^i,\mathfrak{B})$ vanish for each $j \ge 1$ and each $i\in \{0,1,2\}$, we obtain in this way exact sequences
\begin{equation}\label{useful}\Hom_{\mathfrak{B}}\!(Z^1\!(D),\!\mathfrak{B})\!\xrightarrow{j^0} \Hom_{\mathfrak{B}}\!(D^0,\!\mathfrak{B})\! \xrightarrow{j^1} {\rm Ext}^1_{\mathfrak{B}}\!(H^1(D),\!\mathfrak{B}) \!\xrightarrow{j^2} {\rm Ext}^1_{\mathfrak{B}}\!(Z^1(D),\mathfrak{B})\to 0,\end{equation}
%
%
and
\[
\Hom_{\mathfrak{B}}\!(D^1,\mathfrak{B}) \xrightarrow{k^0} \Hom_{\mathfrak{B}}\!(Z^1\!(D),\mathfrak{B}) \xrightarrow{k^1} {\rm Ext}^1_{\mathfrak{B}}\!(B^2\!(D),\mathfrak{B}) \to 0\]%
and isomorphisms ${\rm Ext}^1_{\mathfrak{B}}\!(Z^1(D),\mathfrak{B}) \cong {\rm Ext}^2_{\mathfrak{B}}\!(B^2(D),\mathfrak{B}) \cong {\rm Ext}^3_{\mathfrak{B}}\!(H^2(D),\mathfrak{B}).$

Taking the composite of the latter isomorphism with the map $j^2$ in (\ref{useful}) we obtain a map $\kappa$ as in claim (i).

Turning to claim (ii) we set $t:= {\rm rk}(D^0)$ and choose a $\mathfrak{B}$-basis $\{y_i\}_{1\le i\le t}$ of $D^0$. Then the $\mathfrak{B}$-module $\Hom_{\mathfrak{B}}(D^0,\!\mathfrak{B})$ is free with basis $\{y_j^*\}_{1\le j\le t}$ where each $y_j^*$ is the dual of $y_j$.

For each integer $j$ with $1\le j\le t$ we choose an element
\begin{equation}\label{choose element} b_j := \sum_{i=1}^{i=t}c_{ij}y_i^*\end{equation}
in $\ker(j^1)$ and an element $\phi_j$ of $\Hom_{\mathfrak{B}}\!(Z^1\!(D),\!\mathfrak{B}\!)$ with $j^0(\phi_j) = b_j$. Then for any element $z$ of the group
\[ {\rm Ann}_{\mathfrak{B}}({\rm Ext}^1_{\mathfrak{B}}\!(B^2(D),\mathfrak{B})) = {\rm Ann}_{\mathfrak{B}}({\rm Ext}^2_{\mathfrak{B}}\!(H^2(D),\mathfrak{B}))\]
there exists a homomorphism $\varphi_j$ in $\Hom_{\mathfrak{B}}\!(D^1,\mathfrak{B})$ with
 $k^0(\varphi_j) = z\cdot \phi_j$.

We finally define $\phi$ to be the element of $\Hom_{\mathfrak{B}}(D^1,D^0)$ that sends each element $w$ of $D^1$ to $\sum_{i=1}^{i=t}\varphi_i(w)\cdot y_i$ and consider the homomorphism $D^1 \to D^0\oplus D^2$ that is given by the direct sum $\phi \oplus \delta^1$.

Now, by explicitly comparing this map to the isomorphism $\iota$ defined in (\ref{helpful iso}) one computes that on $e(E\cdot D^0)\oplus e(E\cdot D^2)$ there is an equality of functions
\[ e(E\otimes_R(\phi\oplus \delta^1))\circ e(\iota) = (e(E\otimes_R(\phi\circ \delta^0)),{\rm id}_{e(E\cdot D^2)}).\]
and for each basis element $y_i$ one has
\begin{multline*} (\phi\circ \delta^0)(y_i) = \sum^{j=t}_{j=1} z(\phi_j\circ\delta^0)(y_i)\cdot y_j = \sum^{j=t}_{j=1} z(b_j)(y_i)\cdot y_j \\ = \sum^{j=t}_{j=1}z (\sum_{a=1}^{a=t}c_{aj}y_a^*(y_i))\cdot y_j
 = \sum^{j=t}_{j=1} zc_{ij}\cdot y_j.
 \end{multline*}

Hence, if we write $M(\phi\oplus \delta^1)$ for the matrix of $\phi\oplus \delta^1$ with respect to any choice of $\mathfrak{B}$-bases of $D^1$ and $D^2$ and the fixed basis $\{y_i\}_{1\le i\le t}$ of $D^0$, then one has
\[  {\rm det}(e\cdot M(\phi\oplus \delta^1)) = {\rm det}((zc_{ij})_{1\le i,j\le t})\cdot {\rm det}(e\cdot M(\iota))^{-1}.\]
In addition, for any primitive idempotent $e'$ of $B$, the matrix $e'\cdot M (\phi\oplus \delta^1)$ is invertible only if $e' = e' e$, and so one has ${\rm det}(M(\phi\oplus \delta^1)) = {\rm det}(e \cdot M(\phi\oplus \delta^1))$.

Putting everything together we find that (\ref{initial step}) implies that

\begin{align}\label{almost2} z^t\cdot {\rm det}((c_{ij})_{1\le i,j\le t})e \cdot \vartheta_{\mu}({\rm Det}_{\mathfrak{B}}(D))^{-1} &={\rm det}((zc_{ij})_{1\le i,j\le t})e \cdot \vartheta_{\mu}({\rm Det}_{\mathfrak{B}}(D))^{-1}\\
&= \mathfrak{B}e\cdot
{\rm det}((zc_{ij})_{1\le i,j\le t})\cdot {\rm det}(M(\iota))^{-1}\notag\\
&= \mathfrak{B}\cdot
{\rm det}((zc_{ij})_{1\le i,j\le t})\cdot {\rm det}(e\cdot M(\iota))^{-1}\notag\\
&= \mathfrak{B}\cdot {\rm det}(e\cdot M(\phi\oplus \delta^1))\notag\\
&= \mathfrak{B}\cdot {\rm det}( M(\phi\oplus \delta^1))\notag\\
&= {\rm Fit}^0_{\mathfrak{B}}( {\rm cok}(\phi\oplus\delta^1))\notag\\
&\subseteq  {\rm Fit}^0_{\mathfrak{B}}(H^2(D)),\notag\end{align}
where the last equality follows directly from the definition of zero-th Fitting invariant and the upper row in the following exact commutative diagram
\[\begin{CD}
 D^1 @> \phi\oplus\delta^1 >> D^0 \oplus D^2 @> >> {\rm cok}(\phi\oplus\delta^1) @> >> 0\\
 @\vert @V (0,{\rm id})VV @V\epsilon VV \\
 D^1 @> \delta^1 >> D^2 @> >> H^2(D
 ) @> >> 0,\end{CD}\]
and the inclusion in (\ref{almost2}) from (a standard property of Fitting ideals with respect to surjective maps and) the fact that the map $\epsilon$ in this diagram is surjective.

Now the exactness of the sequence (\ref{useful}) implies that as the elements $b_j$ in (\ref{choose element}) range over $\ker(j^1)$ the determinants of the matrices $(c_{ij})_{1\le i,j\le t}$ range over a set of generators of the ideal ${\rm Fit}_{\mathfrak{B}}^0(\ker(j^2)) = {\rm Fit}_{\mathfrak{B}}^0(\ker(\kappa))$.

The inclusion (\ref{almost2}) therefore implies that for any element $z$ of ${\rm Ann}_{\mathfrak{B}}({\rm Ext}^2_{\mathfrak{B}}(H^2(D),\mathfrak{B}))$  one has $z^t\cdot{\rm Fit}_{\mathfrak{B}}^0(\ker(\kappa))\cdot\vartheta_{\mu}({\rm Det}_{\mathfrak{B}}(D))^{-1} \subseteq {\rm Fit}^0_{\mathfrak{B}}(H^2(D))$, as required to prove claim (ii).

To prove claim (iii) we note first that, under the stated conditions, the inclusion of claim (ii) becomes
\[ {\rm Fit}^0_{\mathfrak{B}}({\rm Ext}^1_{\mathfrak{B}}(H^1(D),\mathfrak{B})) \cdot \vartheta_{0}({\rm Det}_{\mathfrak{B}}(D))^{-1} \subseteq {\rm Fit}^0_{\mathfrak{B}}(H^2(D)),\]
with $0$ denoting the unique automorphism of the zero space.

In addition, in this case the argument of Lemma \ref{dual preserve} implies $D^* := R\Hom_R(D,R[-3])$ belongs to $D^{\rm s}(\mathfrak{B})$ and, as $H^1(D)$ and $H^2(D)$ are assumed to be finite, that the group $H^i(D^*)$ for $i \in \{1,2\}$ is isomorphic to
${\rm Ext}^1_{R}(H^{3-i}(D),R)\cong \Hom_{R}(H^{3-i}(D),F/R)$, where the last isomorphism is induced by applying the functor $\Hom_R(H^{3-i}(D),-)$ to the tautological exact sequence $0 \to R \to F \to F/R \to 0$.

Hence, if we apply claim (ii) to the complex $D^*$, we obtain an inclusion
\[ {\rm Fit}^0_{\mathfrak{B}}(H^2(D)) \cdot \vartheta_{0}({\rm Det}_{\mathfrak{B}}(D^*))^{-1} \subseteq {\rm Fit}^0_{\mathfrak{B}}(\Hom_R(H^1(D),F/R)).\]

Noting that $\vartheta_{0}({\rm Det}_{\mathfrak{B}}(D^*))^{-1} = \vartheta_{0}({\rm Det}_{\mathfrak{B}}(D))$, we can therefore combine the last two displayed inclusions to obtain the equality in claim (iii).\end{proof}

\begin{remark}{\em The equality in Proposition \ref{almost there}(iii) was first proved (by using a different argument and under certain additional hypotheses on $\mathfrak{B}$) by Cornacchia and Greither in \cite[Prop. 6]{cg} and has since been extensively used in the literature to compute Fitting ideals of natural arithmetic modules. }
\end{remark}

\subsubsection{}The next result relates the element $\eta_\mathcal{X}$ in Theorem \ref{char els} to the terms that occur in Lemma \ref{inclusion lemma}.

\begin{lemma}\label{inclusion lemma2} For each subset $\{\varphi'_j\}_{1\le j\le a}$ of $\Hom_{\mathfrak{A}'}(H^1(C'),\mathfrak{A}')$ one has
\[ ({\wedge}_{i=1}^{i=a}\varphi'_i)(\eta_\mathcal{X}) \in {\rm Fit}^0_{\mathfrak{A}'}(\ker(\kappa))\cdot ({\rm det}(M(\lambda))\mathcal{L})^{-1},\]
where $\kappa$ is the homomorphism in Lemma \ref{inclusion lemma}.
\end{lemma}

\begin{proof} To do this we note first that the image of $\eta_\mathcal{X}$ under the injective map $\lambda$ is, by its very definition,  equal to
\begin{align*}e_a\cdot\mathcal{L}^{-1}\cdot {\wedge}_{i=1}^{i=a}x_i = &\mathcal{L}^{-1}\cdot {\wedge}_{i=1}^{i=a}e_a x_{i} \\
= &\mathcal{L}^{-1}\cdot {\rm det}(M(\lambda))^{-1}\cdot\lambda({\wedge}_{j=1}^{j=a}e_ax_j')\\
 = &({\rm det}(M(\lambda))\mathcal{L})^{-1}\cdot\lambda({\wedge}_{j=1}^{j=a}x_j')\\
  = &\lambda(({\rm det}(M(\lambda))\mathcal{L})^{-1}\cdot ({\wedge}_{j=1}^{j=a}x_j'))\end{align*}
and hence that
\begin{equation}\label{almost wanted} \eta_\mathcal{X} = ({\rm det}(M(\lambda))\mathcal{L})^{-1}\cdot {\wedge}_{j=1}^{j=a}x_j'.\end{equation}

We next apply the functor $\Hom_{\mathfrak{A}'}(-,\mathfrak{A}')$ to the short exact sequence
\[0 \to X'\xrightarrow{\iota'} H^1(C') \to H^1(D) \to 0\]
to obtain an exact sequence
\begin{multline}\label{usefulexact2}  \Hom_{\mathfrak{A}'}(H^1(C'),\mathfrak{A}') \xrightarrow{\ell^0} \Hom_{\mathfrak{A}'}(X',\mathfrak{A}') \xrightarrow{\ell^1} {\rm Ext}^1_{\mathfrak{A}'}(H^1(D),\mathfrak{A}')\\ \xrightarrow{\ell^2} {\rm Ext}^1_{\mathfrak{A}'}(H^1(C'),\mathfrak{A}')\to 0.\end{multline}

Now the $\mathfrak{A}'$-module $\Hom_{\mathfrak{A}'}(X',\mathfrak{A}')$ is free on the basis $\{x_j'^*\}_{1\le j\le a}$ where each $x_j'^*$ is dual to $x'_j$. In particular, for the given homomorphisms $\varphi'_j$ we can write the element $\ell^0(\varphi'_j) = \varphi'_j\circ\iota'$ of $\ker(\ell^1)$ as $\sum^{i=a}_{i=1}c'_{ij}x_i'^*$ with each $c'_{ij}$ in $\mathfrak{A}'$.

The equality (\ref{almost wanted}) therefore implies that

\begin{align}\label{almost 3}
{\rm det}(M(\lambda))\mathcal{L}\cdot({\wedge}_{i=1}^{i=a}\varphi'_i)(e_a\cdot\eta_\mathcal{X}) &=  ({\wedge}_{i=1}^{i=a}\varphi'_i)({\wedge}^{j=a}_{j=1}x_j') \\
&= ({\wedge}_{i=1}^{i=a}(\varphi'_i\circ \iota))({\wedge}_{j=1}^{j=a}x_j') \notag\\
&= {\rm det}((\varphi'_i\circ \iota)(x_j'))_{1\le i,j\le a})\notag\\
&= {\rm det}((\sum^{m=a}_{m=1}c'_{mi}x_m'^*)(x_j'))_{1\le i,j\le a})\notag\\
&= {\rm det}((c'_{ji})_{1\le i,j\le a})\notag\\
&\in {\rm Fit}^0_{\mathfrak{A}'}(\ker(\ell^2)),\notag
\end{align}
where the containment follows from the exactness of (\ref{usefulexact2}).

Finally, we note that, since $D$ is the complex (\ref{D rep}), one has $Z^1(D) = \ker(e_{(a)}\delta) = H^1(C')$ and the map $\ell^2$ coincides with the map $j^2$ in (\ref{useful}). It follows that $\ker(\ell^2) = \ker(j^2) = \ker(\kappa)$, with $\kappa$ the map in Lemma \ref{inclusion lemma}, and so the claimed result follows directly from (\ref{almost 3}). \end{proof}

In view of Lemma \ref{inclusion lemma2} it is important to explain the connection between the homomorphism groups $\Hom_{\mathfrak{A}}(H^1(C),\mathfrak{A})$ and $\Hom_{\mathfrak{A}'}(H^1(C'),\mathfrak{A}')$.

\begin{lemma}\label{missing one} For each $y$ in ${\rm Ann}_{\mathfrak{A}}({\rm Ext}^2_{\mathfrak{A}}(H^2(C),\mathfrak{A}))$ and each $\varphi$ in  $\Hom_{\mathfrak{A}}(H^1(C),\mathfrak{A})$ the product $ye_{(a)}\cdot \varphi$ belongs to $\Hom_{\mathfrak{A}'}(H^1(C'),\mathfrak{A}')$.
\end{lemma}

\begin{proof} The representative (\ref{rep2}) of $C$ gives rise to tautological exact sequences
\[ \begin{cases} 0 \to H^1(C) \to P^1 \to B^2(C) \to 0,\\
                 0 \to B^2(C) \to P^2 \to H^2(C) \to 0.\end{cases}\]

These sequences in turn give rise to an exact sequence
\[ \Hom_{\mathfrak{A}}(P^1,\mathfrak{A}) \to \Hom_{\mathfrak{A}}(H^1(C),\mathfrak{A}) \to {\rm Ext}^1_{\mathfrak{A}}(B^2(C),\mathfrak{A})\to 0 \]
and an isomorphism ${\rm Ext}^1_{\mathfrak{A}}(B^2(C),\mathfrak{A})\cong {\rm Ext}^2_{\mathfrak{A}}(H^2(C),\mathfrak{A})$.

Thus for each $y$ in ${\rm Ann}_{\mathfrak{A}}({\rm Ext}^2_{\mathfrak{A}}(H^2(C),\mathfrak{A}))$ and each $\varphi$ in  $\Hom_{\mathfrak{A}}(H^1(C),\mathfrak{A})$ there exists a homomorphism $\varphi_y$ in $\Hom_{\mathfrak{A}}(P^1,\mathfrak{A})$ which restricts to give $y\cdot \varphi$ on $H^1(C)$.

Since $H^1(C')$ is a submodule of $e_{(a)}P^1$ one therefore has
\[ (ye_{(a)}\cdot \varphi)(H^1(C')) = \varphi_y(H^1(C')) \subseteq  \varphi_y(e_{(a)}P^1) = \varphi_y(P^1)\cdot e_{(a)} \subseteq \mathfrak{A}\cdot e_{(a)} = \mathfrak{A}' ,    \]
as required.
\end{proof}

\subsubsection{}We are now ready to prove Theorem \ref{char els}(i).

As a first step we combine the results of Lemmas \ref{inclusion lemma}, \ref{inclusion lemma2} and \ref{missing one} to deduce that for each subset $\{\varphi_j\}_{1\le j\le a}$ of $\Hom_{\mathfrak{A}}(H^1(C),\mathfrak{A})$, each $y_1$ in ${\rm Ann}_{\mathfrak{A}'}({\rm Ext}^2_{\mathfrak{A}'}(H^2(C'),\mathfrak{A}'))$ and each $y_2$ in ${\rm Ann}_{\mathfrak{A}}({\rm Ext}^2_{\mathfrak{A}}(H^2(C),\mathfrak{A}))$ one has
\[ (y_1y_2)^a\cdot ({\wedge}_{i=1}^{i=a}\varphi_i)(\eta_\mathcal{X}) = (y_1)^a \cdot ({\wedge}_{i=1}^{i=a} (y_2e_{(a)}\cdot\varphi_i))(\eta_\mathcal{X}) \subseteq {\rm Fit}^0_{\mathfrak{A}'}(H^2(C')/X).\]

Set $\mathfrak{A}^\dagger := \mathfrak{A}\cap \mathfrak{A}'$. Then to deduce the result of Theorem \ref{char els}(i) from the above inclusion it suffices to prove that one has both
\begin{equation}\label{last step} \mathfrak{A}^\dagger\cdot {\rm Fit}^0_{\mathfrak{A}'}(H^2(C')/X)\subseteq {\rm Fit}^a_\mathfrak{A}(H^2(C))\,\,\text{ and }\,\, \mathfrak{A}^\dagger\cdot {\rm Fit}^0_{\mathfrak{A}'}(H^2(C')/X)\subseteq {\rm Ann}_\mathfrak{A}(H^2(C)'_{\rm tor}).\end{equation}

Since $\mathfrak{A}^\dagger \cdot \mathfrak{A}' \subseteq \mathfrak{A}$ the first of these inclusions follows directly from the fact that
\[ {\rm Fit}^0_{\mathfrak{A}'}(H^2(C')/X) \subseteq {\rm Fit}^a_{\mathfrak{A}'}(H^2(C')) = {\rm Fit}^a_{\mathfrak{A}'}(\mathfrak{A}'\otimes_{\mathfrak{A}}H^2(C)) = \mathfrak{A}' \cdot {\rm Fit}^a_{\mathfrak{A}}(H^2(C)).\]
Here the inclusion is true because $X$ is a free $\mathfrak{A}'$-module of rank $a$, the first equality because $H^2(C')$ is isomorphic to $\mathfrak{A}'\otimes_{\mathfrak{A}}H^2(C)$ and the second equality follows from a standard property of Fitting ideals under change of ring.

Next we note that the given surjective homomorphism of $\mathfrak{A}$-modules $H^2(C) \to H^2(C)'$ induces a surjective homomorphism of $\mathfrak{A}'$-modules $\varpi: H^2(C') \to \mathfrak{A}'\otimes_\mathfrak{A}H^2(C)'$ and we recall that, by assumption, the $\mathfrak{A}'$-module $\varpi(X)$ is free of rank $a$. This implies that
\begin{equation}\label{second incl} {\rm Fit}^0_{\mathfrak{A}'}(H^2(C')/X) \subseteq
{\rm Ann}_{\mathfrak{A}'}((\mathfrak{A}'\otimes_\mathfrak{A}H^2(C)')/\varpi(X)) \subseteq {\rm Ann}_{\mathfrak{A}'}((\mathfrak{A}'\otimes_\mathfrak{A}H^2(C)')_{\rm tor}).\end{equation}

We now set $\mathfrak{A}^{\sharp} := \mathfrak{A}\cap \mathfrak{A}(1-e_{(a)})$ and note that the tautological short exact sequence
$0 \to \mathfrak{A}^\sharp \to \mathfrak{A} \to \mathfrak{A}' \to 0$ gives rise to an exact sequence of $\mathfrak{A}$-modules of the form %
\[ (\mathfrak{A}^\sharp\otimes_\mathfrak{A}H^2(C)')_{\rm tor} \to H^2(C)'_{\rm tor} \to (\mathfrak{A}'\otimes_\mathfrak{A}H^2(C)')_{\rm tor}.\]
Since the first module in this sequence is clearly annihilated by $\mathfrak{A}^\dagger$ one therefore has
\[ \mathfrak{A}^\dagger\cdot {\rm Ann}_{\mathfrak{A}'}((\mathfrak{A}'\otimes_\mathfrak{A}H^2(C)')_{\rm tor}) \subseteq {\rm Ann}_{\mathfrak{A}}(H^2(C)'_{\rm tor})\]
and this combines with (\ref{second incl}) to imply the second inclusion of (\ref{last step}), as required to complete the proof of Theorem \ref{char els}(i).

\subsubsection{} As preparation for the proof of Theorem \ref{char els}(ii), in this section we show that strictly admissible complexes admit resolutions with certain useful features.

For any $\mathfrak{A}$-module $M$ we set $M^* := \Hom_\mathfrak{A}(M,\mathfrak{A})$, endowed with the natural action of $\mathfrak{A}$. We recall that $M$ is said to be `reflexive' if the natural map $M \to (M^*)^*$ is bijective.

\begin{lemma}\label{adm-rep} For any object $C$ of $D^{\rm s}(\mathfrak{A})$, the following claims are valid.
\begin{itemize}
\item[(i)] The $\mathfrak{A}$-module $H^1(C)$ is reflexive.
\item[(ii)] Assume that the Euler characteristic of $C$ in $K_0(\mathfrak{A})$ vanishes. Then for each separable subset $\mathcal{X}$ of $H^2(C)$, there exists an isomorphism in $D(\mathfrak{A})$ between $C$ and a complex of the form $P \xrightarrow{\psi} P$, where $P$ is a finitely generated free $\mathfrak{A}$-module, the first term is placed in degree one and both of the following conditions are satisfied.
\begin{itemize}
\item[(a)] The natural map $P^* \to \ker(\psi)^*  \cong H^1(C)^*$ is surjective.
\item[(b)] Set $a := |\mathcal{X}|$, denote the elements of $\mathcal{X}$ by $\{x_i\}_{1\le i\le a}$ and write $d$ for the $\mathfrak{A}$-rank of $P$. Then $d \ge a$ and there exists an ordered basis $\{b_i\}_{1\le i\le d}$ of $P$ such that $\im(\psi)$ is contained in the $\mathfrak{A}$-submodule generated by $\{b_i\}_{a < i \le d}$ and for each $i$ with $1\le i\le a$ the natural map $P \to {\rm cok}(\psi)\cong H^2(C)$ sends $b_i$ to $x_i$.
\end{itemize}
\end{itemize}
\end{lemma}

\begin{proof} It is enough to prove claim (i) after localising at each prime ideal of $R$ and so we may assume that $C$ is represented by a complex of the form $P \rightarrow P$, where $P$ is a finitely generated free $\mathfrak{A}$-module and the first term is placed in degree one.

For any complex $C'$ in $D^{\rm p}(\mathfrak{A})$ we set $(C')^\dagger := R\Hom_\mathfrak{A}(C,\mathfrak{A}[-3])$.

Then, since $P$ is both isomorphic to $P^*$ and reflexive, the complex $C^\dagger$ belongs to $D^{\rm s}(\mathfrak{A})$ and the complex  $(C^\dagger)^\dagger$ identifies with $C$.

In addition, by using the universal coefficient spectral sequence, one computes that $H^1(C) = H^1((C^\dagger)^\dagger)$ identifies with $H^2(C^\dagger)^*$ and that there is a natural exact sequence of $\mathfrak{A}$-modules
\[ 0 \to {\rm Ext}^1_\mathfrak{A}(H^2(C),\mathfrak{A}) \to H^2(C^\dagger) \to H^1(C)^* \to {\rm Ext}^2_\mathfrak{A}(H^2(C),\mathfrak{A}) \to 0.\]

In particular, since the modules ${\rm Ext}^1_\mathfrak{A}(H^2(C),\mathfrak{A})$ and ${\rm Ext}^2_\mathfrak{A}(H^2(C),\mathfrak{A})$ are both finite, the $\mathfrak{A}$-linear dual of this exact sequence induces an injective homomorphism from $(H^1(C)^*)^*$ to $H^2(C^\dagger)^* = H^1(C)$ that is inverse to the canonical map $H^1(C) \to (H^1(C)^*)^*$. This shows that $H^1(C)$ is reflexive, and hence proves claim (i).

To prove claim (ii) we first fix a surjective homomorphism of $\mathfrak{A}$-modules $\theta: P_1 \to H^1(C)^*$ where $P_1$ is finitely generated and free. Then, taking account of claim (i), the $\mathfrak{A}$-linear dual $\theta^*$ of $\theta$ gives an injective homomorphism from $H^1(C)$ to the (free) $\mathfrak{A}$-module $P_1^*\cong P_1$ with the property that its linear dual $(\theta^*)^*$ is the surjective homomorphism $\theta$.

By a standard argument one now shows that $C$ is isomorphic to a complex of the form $P_1 \xrightarrow{\psi'} M$, where $\ker(\psi') = \im(\theta^*)$ and the module $M$ is finitely generated. Then, since $P_1$ is free and $C$ belongs to $D^{\rm p}(\mathfrak{A})$ the module $M$ has a finite projective resolution and hence, as $\mathfrak{A}$ has dimension one, there exists an exact sequence of finitely generated $\mathfrak{A}$-modules
\[ 0 \to P_2 \xrightarrow{\psi''} P_3 \to M \to 0\]
where $P_2$ is projective and $P_3$ is free. It is then clear that $C$ is isomorphic in $D(\mathfrak{A})$ to the complex $P_1\oplus P_2 \to P_3$, where the first term is placed in degree one, the differential is $\psi := (\tilde{\psi'},\psi'')$ with $\tilde{\psi'}$ any lift of $\psi'$ through the given surjection $P_2\to M$ and the map $\theta^*$ induces an identification of $H^1(C)$ with $\ker(\psi) = \{(x,y): x \in \im(\theta^*), \psi''(y) = -\tilde{\psi'}(x)\}$.

In particular, since the Euler characteristic of $C$ in $K_0(\mathfrak{A})$ is assumed to vanish, the $\mathfrak{A}$-module $P_1\oplus P_2$ must be isomorphic to the free module $P_3$.

Thus, if we set $P:= P_1\oplus P_2$, then at this stage we have shown $C$ to be isomorphic to a complex $P \xrightarrow{\psi} P$ that has the property in claim (ii)(a).

Next we note that, since the $\mathfrak{A}$-module $X$ generated by $\mathcal{X}$ is a direct summand of $H^2(C)$ we can fix a left inverse $H^2(C) \to X$ to the inclusion $X \subseteq H^2(C)$. Then, since $X$ is free and the composite homomorphism $\pi: P \to {\rm cok}(\psi) \cong H^2(C) \to X$ is surjective, we obtain a direct sum decomposition of $\mathfrak{A}$-modules $P = \ker(\pi)\oplus X'$ in which $\ker(\pi)$ is free and $X'$ is mapped bijectively by $\pi$ to $X$.

Hence, since $\im(\psi) \subseteq \ker(\pi)$, we obtain a basis of the sort required in claim (ii)(b) by taking $\{b_i\}_{a < i \le d}$ to be any basis of $\ker(\pi)$ and, for each $1\le i\le a$, defining $b_i$ to be the unique element of $X'$ with $\pi(b_i) = x_i$.
\end{proof}

\subsubsection{}We now complete the proof of Theorem \ref{char els} by proving claim (ii).

As a first step we note that, since $\mathcal{X}$ is now assumed to be separable it spans a free $\mathfrak{A}$-module $X$ of rank $a$. In this case it is thus clear that $e_{C,(a)} = 1$ (so $\mathfrak{A}' = \mathfrak{A}$) and hence that Theorem \ref{char els}(ii) with $x=1$ gives an inclusion $y^a\cdot I(\eta) \subseteq {\rm Fit}^a_\mathfrak{A}(H^2(C))$.

It therefore only remains to prove that ${\rm Fit}^a_\mathfrak{A}(H^2(C))$ is contained in $I(\eta)$. By Lemma \ref{adm-rep}, one can choose a distinguished representative of $C$ of the form $P\xrightarrow{\psi} P$ with the described properties. Consider the presentation of $H^2(C)$ given by $P\xrightarrow{\psi}P\xrightarrow{\pi} H^2(C)$. Define $\psi_i:= b_i^* \circ \psi$ where $b_i^*$ is the dual of $b_i$. Note that $\psi_i=0$ for each $i\leq a$ because, by Lemma \ref{adm-rep} (ii), $\im(\psi)$ is contained in the $\mathfrak{A}$-module generated by $\{b_i\}_{a<i\leq d}$. Thus, $\Fit_\mathfrak{A}^a(H^2(C))$ is generated by the minors $\{\det(\psi_i(b_{\sm(k)})_{a<i,k\leq d}\}_{\sm\in\mathfrak{S}_{d,a}}$, where $\mathfrak{S}_{d,a}$ denotes the set $\{\sm\in S_d \! : \! \sm(1)<\cdots<\sm(a)\mbox{ and } \sm(a+1)<\cdots <\sm(d)\}$.

In the sequel, we identify $\Det_\mathfrak{A}(C)$ with $\bigwedge^d_{\mathfrak{A}}P^*\otimes_\mathfrak{A}\bigwedge^d_{\mathfrak{A}}P$. Define $z_b=z{\bigwedge}_{i=1}^d b_i$ in ${\bigwedge}_{\mathfrak{A}}^d P$ with $z\in \mathfrak{A}^\times$ to be the pre-image of $\mathcal{L}^{-1}$ under the composite isomorphism of $\mathfrak{A}$-modules
$${\bigwedge}_{\mathfrak{A}}^d P \stackrel{\sim}{\longrightarrow} {\bigwedge}_{\mathfrak{A}}^d P \otimes {\bigwedge}_{\mathfrak{A}}^dP^* = {\rm Det}^{-1}_{\mathfrak{A}}(C) \xrightarrow{\vartheta_\lambda} \mathfrak{A}\cdot\mathcal{L}^{-1},$$
where the first map sends each $v$ to $ v\otimes{\bigwedge}_{1\leq i\leq d } b_i^*.$

We consider the composition
\begin{multline*}\Det^{-1}_{\mathfrak{A}}(C)\subset \Det^{-1}_{A_E}(E\otimes_R C) \xrightarrow{\times e_{C,a}} e_{C,a}\Det^{-1}_{A_E}(E\otimes_{R}C) \\
 \cong (e_{C,a}\Det_{A_E}(E\otimes_R H^1(C)))\otimes_{A_E}(e_{C,a}\Det^{-1}_{A_E}(E\otimes_R H^2(C)))\end{multline*}
where the inclusion is obvious and the final map is the `passage to cohomology' isomorphism.

Then \cite[Lem 4.3]{bks1} in this setting implies that
\begin{center}
$(-1)^{a(d-a)}\left(\bigwedge_{i=1}^{i=a}\lambda\circ\bigwedge_{a<i\leq d} \psi_i(z_b)\right)\otimes\bigwedge_{i=1}^{i=a} x_i^*=(e_{C,a}\cdot\mathcal{L}^{-1} \bigwedge_{i=1}^{i=a} x_i)\otimes\bigwedge_{i=1}^{i=a} x_i^*.$
\end{center}
Hence, combining this with the definition of the higher special element, we have $\eta=(-1)^{a(d-a)}\bigwedge_{a<i\leq d} \psi_i(z_b)$. The explicit formula in \cite[Prop. 4.1]{bks1} thus implies that
\[
\eta=(-1)^{a(d-a)} z\sum_{\sm\in \mathfrak{S}_{d,a}}{\rm sgn}(\sm)\det(\psi_i(b_{\sm(k)}))_{a<i,k\leq d} (b_{\sm(1)}\wedge\cdots\wedge b_{\sm(a)}),
\]
and hence that
\begin{equation*}\label{Fit-gen}
(b_{\sm(1)}^*\wedge\cdots\wedge b_{\sm(a)}^*)(\eta)=\pm z\cdot\det(\psi_i(b_{\sm(k)})_{a<i,k\leq d}.
\end{equation*}
Since $P^*$ surjects onto $H^1(C)^*$ by construction, this last equality implies that
\[ z\cdot\det(\psi_i(b_{\sm(k)}))_{a<i,k\leq d}\in I(\eta)= \{ (\wedge_{i=1}^{i=a}\varphi_i)(\eta)\! :\! (\varphi_i)_{i}\in \Hom_\mathfrak{A}(H^1(C),\mathfrak{A})^{a}\}.\]
Then, as $z$ is unit in $\mathfrak{A}$, this implies the required inclusion $\Fit^a_\mathfrak{A}(H^2(C))\subseteq I(\eta)$.

This completes the proof of Theorem \ref{char els}.

\subsection{The structure of exterior power biduals}\label{str bidual sect} In this section, we study the exterior power biduals of $H^1(C)$ under the assumption (in general) that $\mathfrak{A}$ is Gorenstein. 

\subsubsection{Definition} In this subsection, we let $X$ be a finitely generated $\mathfrak{A}$-module and recall that we write $(-)^*:=\Hom_{\mathfrak{A}}(-,\mathfrak{A})$ to be the linear dual functor.

\begin{definition}\label{exterior-bidual} {\em For any non-negative integer $a$, we define the `$a$-th exterior power bidual' of $X$ to be the $\mathfrak{A}$-module by setting
\[{\bigcap}_\mathfrak{A}^a X:=\left({\bigwedge}^a_\mathfrak{A} (X^*)\right)^*.\]
}
\end{definition}

\begin{remark}{\em The $A$-module $F\otimes_R X$ is both finitely generated and projective and so there is a natural identification of $F\otimes_R\bigwedge^a_{\mathfrak{A}} X$ with  ${\bigcap}^a_{A}(F\otimes_R X)=\Hom_\mathfrak{A}({\bigwedge}_\mathfrak{A}^a (X^*), A)$. Consequently, the map $x\mapsto (\Phi\mapsto\Phi(x))$ induces an identification
\[\left\{x\in F\otimes_R{\bigwedge}^a_{\mathfrak{A}} X : \Phi(x)\in \mathfrak{A} \mbox{ for all } \Phi\in {\bigwedge}^a_{\mathfrak{A}}(X^*)\right\}\xrightarrow{\sim}{\bigcap}_{\mathfrak{A}}^a X.\]
In this way, we can regard ${\bigcap}_{\mathfrak{A}}^a X$ as an $\mathfrak{A}$-submodule of $F\otimes_R{\bigwedge}_\mathfrak{A}^a X$.
}
\end{remark}

Here we record an easy consequence of Theorem \ref{char els}(i).

\begin{theorem}\label{hse-integrality} Fix data $\mathfrak{A}, C, \lambda$ and $\mathcal{L}$ as in Theorem \ref{char els}. Then for any non-negative integer $a$ and any elements $x$ and $y$ of $\mathfrak{A}$ as in Theorem \ref{char els}(i) the following claims are valid.
\begin{itemize}
\item[(i)] For any ordered subset $\mathcal{X}$ of $H^2(C)_{\rm tf}$ such that $|\mathcal{X}|=a$, the higher special element $\eta_\mathcal{X}$ associated with the data $(C,\lambda,\mathcal{L},\mathcal{X})$ satisfies  $xy^a\cdot\eta_{\mathcal{X}}\in {\bigcap}_{\mathfrak{A}}^aH^1(C)$.
\item[(ii)] There is an inclusion
\[ xy^a\cdot e_{C,a}\cdot\mathcal{L}^{-1}\cdot ({\bigwedge}^a_{\mathfrak{A}}H^2(C))_{\rm tf} \subseteq ({\bigwedge}_{A_E}^a\lambda)({\bigcap}_{\mathfrak{A}}^aH^1(C)).\]
\end{itemize}
\end{theorem}

\begin{proof} Since ${\rm Fit}^a_\mathfrak{A}(H^2(C))$ is contained in $\mathfrak{A}$, the containment in claim (i) follows immediately from the first containment in Theorem \ref{char els}(i).

Next we note that the $\mathfrak{A}$-module $({\bigwedge}^a_{\mathfrak{A}}H^2(C))_{\rm tf}$ is spanned by elements of the form
${\wedge}_{x \in \mathcal{X}}x$ where $\mathcal{X}$ runs over (ordered) subsets of $H^2(C)$ of cardinality $a$.

This implies that the  $\mathfrak{A}$-module $e_{C,a}\cdot\mathcal{L}^{-1}\cdot ({\bigwedge}^a_{\mathfrak{A}}H^2(C))_{\rm tf}$ is generated by  elements of the form $({\bigwedge}_{A_E}^a\lambda)(\eta_{\mathcal{X}})$ and so claim (ii) follows directly from claim (i). \end{proof}

\subsubsection{} In this section, we study the structure of the quotient of exterior power biduals by the submodule generated by the corresponding special elements.

For an $\mathfrak{A}$-lattice $X$ and an idempotent $e$ of $A$ we set
\[ X^e:=\{x\in X: e\cdot x=x \mbox{ in } F\otimes_R X\}\]
(so that $X^e = e\cdot X$ if $e$ belongs to $\mathfrak{A}$).

For any $\mathfrak{A}$-module $X$ we write $X_{\rm tor}$ for the submodule comprising all $R$-torsion elements. For any element $x$ of $X$ we denote the $\mathfrak{A}$-submodule that $x$ generates by $\langle x\rangle$.

The main algebraic result in this section is the following.

\begin{theorem}\label{str-bidual} Assume that $\mathfrak{A}$ is Gorenstein. Fix data $(C,\lambda, \mathcal{L}, \mathcal{X})$ as in Definition \ref{hse def}. Set $a := |\mathcal{X}|$, $e_a:=e_{C,a}$, $e_{(a)}:=e_{C,(a)}$ and $\mathfrak{A}':=\mathfrak{A}e_{(a)}$ and write $\eta$ in place of  $\eta_{(C,\lambda,\mathcal{L},\mathcal{X})}$.

Then for any elements $x$ of $\mathfrak{A} \cap \mathfrak{A}'$ and $y$ of  ${\rm Ann}_{\mathfrak{A}}
({\rm Ext}^2_{\mathfrak{A}'}(\mathfrak{A}'\otimes_{\mathfrak{A}}H^2(C),\mathfrak{A}'))$ such that both $xe_{a}$ and $ye_{a}$ are invertible in $Ae_a$, the following claims are valid.

\begin{itemize}
\item[(i)] There exists a canonical perfect $\mathfrak{A}$-bilinear pairing of finite modules
%
\[ \left(\dfrac{{\bigcap}^a_\mathfrak{A} H^1(C)}{\langle xy^a\cdot\eta\rangle}\right)_{\!{\rm tor}}\times \left(\dfrac{\mathfrak{A}}{xy^a\cdot I(\eta)}\right)_{\rm tor} \to \frac{A}{\mathfrak{A}}.\]
(An explicit description of this pairing is given in Remark \ref{explicit desc} below).

\item[(ii)] There exists a canonical short exact sequence of $\mathfrak{A}$-modules
\[0\rightarrow\left(\dfrac{\Fit^a_{\mathfrak{A}}(H^2(C))}{xy^a\cdot I(\eta)}\right)_{\rm tor}\rightarrow\Hom_\mathfrak{A}\left(\left(\dfrac{{\bigcap}^a_\mathfrak{A} H^1(C)}{\langle xy^a\cdot\eta\rangle}\right)_{\!{\rm tor}},\frac{A}{\mathfrak{A}}\right)\rightarrow\left(\dfrac{\mathfrak{A}}{\Fit_\mathfrak{A}^a(H^2(C))^{e_a}}\right)_{\rm tor}\rightarrow 0.\]
\end{itemize}
\end{theorem}

\begin{remark}{\em If $e_{(a)}$ belongs to $\mathfrak{A}$ (as is automatically the case, for example, if $\mathcal{X}$ is separable), then $\mathfrak{A}\cap\mathfrak{A}'=\mathfrak{A}'$ is Gorenstein and so $\Ext^2_{\mathfrak{A}'}(\mathfrak{A}'\otimes_\mathfrak{A}H^2(C), \mathfrak{A}')$ vanishes. In such a case we can therefore take $x=y=e_{(a)}$ and omit them from the statement of Theorem \ref{str-bidual}. In particular, if $\mathcal{X}$ is separable, then Theorem \ref{char els}(ii) and Remark \ref{explicit verson of char els} combine to imply  $I(\eta)=\Fit^a_{\mathfrak{A}}(H^2(C))$ (which is equal to $\Fit^a_{\mathfrak{A}}(H^2(C))^{e_a}$ in this case) and hence that there is a canonical perfect pairing
\[ \left(({\bigcap}^a_\mathfrak{A} H^1(C))/\langle \eta\rangle\right)_{\!{\rm tor}}\times \left(\mathfrak{A}/\Fit_\mathfrak{A}^a(H^2(C))\right)_{\rm tor} \to A/\mathfrak{A}.\]
}\end{remark}

The proof of Theorem \ref{str-bidual} will occupy the rest of \S\ref{str bidual sect}.

\subsubsection{Dual modules} In the sequel for any $\mathfrak{A}$-module $X$ we set $X^*:=\Hom_\mathfrak{A}(X,\mathfrak{A})$ and $X^\vee:=\Hom_\mathfrak{A}(X,A/\mathfrak{A})$ and regard both as $\mathfrak{A}$-modules in the natural way.

In this section we collect some useful results regarding these dual modules.

\begin{proposition}\label{dual-frac} Assume $\mathfrak{A}$ is Gorenstein. Fix an idempotent $e$ in $A$ and a finitely generated $\mathfrak{A}$-module $X$.
\begin{itemize}
\item[(i)] If $X$ is $R$-torsion-free, then $X$ is reflexive. i.e. $X^{**}=X$.
\item[(ii)] $(Xe)^*\cong (X^*)^e$ and if $X$ is $R$-torsion-free, then $X^*e\cong(X^e)^*$.
\item[(iii)] If $X$ is a finite $\mathfrak{A}$-module, then one has $(X^\vee)^\vee=X$.
\end{itemize}
\end{proposition}
\begin{proof} Claim (i) follows directly from Bass's result that, if $\mathfrak{A}$ is Gorenstein, then every finitely generated $R$-torsion-free $\mathfrak{A}$-module is reflexive (cf. \cite[Th. 6.2]{bass}).

For claim (ii), define $\pi: X\rightarrow Xe$ by $x\mapsto xe$. We claim that the assignment $\theta\mapsto \theta\circ \pi$ gives an the isomorphism $(Xe)^*\xrightarrow{\sim}(X^*)^e$. To prove this, we will prove that this assignment is invertible. Suppose we are given an element $\theta\in(X^*)^e$, i.e. $\theta\in X^*$ such that $\theta(x)=e\theta(x)$. If we denote $\theta_F : F\otimes Xe\rightarrow A$ to be its induced morphism of $\theta$ by extension of scalars, then for any $x\in X$, one has $\theta_F(ex)=e\theta_F(x)=\theta(x)\in\mathfrak{A}$. Therefore the restriction of $\theta_F$ on $Xe$ belongs to  $(Xe)^*$. We are done. For the second half of the statement of (ii), suppose that $X$ is $R$-torsion-free. Then so is $Xe$. So (i) implies that these modules are both reflexive. Hence, the isomorphism $(Xe)^*\cong (X^*)^e$ implies that $Xe\cong((X^*)^e)^*$. The claim follows by replacing $X$ with $X^*$ in the later isomorphism.

To prove claim (iii) we let $d$ be a sufficient large integer so that there exists a surjection $\mathfrak{A}^d\rightarrow X$ and denote its kernel by $K$. Since $X$ is a finite module and $\mathfrak{A}$ is $R$-torsion-free, we have $\Hom_\mathfrak{A}(X,\mathfrak{A})=0$. Therefore, by applying the functor $\Hom_\mathfrak{A}(X,-)$ to the short exact sequence $0\rightarrow \mathfrak{A}\rightarrow A\rightarrow A/\mathfrak{A}\rightarrow 0$, we have that $\Ext_{\mathfrak{A}}^1(X,\mathfrak{A})=X^\vee$. Recall that for any free $\mathfrak{A}$-module $N$, $\Ext_{\mathfrak{A}}^1(N,\mathfrak{A})=0$. Therefore,  if we apply the functor $\Hom_\mathfrak{A}(-,\mathfrak{A})$ to the short exact sequence $0\rightarrow K\rightarrow \mathfrak{A}^d\rightarrow X\rightarrow 0$, we obtain another short exact sequence $0\rightarrow(\mathfrak{A}^d)^*\rightarrow K^*\rightarrow X^\vee\rightarrow 0$. Now by applying the functor $\Hom_\mathfrak{A}(-,\mathfrak{A})$ again to this new sequence and repeating the above argument, we obtain a short exact sequence
$0\rightarrow (K^*)^*\rightarrow((\mathfrak{A}^d)^*)^*\rightarrow  (X^\vee)^\vee\rightarrow 0$. Since $\mathfrak{A}^d$, and hence $K$, are finitely generated and $R$-torsion-free, we have by (i) that $((\mathfrak{A}^d)^*)^*=\mathfrak{A}^d$ and $(K^*)^*=K$. Hence, the commutative diagram
\[ \begin{CD}
0 @> >> (K^*)^* @> >> ((\mathfrak{A}^d)^*)^* @>  >>  (X^\vee)^\vee@> >> 0\\
@. @\vert @\vert @V VV \\
0 @> >> K @> >> \mathfrak{A}^d @>  >>  X@> >> 0\\
\end{CD}\]
implies that $(X^\vee)^\vee=X$.\end{proof}

\subsubsection{} We can now prove Theorem \ref{str-bidual}. For convenience, we abbreviate $xy^a\cdot\eta$ to $\tilde{\eta}$.

At the outset we note that, by its very definition, $\eta$ has non-zero component at each simple component of $Ae_a$. As $xe_a$ and $ye_a$ are both assumed to be invertible in $Ae_a$, the element $\tilde{\eta} = e_a\tilde{\eta}$ also has non-zero component at each simple component of $Ae_a$ and so the assignment $e_a \mapsto \tilde{\eta}$ induces an isomorphism of $\mathfrak{A}$-modules
\begin{equation}\label{first iso} \mathfrak{A}e_a\cong \langle \tilde{\eta}\rangle.\end{equation}
This in turn implies that the assignment $\Phi\mapsto \Phi(\tilde{\eta})$ gives an isomorphism of $\mathfrak{A}$-modules
\begin{equation}\label{isom2}\langle \tilde{\eta}\rangle^*\xrightarrow{\sim}\mathfrak{A}^{e_a}
\end{equation}

We next observe that Proposition \ref{dual-frac}(ii) induces a natural isomorphism
\begin{equation*}\left(\left({\bigcap}^a_\mathfrak{A} H^1(C)\right)^{e_a}\right)^*\xrightarrow{\sim}
e_a\cdot\left({\bigcap}^a_\mathfrak{A} H^1(C)\right)^* =e_a\cdot\bigwedge^a_\mathfrak{A} (H^1(C)^*).\end{equation*}
In particular, since the definition of $I(\tilde{\eta})$ implies that
\[{\bigwedge}^a_\mathfrak{A} (H^1(C)^*)\cdot e_a (\tilde{\eta})= {\bigwedge}^a_\mathfrak{A} (H^1(C)^*)\cdot(\tilde{\eta})= I(\tilde{\eta})=x y^a\cdot I(\eta),\]
the map $e_a\cdot{\bigwedge}^a_\mathfrak{A} (H^1(C)^*) \to I(\tilde{\eta})$ that sends each $\Phi$ to $\Phi(\tilde{\eta})$ induces an isomorphism
\begin{equation}\label{isom1}\left(\left({\bigcap}^a_\mathfrak{A} H^1(C)\right)^{e_a}\right)^*\xrightarrow{\sim} xy^a\cdot I(\eta).
\end{equation}

We set $Q :=({\bigcap}^a_\mathfrak{A} H^1(C))^{e_a}/\langle \tilde{\eta}\rangle$ and note that this module is finite as a consequence of the isomorphism (\ref{first iso}).

Since $Q$ is finite the module $\Ext^1_{\mathfrak{A}}(Q,\mathfrak{A})$ identifies with $Q^\vee$. In addition, since $\mathfrak{A}$ is Gorenstein and the module $({\bigcap}^a_\mathfrak{A} H^1(C))^{e_a}$ is $R$-torsion-free, the group $\Ext^1_{\mathfrak{A}}(({\bigcap}^a_\mathfrak{A} H^1(C))^{e_a}, \mathfrak{A})$ vanishes.

Consequently, if we apply the functor $\Hom_{\mathfrak{A}}(-,\mathfrak{A})$ to the tautological exact sequence
\[ 0\rightarrow \langle\tilde{\eta}\rangle\rightarrow ({\bigcap}^a_\mathfrak{A} H^1(C))^{e_a}\rightarrow Q\rightarrow 0,\]
then we obtain the upper row in the following exact commutative diagram
%
\[ \begin{CD}
0 @> >> (({\bigcap}^a_\mathfrak{A} H^1(C))^{e_a})^* @> >> \langle\tilde{\eta}\rangle^* @>  >>  Q^\vee@> >> 0\\
@. @V(\ref{isom1})VV @V (\ref{isom2})VV @.  \\
0 @> >> x y^a\cdot I(\eta) @> >> \mathfrak{A}^{e_a} @>  >>  \dfrac{\mathfrak{A}^{e_a}}{x y^a\cdot I(\eta)}@> >> 0.\\
\end{CD}\]

Note that the lower row in this diagram is the tautological short exact sequence and that, after unwinding the definitions of all involved maps, it is straightforward to check the square commutes.

From the diagram we therefore obtain an induced isomorphism of $\mathfrak{A}$-modules of the form
\begin{equation}\label{isom3} \iota: \left(\dfrac{({\bigcap}^a_\mathfrak{A} H^1(C))^{e_a}}{\langle\tilde{\eta}\rangle}\right)^\vee = Q^\vee \cong \dfrac{\mathfrak{A}^{e_a}}{x y^a\cdot I(\eta)}.
\end{equation}

This isomorphism in turn gives rise to an $\mathfrak{A}$-bilinear pairing of finite modules
\[ \dfrac{({\bigcap}^a_\mathfrak{A} H^1(C))^{e_a}}{\langle\tilde{\eta}\rangle}\times \dfrac{\mathfrak{A}^{e_a}}{x y^a\cdot I(\eta)} \to \dfrac{A}{\mathfrak{A}}\]
that sends each pair $(u,v)$ to $(\iota^{-1}(v))(u)$ and this pairing is perfect as a consequence of Proposition \ref{dual-frac}(iii). To derive the result of Theorem \ref{str-bidual}(i) we then need only apply the result of Lemma \ref{torsion} below to each of the modules involved in the above pairing.

To prove Theorem \ref{str-bidual}(ii) we observe Theorem \ref{char els}(i) implies that $xy^a\cdot I(\eta)$ is contained in $\Fit^a_{\mathfrak{A}}(H^2(C))^{e_a}$ and hence that there exists a tautological short exact sequence
\[ 0\rightarrow\dfrac{\Fit^a_{\mathfrak{A}}(H^2(C))^{e_a}}{x y^a\cdot I(\eta)}\rightarrow\dfrac{\mathfrak{A}^{e_a}}{x y^a\cdot I(\eta)}\rightarrow\dfrac{\mathfrak{A}^{e_a}}{\Fit_\mathfrak{A}^a(H^2(C))^{e_a}}\rightarrow 0.\]

This gives the exact sequence in Theorem \ref{str-bidual}(ii) after one takes account of the isomorphism (\ref{isom3}) and then applies the observation in the following result.

\begin{lemma}\label{torsion} Let $X$ be a $\mathfrak{A}$-lattice  and $e$ be an idempotent of $A$. If $X'$ is a full sub-lattice of $X^e$, then one has $(X/X')_{\rm tor}=X^e/X'$.
\end{lemma}
\begin{proof} As $X'$ is a full sub-lattice of $X^e$, one has $F\otimes_R X'=F\otimes_R X^e$ and this implies that $X^e/X'$ is contained in $(X/X')_{\rm tor}$.

To prove the reverse inclusion, we note that since $X$ is a $\mathfrak{A}$-lattice (and therefore $R$-torsion-free), the image in $X/X'$ of any element $x$ of $X$ with $x(1-e)\neq 0$ is not annihilated by any non-zero element of $R$. This immediately implies $(X/X')_{\rm tor}$ is contained inn $X^e/X'$, as required.
\end{proof}

\begin{remark}\label{explicit desc}{\em The above argument gives the following explicit description of the pairing in Theorem \ref{str-bidual}(i). If $u$ and $v$ are any elements of finite order in ${\bigcap}^a_\mathfrak{A} H^1(C)/\langle xy^a\cdot\eta\rangle$ and $\mathfrak{A}/(xy^a\cdot I(\eta))$, then they can be respectively represented by elements $\tilde u$ and $\tilde v$ in $({\bigcap}^a_\mathfrak{A} H^1(C))^{e_a}$ and $\mathfrak{A}^{e_a}$ and the pairing in Theorem \ref{str-bidual}(i) sends $(u,v)$ to $\tilde v\cdot \theta(\tilde u)$, where $\theta$ is the unique element of $\langle xy^a\cdot \eta\rangle^*$ with $\theta(xy^a\cdot \eta) = e_a$.} \end{remark}

\begin{remark}{\em If $e_a$ belongs to $\mathfrak{A}$, then the isomorphism (\ref{isom3}) implies $\left({\bigcap}^a_{\mathfrak{A}}H^1(C)/\langle xy^a \cdot \eta\rangle\right)_{\rm tor}$ is a cyclic $\mathfrak{A}$-module. To deal with the general case we assume (after localising) that $\mathfrak{A}$ is a (one-dimensional) local ring with maximal ideal $\mathfrak{m}$ and write $g_\mathfrak{A}(N)$ for  the minimal number of generators of an $\mathfrak{A}$-module $N$. We recall that, under these conditions, the supremum $\nu(\mathfrak{A})$ of $g_\mathfrak{A}(I)$ as $I$ runs over all ideals of $\mathfrak{A}$ is finite and bounded explicitly in terms of the Hilbert function of the graded module ${\rm gr}_\mathfrak{m}(\mathfrak{A})$ (cf. \cite{shaler}). In this case, therefore, the isomorphism (\ref{isom3}) implies that $g_\mathfrak{A}(\left({\bigcap}^a_{\mathfrak{A}}H^1(C)/\langle xy^a \cdot \eta\rangle\right)_{\rm tor})$ is at most $\nu(\mathfrak{A})$.}
\end{remark}

\subsection{Separability, $G$-valued pairings and congruences} In this section we shall describe an abstract formalism of pairings that are valued in (direct sums of) Galois groups and associated `refined class number formulas' that relate the higher special elements of separable subsets.

Upon appropriate specialisation (in the setting of \S\ref{csc exam}), the main result of this section can be combined with the standard  leading term conjecture for Dirichlet $L$-series to recover the central conjectures that are formulated by Mazur and Rubin in \cite{MR, MR2} and by the second author in \cite{sano} and, with a little more work, the same formalism can also be used recover the refined version of these conjectures that are studied in \cite{bks1}.

However, our purpose here is to abstract certain arguments from loc. cit. in order to establish a general formalism that one can then specialise to other interesting cases.

In particular, in \cite{bmc2} the main result (Theorem \ref{mrstheorem}) of this section is applied in the setting of \S\ref{fc exam} in order to formulate new refinements of the Birch and Swinnerton-Dyer conjecture that both refine and extend the central conjectures that are formulated by Mazur and Tate in \cite{mt}.

We again fix a Dedekind domain $R$ of characteristic zero and write $F$ for its fraction field. Throughout this section we also assume to be given a finite abelian group $G$, a subgroup $J$ of $G$ and an object $C$ of $D^{\rm s}(R[G])$.

\subsubsection{}We start with a useful technical lemma concerning the object
\[ C_J := R[G/J]\otimes_{R[G]}^\mathbb{L}C\]
of $D(R[G/J])$.

\begin{lemma}\label{descent lema} The complex $C_J$ belongs to $D^{\rm s}(R[G/J])$ and there are natural identifications
$ H^1(C_J) = H^1(C)^J$ and $H^2(C_J) = H^2(C)_J$.\end{lemma}

\begin{proof} Since $C$ belongs to $D^{\rm s}(R[G])$ it can be represented by a complex of finitely generated projective $R[G]$-modules of the form $P_1 \xrightarrow{\psi} P_2$, where the first term occurs in degree one (see Remark \ref{representative}).

It follows that $C_J$ is represented by the complex of $J$-coinvariants $P_{1,J} \xrightarrow{\psi_J} P_{2,J}$ and this representative makes it clear both that $C_J$ belongs to $D^{\rm s}(R[G/J])$ and that there is a canonical identification $H^2(C_J) = {\rm cok}(\psi_J) = {\rm cok}(\psi)_J = H^2(C)_J$.

Finally, the identification $H^1(C_J) = H^1(C)^J$ is induced by the commutative diagram
\[ \begin{CD}
0 @> >> H^1(C_J) @> >> P_{1,J} @> \psi_J >> P_{2,J}\\
@. @. @V VV @V VV \\
0 @> >> H^1(C)^J @> >> P_1^J @> \psi^J >> P_2^J\end{CD}\]
in which both rows are exact and the two vertical arrows are the bijections $P_{i,J} \to P_i^J$ that are induced by sending each element $x$ of $P_i$ to $\sum_{j \in J}j(x)$. \end{proof}

\begin{remark}{\em In the sequel we shall sometimes use the identifications $ H^1(C_J) = H^1(C)^J$ and $H^2(C_J) = H^2(C)_J$ given by Lemma \ref{descent lema} without explicit comment. }\end{remark}

We now assume to be given a separable subset $\mathcal{X}$ of $H^2(C)$ and, noting that the image $\mathcal{X}_J$ of $\mathcal{X}$ under the projection $H^2(C) \to H^2(C)_J = H^2(C_J)$ is separable in $H^2(C_J)$, we also fix a separable subset $\mathcal{X}'$ of $H^2(C_J)$ that contains $\mathcal{X}_J$.

We set $a := |\mathcal{X}|$ and $a':= |\mathcal{X}'|$ (so that $a' \ge a$) and write $X$ and $X'$ for the $R[G]$-module direct summands of $H^2(C)$ and $H^2(C_J)$ that are respectively generated by $\mathcal{X}$ and $\mathcal{X}'$.

We write $I_R(J)$ for the augmentation ideal of $R[J]$ and $\mathcal{I}_R(J)$ for the ideal of $R[G]$ generated by $I_R(J)$. Then the canonical short exact sequence
\begin{equation}\label{can ses} 0 \to \mathcal{I}_R(J) \to R[G] \to R[G/J]\to 0\end{equation}
induces an exact triangle in $D(R[G])$ of the form
\begin{equation*}\label{descent triangle} \mathcal{I}_R(J)\otimes^\mathbb{L}_{R[G]}C \to C \to C_J \xrightarrow{\beta}\mathcal{I}_R(J)\otimes^\mathbb{L}_{R[G]}C[1].\end{equation*}

For each $x$ in $\mathcal{X}'\setminus \mathcal{X}_J$ this triangle gives rise to a  composite homomorphism of $R[G]$-modules ${\rm Boc}^C_{x}$ of the form
\begin{multline*} H^1(C)^J =  H^1(C_J)  \xrightarrow{H^1(\beta)} H^1(\mathcal{I}_R(J)\otimes_{R[G]}C[1])  = \mathcal{I}_R(J)\otimes_{R[G]}H^2(C) \\
 \to \mathcal{I}_R(J)/\mathcal{I}_R(J)^2\otimes_{R[G]} H^2(C_J)\xrightarrow{{\rm id}\otimes c_x} \mathcal{I}_R(J)/\mathcal{I}_R(J)^2. \end{multline*}
Here the unlabeled arrow is induced by the projection maps $\mathcal{I}_R(J) \to \mathcal{I}_R(J)/\mathcal{I}_R(J)^2$ and $H^2(C) \to H^2(C)_J = H^2(C_J)$ and $c_x$ denotes the homomorphism
$H^2(C_J) \to R[G/J]$ obtained as the composite of a choice of projection from $H^2(C_J)$ to $X'$ and the homomorphism $X' \to R[G/J]$ given by projecting an element to its coefficient at the basis element $x$.

\begin{remark}\label{explicit pairings}{\em Since $I_R(J)/I_R(J)^2$ is canonically isomorphic to $R\otimes J$ each map ${\rm Boc}^C_{x}$ can be regarded as taking values in the direct sum $R[G/J]\otimes J$ of copies of $R\otimes J$. There are two important cases in which such pairings have been constructed in the literature. To discuss them we fix a finite Galois extension of number fields $F/k$ that is unramified outside a finite set of places $S$ containing all archimedean and all $p$-adic places, for some fixed prime $p$, and set $G := \Gal(F/k)$. We also assume that $R = \ZZ_p$.

\noindent{}(i) In the notation of Proposition \ref{admiss constructions}(iii), the complex $C := R\Hom_{\ZZ_p}(C_\Sigma(\ZZ_p),\ZZ_p[-3])$ is an object of $D^{\rm s}(\ZZ_p[G])$ for which the subsets of $S$ comprising places that split completely in $F
/k$, respectively in $F^J/k$, correspond naturally to separable subsets $\mathcal{X}$ and $\mathcal{X}'$ of $H^2(C)$ and $H^2(C_J)$. In this context the computation in \cite[Lem. 5.21]{bks1} shows that for $x$ in $\mathcal{X}'\setminus \mathcal{X}_J$ the map ${\rm Boc}^C_{x}$ can be described in terms of the local reciprocity map at the place in $S$ that corresponds to $x$.

\noindent{}(ii) Assume now the hypotheses of \S\ref{fc exam} and set $C:= C_f(T_F)$ and $J = G$. Take $\mathcal{X}$ to be the empty set and $\mathcal{X}'$ to be any basis of the free abelian group $\Hom_\ZZ(A(k),\ZZ)$. Then the sets $\mathcal{X}$ and $\mathcal{X}'$ are respectively separable in $H^2(C)=\ZZ_p\otimes\Hom_\ZZ(A(F),\ZZ)$ and in $H^2(C_J) = \ZZ_p\otimes\Hom_\ZZ(A(k),\ZZ)$ and, in this setting, it can be shown that the homomorphisms ${\rm Boc}^C_x$ extend the canonical $G$-valued height pairings that are constructed by Mazur and Tate in \cite{mt} (for details see the appendix to \cite{bmc2}).
}\end{remark}

In the sequel we set $Q := I_R(J)^{a'-a}/I_R(J)^{1+a'-a}$. Then by the same method as in \cite[Prop. 2.7]{sano} (or \cite[Cor. 2.1]{MR2}), the maps ${\rm Boc}^C_{x}$ for $x$ in $\mathcal{X}'\setminus \mathcal{X}_J$ can be combined to give a canonical homomorphism
\begin{equation}
{\rm Boc}^C_{\mathcal{X},\mathcal{X'}}: {\bigcap}_{R[G/J]}^{a'} H^1(C)^J \longrightarrow ({\bigcap}_{R[G/J]}^{a} H^1(C)^J)\otimes_{R}Q.
\end{equation}

Finally, following an original idea of Darmon in \cite{D}, we define the `norm operator'
$$\mathcal{N}_J : {\bigcap}_{R[G]}^a H^1(C) \longrightarrow ({\bigcap}_{R[G]}^a H^1(C)) \otimes_{R} R[J]/I_R(J)^{1+a'-a}$$
to be the homomorphism induced by sending each $x$ to $\sum_{\sigma \in J} \sigma (x) \otimes \sigma^{-1}$ and recall that the approach of \cite[Prop. 4.12]{bks1} constructs a canonical injective homomorphism of $R[G]$-modules
\[ \nu_J : ({\bigcap}_{R[G/J]}^{a} H^1(C)^J)\otimes_{R}Q \rightarrow ({\bigcap}_{R[G]}^{a} H^1(C))\otimes_{R}Q.\]

We can now state the main result of this section.

\begin{theorem}\label{mrstheorem} We suppose given an extension $E$ of $F$, an isomorphism of $E[G]$-modules $\lambda: H^1(C)_E \to H^2(C)_E$ and  an element $\mathcal{L}$ of $E[G]^\times$ with
$\vartheta_\lambda({\rm Det}_{R[G]}(C)) = (R[G]\cdot\mathcal{L}^{-1},0)$.

Then the following claims are valid.
\begin{itemize}
\item[(i)] Write $\mathcal{L}_J$ for the image of $\mathcal{L}$ under the projection $E[G] \to E[G/J]$ and $\lambda_J$ for the composite isomorphism of $E[G/J]$-modules
\[ H^1(C_J)_E = H^1(C)^J_E \to H^1(C)_{E,J} \xrightarrow{\lambda} H^2(C)_{E,J} = H^2(C_J)_E\]
where the first arrow is induced by the projection $H^1(C)\to H^1(C)_J$. Then one has $\vartheta_{\lambda_J}({\rm Det}_{R[G/J]}(C_J)) = (R[G/J]\cdot\mathcal{L}^{-1}_J,0)$.

\item[(ii)] Set $\eta_\mathcal{X} := \eta_{(C,\lambda,\mathcal{L},\mathcal{X})}$ and, with the notation of claim (i), also $\eta_{\mathcal{X}'} :=  \eta_{(C_J,\lambda_J,\mathcal{L}_J,\mathcal{X}')}$.

Then $\eta_{\mathcal{X}}$ belongs to ${\bigcap}_{R[G]}^{a}H^1(C)$ and $\eta_{\mathcal{X}'}$ to ${\bigcap}_{R[G/J]}^{a'} H^1(C)^J$ and 
are such that
\[ \mathcal{N}_J(\eta_\mathcal{X})=(-1)^{a(a'-a)}\cdot \nu_J( {\rm Boc}^C_{\mathcal{X},\mathcal{X}'}  (\eta_{\mathcal{X}'}))\]
in $({\bigcap}_{R[G]}^{a} H^1(C))\otimes_{R} Q
$.
\end{itemize}
\end{theorem}

\begin{remark}{\em In the setting of Remark \ref{explicit pairings}(i) the congruence in Theorem \ref{mrstheorem}(ii) can be used to show that the central conjectures of Mazur and Rubin in \cite{MR2} and of the second author in \cite{sano} follow as consequences of the relevant special case of the equivariant Tamagawa number conjecture (see \cite[Th. 5.16]{bks1}). In \cite{bmc2} the congruence of Theorem \ref{mrstheorem}(ii) is considered in detail in the setting of Remark \ref{explicit pairings}(ii).}\end{remark}

\subsubsection{}The proof of Theorem \ref{mrstheorem} will occupy the rest of this section.

Note first that to prove claim (i) it is enough to show $\vartheta_{\lambda_J}({\rm Det}_{R[G/J]}(C_J))$ identifies with the image of $\vartheta_{\lambda}({\rm Det}_{R[G]}(C))$ under the natural projection $(E[G],0) \to (E[G/J],0)$ and this is a straightforward exercise.

In addition, since the sets $\mathcal{X}$ and $\mathcal{X}'$ are respectively assumed to be separable in $H^2(C)$ and $H^2(C_J)$ the containments  $\eta_{\mathcal{X}}\in {\bigcap}_{R[G]}^{a}H^1(C)$ and $\eta_{\mathcal{X}'}\in {\bigcap}_{R[G/J]}^{a'} H^1(C)^J$ in claim (ii) follow from
 Remark \ref{explicit verson of char els}.

However, to prove the congruence relation in claim (ii) we must extend the rather delicate approach used to prove \cite[Th. 5.16]{bks1}.

To do this we label, and thereby order, the elements of $\mathcal{X}$ as $\{x_i\}_{1\le i\le a}$ and the elements of $\mathcal{X}'\setminus \mathcal{X}_J$ as $\{x'_j\}_{a < j \le a'}$. We write $\kappa_J$ for the natural map $H^2(C) \to H^2(C_J)$ and fix left inverses $\sigma_X$ and $\sigma_{X'}$ to the inclusions $X \subseteq H^2(C)$ and $X'\subseteq H^2(C_J)$ respectively.

Since it suffices to verify the stated congruence after localizing at each prime ideal of $R$ we can, and will, assume in the sequel that $R$ is local.

Then, by an easy adaptation of the argument in Lemma \ref{adm-rep}(ii), one shows $C$ is represented by a complex $P\xrightarrow{\psi}P$, where the first term is placed in degree one and $P$ is a finitely generated free $R[G]$-module for which one can fix an ordered basis $\{b_i\}_{1\le i\le d}$ and a surjective homomorphism of $R[G]$-modules $\pi: P \to H^2(C)$ with all of the following properties:
\begin{equation}\label{key props of basis} \begin{cases} \ker(\pi) = \im(\psi),\\
\pi(b_i) = x_i \,\,\text{ for }\,\, 1\le i\le a,\\
\{\pi(b_i)\}_{a< i \le d} \subset \ker(\sigma_X),\\
\kappa_J(\pi(b_i)) = x_i' \,\,\text{for}\,\, a < i\le a',\\
R[G]\cdot\{\kappa_J(\pi(b_i))\}_{a'< i \le d} = \ker(\sigma_{X'}).\end{cases}\end{equation}
(For details of a similar construction see \cite[\S5.4]{bks1}.)


The differential $\psi$ of $C$ then has the following key properties.

\begin{lemma}\label{explicit psi description} Fix an integer $j$ with $1\le j\le a'$ and set $\psi _j := b_j^\ast\circ \psi$.
\begin{itemize}
\item[(i)] If $j \le a$, then $\psi_j = 0$.
\item[(ii)] If $j > a$, then $\im(\psi_j)\subseteq \mathcal{I}_R(J)$.
\item[(iii)] If $j > a$ let $\tilde\psi_j$ denote the element of $\Hom_{R[G/J]}(P^J,\mathcal{I}_R(J)/\mathcal{I}_R(J)^2)$ that sends each $c$ in $P^J$ to the image in $\mathcal{I}_R(J)/\mathcal{I}_R(J)^2$ of $\psi_j(\tilde c)$, where $\tilde c$ is any element of $P$ with $T_J(\tilde c) = c$ in $P^J = T_J(P)$, with $T_J:=\sum_{j \in J}j$. Then the restriction of $\tilde\psi_j$ to $H^1(C)^J$ is equal to ${\rm Boc}^C_{x_j'}$.
\end{itemize}  \end{lemma}

\begin{proof} Claims (i) and (ii) are easy to check directly (by using a similar argument to that in Lemma \ref{adm-rep}(ii)).

%
%

To prove claim (iii) we fix $j$ with $a < j \le a'$, set $y := x_j'$ and use the fact that ${\rm Boc}^C_{y}$ can be computed as the composite of
the connecting homomorphism in the following commutative diagram

\[\begin{CD}
@. @. @. H^1(C)^J\\ @. @. @. @V VV\\ 0 @>
>> \mathcal{I}_R(J)\otimes _{R[G]}P @> \subseteq >> P @> \cdot T_{J} >>
 P^{J}  @> >> 0\\
@. @VV {\rm id}\otimes_{R[G]}\psi V @VV\psi V @VV\psi^{J}
V\\ 0 @>
>> \mathcal{I}_R(J)\otimes _{R[G]}P @> \subseteq >> P @> \cdot T_{J} >>
 P^{J}  @> >> 0\\
@. @VV {\rm id}\otimes_{R[G]}\pi V \\ @. \mathcal{I}_R(J)\otimes_{R[G]}H^2(C)\end{CD}\]
together with the surjective homomorphism
\[ \varrho_y: \mathcal{I}_R(J)\otimes_{R[G]}H^2(C) \rightarrow (\mathcal{I}_R(J)/\mathcal{I}_R(J)^2)\otimes_{R[G]}H^2(C_J) \xrightarrow{{\rm id}\otimes c_{y}} \mathcal{I}_R(J)/\mathcal{I}_R(J)^2.\]

In this way one computes that

\begin{multline*} {\rm Boc}^C_y(c) = \varrho_y (({\rm id}\otimes_{R[G]}\pi)(\psi(\tilde c))) = \sigma_y (({\rm id}\otimes_{R[G]}\pi)(\psi_j(\tilde c)\cdot b_j))\\ = \varrho_y (\psi_j(\tilde c)\otimes_{R[G]} \pi(b_j)) = ({\rm id}\otimes c_{y})(\overline{\psi_j(\tilde c)}\otimes_{R[G]} y) = \overline{\psi_j(\tilde c)}.\end{multline*}
Here $\overline{\psi_j(\tilde c)}$ denotes the image in $\mathcal{I}_R(J)/\mathcal{I}_R(J)^2$ of $\psi_j(\tilde c)$, the second and fourth equalities both follow as a consequence of the properties (\ref{key props of basis}) and the last equality is true because $c_y(y) = 1$. This proves claim (iii).
\end{proof}

We now define $z_b$ in ${\bigwedge}_{R[G]}^d P$ to be the pre-image of $\mathcal{L}^{-1}$ under the composite isomorphism of $R[G]$-modules
$${\bigwedge}_{R[G]}^d P \stackrel{\sim}{\longrightarrow} {\bigwedge}_{R[G]}^d P \otimes {\bigwedge}_{R[G]}^d\Hom_{R[G]}(P,R[G]) = {\rm Det}^{-1}_{R[G]}(C) \xrightarrow{\vartheta_\lambda} R[G]\cdot\mathcal{L}^{-1},$$
where the first map sends each $x$ to $x \otimes {\bigwedge}_{1\leq i\leq d } b_i^\ast.$

The connection between $z_b$ and the elements $\eta_\mathcal{X}$ and $\eta_{\mathcal{X}'}$ is described in the next result. This result will be proved by applying several general results about exterior powers that are established in \cite{bks1}.

We use the homomorphism of $R[G]$-modules
\[ \N_J^d: {\bigwedge}_{R[G]}^d P \to {\bigwedge}_{R[G/J]}^d P^J\]
that is given by the $d$-th exterior power of the homomorphism $P \to P^J$ that sends each $x$ to $T_J(x):=\sum_{j \in J}j(x)$.

\begin{proposition}\label{zeta higher special}\
\begin{itemize}
\item[(i)] $\eta_\mathcal{X} = (-1)^{a(d-a)}({\bigwedge}_{a <i\leq d}\psi_i)(z_b)$.
\item[(ii)] $\eta_{\mathcal{X}'} = (-1)^{a'(d-a')}({\bigwedge}_{a' <i\leq d}\psi_i^J)(\N_J^d(z_b))$.
\end{itemize}
\end{proposition}

\begin{proof} For each $\chi$ in $\Hom(G,F^{c,\times})$ we set $r_\chi = r _{C,\chi} := {\rm dim}_{F^c}(e_\chi(F^c\cdot H^2(C)))$, where $e_\chi$ is the idempotent of $F^c[G]$ defined in Remark \ref{explicit cong}.

Then, since $X$ is a direct summand of $H^2(C)$, one has $r_\chi \ge a$ and we claim first that the homomorphism
$$\Psi_\chi:=\bigoplus_{a<i \leq d}\psi_i : e_\chi (F^c\cdot P) \longrightarrow e_\chi F^c[G]^{\oplus (d-a)}$$
is surjective if and only if $r_{\chi}=a$.

To see this we note that if $r_{\chi}=a$, then $\{ e_\chi(x) \}_{x \in \mathcal{X}}$ is a $F^c$-basis of $e_\chi (F^c\cdot H^2(C))$ and so
 (\ref{key props of basis}) implies that $e_\chi(F^c\cdot\im (\psi)) = e_\chi(F^c\cdot\ker (\pi)) = \bigoplus_{a<i\leq d}e_\chi F^c[G]b_i$. In this case, therefore, $\Psi_\chi$ is clearly surjective.

Conversely, if $r_{\chi}>a$, then $\dim_{F^c}(e_\chi(F^c\cdot\im(\psi)))=d-r_{\chi}<d-a$ and so $\Psi_\chi$ cannot be surjective.

In particular, setting $\tilde z:= ({\bigwedge}_{a<i\leq d}\psi_i)(z_b)$, the general result of \cite[Lem. 4.2]{bks1} applies in this setting to imply that
$$e_\chi (\tilde z)
\begin{cases}
\in e_\chi(F^c\cdot{\bigwedge}_{R[G]}^aH^1(C)), &\text{if $r_{\chi}=a$,} \\
=0, &\text{if $r_{\chi}>a$}
\end{cases}
$$
As this is true for all $\chi$ in $\Hom(G,F^{c,\times})$ it implies both that $\tilde z$ belongs to $F\cdot{\bigwedge}_{R[G]}^aH^1(C)$ and is such that $\tilde z = e_a\cdot \tilde z.$

On the other hand, the general formula of \cite[Lem. 4.3]{bks1} combines with our definition of $z_b$ to imply that
$$\lambda ((-1)^{a(d-a)}\tilde z)= e_a\lambda ((-1)^{a(d-a)}\tilde z) = e_a\mathcal{L}^{-1}\cdot{\bigwedge}_{x \in \mathcal{X}}x,$$
and hence, upon recalling the definition of $\eta_\mathcal{X}$, that $(-1)^{a(d-a)}\tilde z=\eta_\mathcal{X}.$

This completes the proof of claim (i) and the proof of claim (ii) is entirely similar.
\end{proof}

We next recall that the general observation of \cite[Rem. 4.8]{bks1} implies that the natural map
\begin{equation*} ({\bigcap}_{R[G/J]}^{a}H^1(C)^J) \otimes_{R}Q \to ({\bigwedge}_{R[G/J]}^{a}F^J)\otimes_{R}Q\end{equation*}
is injective and hence that the equality of Theorem \ref{mrstheorem} can be verified by computing in $({\bigwedge}_{R[G/J]}^{a}F^J)\otimes_{R}Q$.

To complete the proof of Theorem \ref{mrstheorem}(ii) it is thus enough to note that
\begin{eqnarray*}\label{mrs equalities}
(-1)^{a(a'-a)}
{\rm Boc}^C_{\mathcal{X},\mathcal{X}'}(\eta_{\mathcal{X}'}) &=&
(-1)^{a(a'-a)}({\bigwedge}_{a< j\le a'}\tilde\psi_j)((-1)^{a'(d-a')}{\bigwedge}_{a' <i\leq d}\psi_i^J)(\N_J^d z_b)) \\
&=& \widetilde{\rm Boc}(\N_J^d (z_b))\notag\\
&=& (-1)^{a(d-a)} \nu_J^{-1}(\mathcal{N}_J(({\bigwedge}_{a <i\leq d}\psi_i)(z_b))) \\
&=&   \nu_J^{-1}(\mathcal{N}_J(\eta_\mathcal{X}))\notag
\end{eqnarray*}
where $\widetilde{\rm Boc}$ denotes the composition of $(-1)^{a'(d-a')}{\bigwedge}_{a'<i \leq d}\psi_i^J$ and
$(-1)^{a(a'-a)}{\bigwedge}_{a<i \leq a'}\widetilde\psi_i$. Here the first equality in the above display is proved by combining the explicit definition of ${\rm Boc}^C_{\mathcal{X},\mathcal{X}'}$ with the
formula in Proposition \ref{zeta higher special}(ii) and the description in Lemma \ref{explicit psi description}(iii); the second equality is an immediate consequence of the definition of $\widetilde{\rm Boc}$; the third equality is true because the argument of \cite[Lem 5.21]{bks1} implies that there is a commutative diagram
\begin{equation*}\xymatrix{
{{\bigwedge}_{R[G]}^d P} \ar[d]_{{\rm N}_J^d} \ar[rrrr]^{(-1)^{a(d-a)}{\bigwedge}_{a <i \leq d}\psi_i} &  & & & {I_R(J)^{a'-a}\cdot{\bigwedge}_{R[G]}^aP} \ar[d]^{\nu_J^{-1}\circ \mathcal{N}_J} \\
{{\bigwedge}_{R[G/J]}^dP^J} \ar[rrrr]^{\widetilde{\rm Boc}} & & & & {({\bigwedge}_{R[G/J]}^aP^J) \otimes_{R} Q;}
}\end{equation*}
(which is well-defined since $\mathcal{N}_J(I_R(J)^{a'-a}\cdot{\bigwedge}_{R[G]}^aP) \subseteq \im(\nu_J)$); finally the fourth equality follows directly from the formula in Proposition \ref{zeta higher special}(i).

This completes the proof of Theorem \ref{mrstheorem}.

\section{Compactly supported $p$-adic cohomology}\label{compact-applications}

In this section we discuss the compactly supported $p$-adic cohomology of $p$-adic representations and, in particular, prove Proposition \ref{admiss constructions}. The argument presented here is a development of that given by Barrett and the first author in \cite[\S 3.1]{barretb}.

We assume throughout the notation of \S\ref{csc exam}. In particular, we assume to be given a pair $(\mathfrak{A},T)$ comprising  a $\ZZ_p$-order $\mathfrak{A}$ and a continuous $\ZZ_p[G_{k,S}]$-module $T$ that is endowed with a commuting action of $\mathfrak{A}$ with respect to which it is a projective module.

For any module, or complex of modules, $N$ over $\mathfrak{A}\times \ZZ_p[G_{k,S}]$ we write $N^\vee$ for the Pontryagin dual $\Hom_{\ZZ_p}(N,\QQ_p/\ZZ_p)$ which we endow with the contragredient action of $\ZZ_p[G_{k,S}]$ and the obvious action of $\mathfrak{A}$.

We also set $V := \QQ_p\otimes_{\ZZ_p}T$, $V^*(1) = \QQ_p\otimes_{\ZZ_p}T^*(1)$, $W := V/T$ and $W^*(1) :=
V^*(1)/T^*(1) \cong T^\vee(1)$, each regarded as endowed with the actions of $\mathfrak{A}$ and $G_{k,S}$ that are
induced from the respective actions on $T$ and $T^*(1)$.

\subsection{Generalities} If $R$ denotes either $\mathcal{O}_{k,S}, k$ or the completion $k_v$ of $k$ at a
 place $v$ and $\mathcal{F}$ is an \'etale (pro-)sheaf on $\Spec(R)$,
then we abbreviate the complex $R\Gamma_{{\rm \acute e
t}}(\Spec(R),\mathcal{F})$ and in each degree $a$ the group
$H^a_{{\rm \acute e t}}(\Spec(R),\mathcal{F})$ to
$R\Gamma(R,\mathcal{F})$ and $H^a(R,\mathcal{F})$ respectively.

\subsubsection{}For any \'etale (pro-)sheaf $\mathcal{F}$ on $\Spec
(\mathcal{O}_{k,S})$ we then define the compact support
cohomology complex $R\Gamma_c(\mathcal{O}_{k,S},\mathcal{F})$
by means of the exact triangle
\begin{equation}\label{def-tri} R\Gamma_c(\mathcal{O}_{k,S},\mathcal{F}) \to R\Gamma (\mathcal{O}_{k,S},\mathcal{F}) \to \bigoplus_{v \in
S}R\Gamma(k_v,\mathcal{F})\to
R\Gamma_c(\mathcal{O}_{k,S},\mathcal{F})[1]
\end{equation}
where the second arrow is the direct sum of the natural localisation
morphisms. For each integer $a$ we set
$H^a_{c}(\mathcal{O}_{k,S},\mathcal{F}) :=
H^a(R\Gamma_{c}(\mathcal{O}_{k,S},\mathcal{F})).$

This definition of cohomology with compact support differs from that given in \cite[Chap. II]{milne} since Milne uses Tate cohomology at
each archimedean place. To take account of this difference we define for each $v\in S_\infty$ and each continuous $G_{k_v}$-module $N$
 a complex $R\Gamma_{\Delta}(k_v,N)$ by means of the short exact sequence
\[0\rightarrow C^\bullet(k_v,N)\rightarrow C^\bullet_{\rm Tate}(k_v,N)\rightarrow R\Gamma_{\Delta}(k_v,N)[1]\rightarrow 0\]
where $C^\bullet(k_v,N)$ is the standard complex of continuous
cochains of $\Gal(\CC/k_v)$ with values in $N$, $C^\bullet_{\rm Tate}(k_v,N)$ the standard complex
computing Tate cohomology of $N$ over $\Gal(\CC/k_v)$ and the second arrow is the natural inclusion morphism.

Then the same method that is used in loc. cit. to prove the global duality theorem of \cite[Chap. II, Cor. 3.3]{milne} proves there exists a canonical exact triangle
\begin{equation}\label{duality} C(T)\to R\Gamma(\mathcal{O}_{k,S},W^*(1))^\vee[-3] \to  \bigoplus_{v \in S_\infty}R\Gamma_\Delta(k_v,T)\to \end{equation}
in $D(\mathfrak{A})$, where, as in \S\ref{csc exam}, we write $C(T)$ in place of $R\Gamma_c(\mathcal{O}_{k,S},T)$. 

\subsubsection{}\label{sha recall}If $\mathcal{F}$ denotes either of $T, V$ or $W$ as above, then for each
non-archimedean place $v$ of $k$ we write $H^1_f(k_v,\mathcal{F})$
for the finite support cohomology group that is defined by Bloch and
Kato in \cite{bk}.

We recall in particular that $H^1_f(k_v,V)$ is a
subspace of $H^1(k_v,V)$ and that $H^1_f(k_v,T)$, resp.
$H^1_f(k_v,W)$, is
 defined to be the pre-image, resp. image, of $H^1_f(k_v,V)$, under the natural map
$H^1(k_v,T) \to
  H^1(k_v,V)$, resp. $H^1(k_v,V) \to H^1(k_v,W)$.

  If $v$ is archimedean then, as $p$ is
  odd, one has $H^1(k_v,\mathcal{F}) = 0$ and so we set $H^1_f(k_v,\mathcal{F}) = 0$.

Then
  the global finite support cohomology group $H^1_f(k,\mathcal{F})$ of Bloch and
  Kato is defined by the natural exact sequence
\begin{equation}\label{h1f-def} 0 \to H^1_f(k,\mathcal{F}) \xrightarrow{\subseteq} H^1(\mathcal{O}_{k,S},\mathcal{F}) \to
\bigoplus_{v\in
S}\frac{H^1(k_v,\mathcal{F})}{H^1_f(k_v,\mathcal{F})},\end{equation}
and hence there is an induced localisation map $\lambda_{\mathcal{F}}: H^1_f(k,\mathcal{F})\to \bigoplus_{v \in
S} H^1_f(k_v,\mathcal{F})$.

The Selmer and Tate-Shafarevic groups ${\rm Sel}(T)$ and $\sha(T)$ of $T$ are then defined by Bloch and Kato in \cite{bk} to be
equal to $H^1_f(k,W)$ and the cokernel of the natural homomorphism
$H^1_f(k,V) \to H^1_f(k,W) = {\rm Sel}(T)$ respectively and so there is a canonical short exact sequence
\begin{equation*} 0 \to
\bq_p/\bz_p\otimes_{\bz_p} H^1_f(k,T) \to {\rm Sel}(T) \to
\sha(T) \to 0.\end{equation*}
In particular, since $\sha(T)$ is finite (cf. \cite[Chap. II, 5.3.5]{fpr91}) this sequence induces an identification of $\sha(T)$ with ${\rm
Sel}(T)_{\rm cotor}$.

Finally, we recall that the main result of Flach \cite{flach} is the existence of a canonical isomorphism between $\sha(T)$ and $\sha(T^*(1))^\vee$.


\subsubsection{}As a preliminary to the proof of Proposition \ref{admiss constructions}(ii) we note that a comparison between (\ref{h1f-def}) and the long exact cohomology sequence of (\ref{def-tri}) with $\mathcal{F} = T$ shows that the
localisation homomorphism $\lambda_T$ fits into a natural exact
sequence
\begin{equation}\label{first} H^1_f(k,T)\xrightarrow{\lambda_T} \bigoplus_{v \in
S}H^1_f(k_v,T) \xrightarrow{\tilde\lambda_T} H^2(C(T)) \to {\rm
cok}(\tilde\lambda_T)\to 0.\end{equation}
Now for each place $v$ in $S$, the Pontryagin dual of the
tautological exact sequence
\[ 0 \to H^1_f(k_v,T) \xrightarrow{\subseteq} H^1(k_v,T) \to
\frac{H^1(k_v,V)}{H^1_f(k_v,V)}\]
combines with the local duality isomorphism $H^1(k_v,T)^\vee\cong
H^1(k_v,W^*(1))$ and the definition of $H^1_f(k_v,W^*(1))$ to imply
that $H^1_f(k_v,T)^\vee$ is naturally isomorphic to the quotient
$H^1(k_v,W^*(1))/H^1_f(k_v,W^*(1))$. Hence, upon taking the
Pontryagin dual of (\ref{first}), and using the global duality isomorphism $H^2(C(T))^\vee \cong H^1(\mathcal{O}_{k,S},W^*(1))$ induced by the long exact cohomology sequence of (\ref{duality}), one obtains an exact
sequence
\[ 0 \to \cok(\tilde\lambda_T)^\vee \to H^1(\mathcal{O}_{k,S},W^*(1)) \xrightarrow{\tilde\lambda _T^\vee} \bigoplus_{v \in
S}\frac{H^1(k_v,W^*(1))}{H^1_f(k_v,W^*(1))}
\]
in which $\tilde\lambda _T^\vee$ identifies with the sum of the
natural localisation maps. Since ${\rm Sel}(T^*(1))$ is defined to be $\ker(\tilde\lambda_T^\vee)$ we thus
obtain an isomorphism $\cok(\tilde\lambda_T)\cong {\rm
Sel}(T^*(1))^\vee$ and so (\ref{first}) induces an exact sequence
\begin{equation}\label{useful2} 0 \to {\rm cok}(\lambda_T) \to H^2(C(T)) \to {\rm Sel}(T^*(1))^\vee
\to 0.\end{equation}

\subsection{The proof of Proposition \ref{admiss constructions}} We now assume the notation and hypotheses of Proposition \ref{admiss constructions}.

\subsubsection{}\label{admiss prelims}It is well known that $C(T)$ satisfies the conditions (ad$_1$), (ad$_2$) and (ad$_3$) of \S\ref{admissible definition} (see, for example, \cite[\S1.6.5 and \S2.1.3]{fukaya-kato}).

In addition, the long exact cohomology sequence
 of the natural exact triangle
\begin{equation*} C(T) \to
R\Gamma_c(\mathcal{O}_{k,S},V) \to
R\Gamma_c(\mathcal{O}_{k,S},W) \to
 C(T)[1]\end{equation*}
identifies $H^1(C(T))_{\rm tor}$ with
$H^0_{c}(\mathcal{O}_{k,S},W)_{\rm cotor}$ and the
 latter group vanishes because the long exact cohomology sequence of (\ref{def-tri}) implies $H^0_{c}(\mathcal{O}_{k,S},\mathcal{F})$ vanishes
 for all sheaves
 $\mathcal{F}$.

It follows that $C(T)$ also satisfies the assumption (ad$_4$) and hence belongs to $D^{\rm a}(\mathfrak{A})$, as required to prove Proposition \ref{admiss constructions}(i).

\subsubsection{}To prove Proposition \ref{admiss constructions}(ii) we set $P:= \bigoplus_{v \in S_\infty}H^0(k_v,T)$ and so must define suitable homomorphisms $\theta^1:P \to H^1(C(T))$ and $\theta^2: P \to H^2(C(T))$.

We note first that the exact triangle (\ref{def-tri}) gives rise to a canonical long exact sequence
\[ 0 \to H^0(\mathcal{O}_{k,S},T) \xrightarrow{\kappa} \bigoplus_{v \in S}H^0(k_v,T) \xrightarrow{\alpha} H^1(C(T)) \to H^1(\mathcal{O}_{k,S},T) \xrightarrow{\lambda} \bigoplus_{v \in S}H^1(k_v,T),\]
in which $\kappa$ is the diagonal inclusion map and $\lambda$ the diagonal localisation map.

This implies, in particular, that if we define $\theta^1$ to be the restriction of $\alpha$ to $P$, then $\theta^1$ is injective and there is an exact sequence
\[ 0 \to H^0(\mathcal{O}_{k,S},T) \xrightarrow{\kappa'} \bigoplus_{v \in S\setminus S_\infty} H^0(k_v,T) \to {\rm cok}(\theta^1) \to \ker(\lambda) \to 0\]
in which $\kappa'$ is the diagonal inclusion map. It follows that ${\rm cok}(\theta^1)$ is torsion-free if ${\rm cok}(\kappa')$ and  $\ker(\lambda)$ are both torsion-free. But, since $T$ is torsion-free, it is easy to check that ${\rm cok}(\kappa')$ is torsion-free whilst the (assumed) injectivity of the displayed localisation map in Proposition \ref{admiss constructions}(ii) implies directly that $\ker(\lambda)$ is torsion-free.

In addition, if we define $\theta^2$ to be the composite homomorphism of $\mathfrak{A}$-modules
\[  P \to \bigoplus_{v \in S_p}H^1_f(k_v,T) \to {\rm cok}(\lambda_T) \to H^2(C(T)),\]
where the first map is the given homomorphism $\phi$, the second is the tautological projection and the third is from (\ref{useful2}), then the exact sequence (\ref{useful2}) implies that ${\rm Sel}(T^*(1))^\vee$, and hence also $\sha(T) \cong \sha(T^*(1))^\vee$, is isomorphic to a subquotient of ${\rm cok}(\theta^2)$.

Thus, to complete the proof of Proposition \ref{admiss constructions}(ii) it suffices to define $C_{\phi}(T)$ to be the complex $D$  constructed when Lemma \ref{cone construct} is applied to the data $(C,\theta^1,\theta^2)$ and to note that the exact sequence (\ref{les sigma}) identifies $H^2(D)$ with ${\rm cok}(\theta^2)$.

\subsubsection{}\label{sigma complex} Turning to Proposition \ref{admiss constructions}(iii), we note that, as $\Sigma$ is disjoint from $S$, for each $v$ in $\Sigma$ there is a natural morphism $\theta_v$ in $D(\mathfrak{A})$ from $R\Gamma(\co_{k,S},W^*(1))$ to $R\Gamma(\kappa_v,W^*(1))$ for which $H^0(\theta_v)$ is the inclusion $H^0(\mathcal{O}_{k,S},W^*(1)) \to H^0(k_v,W^*(1)) = H^0(\kappa_v,W^*(1))$.

In particular, if $C_\Sigma(W)$ is any complex lying in an exact triangle of the form
\begin{equation*} C_\Sigma(W) \to R\Gamma(\co_{k,S},W^*(1)) \xrightarrow{\theta} \bigoplus_{v \in \Sigma}R\Gamma(\kappa_v,W^*(1)) \to \end{equation*}
with $\theta := \oplus_{v \in \Sigma}\theta_v$, then the module $H^0(C_\Sigma(W))= \ker(\oplus_{v \in \Sigma}H^0(\theta_v))$ vanishes.

Now $R\Gamma(\kappa_v,W^*(1))$ is represented by $T^\vee(1) \xrightarrow{1-\sigma_v} T^\vee(1)$, with the first term placed in degree zero. The complex $R\Gamma(\kappa_v,W^*(1))^\vee$ is thus represented by $T(-1) \xrightarrow{1-\sigma_v} T(-1)$, with the first term placed in degree $-1$, and so is isomorphic to $R\Gamma(\kappa_v,T(-1))[1]$.

In particular, upon applying the functor $M \mapsto M^\vee[-3]$ to the above triangle we obtain the central row in the following commutative diagram of exact triangles

\begin{equation}\label{pre triangle} \begin{CD}
@.  \bigoplus_{v \in S_\infty}R\Gamma_\Delta(k_v,T) @= \bigoplus_{v \in S_\infty}R\Gamma_\Delta(k_v,T)\\
@. @A AA @A AA\\
C'_\Sigma(T) @> \theta^\vee[-3]>> R\Gamma(\co_{k,S},W^*(1))^\vee[-3] @> >> C_\Sigma(W)^\vee[-3]\\
@\vert @A\delta AA @A \delta' AA\\
C'_\Sigma(T) @> \theta' >> R\Gamma_c(\co_{k,S},T) @> >> C_\Sigma(T).\end{CD}\end{equation}
Here we set $C'_\Sigma(T) := \bigoplus_{v \in \Sigma}R\Gamma(\kappa_v,T(-1))[-2]$, the central column is the exact triangle (\ref{duality}) and there exists a morphism $\theta'$ that makes the first square commute because the group $\Hom_{D(\mathfrak{A})}(C'_\Sigma(T), \bigoplus_{v \in S_\infty}R\Gamma_\Delta(k_v,T))$ vanishes (as $C'_\Sigma(T)$ is represented by a complex of projective $\mathfrak{A}$-modules that is concentrated in degrees two and three whilst $H^i(R\Gamma_\Delta(k_v,T))$ vanishes for all $v \in S_\infty$ and all $i > 0$). We then define $C_\Sigma(T)$ to be the mapping cone of $\theta'$ (so that the bottom row is an exact triangle) and $\delta'$ to be the morphism induced by $\delta$ and the commutativity of the first square (so that the third column is an exact triangle).

Since the first two complexes in the lower row of (\ref{pre triangle}) belong to $D^{\rm a}(\mathfrak{A})$ this exact triangle implies directly that $C_\Sigma(T)$ satisfies the conditions (ad$_1$) and (ad$_2$) of \S\ref{admissible definition}. The same fact combines with the long exact sequence of cohomology of the lower row of (\ref{pre triangle}) to show $C_\Sigma(T)$ satisfies (ad$_3$) and also gives an exact sequence
\[ 0 \to H^1_c(\co_{k,S},T) \to H^1(C_\Sigma(T)) \to \ker(H^2(\theta')) \to 0.\]
Then, as $\ker(H^2(\theta'))$ is a submodule of the torsion-free module $\bigoplus_{v \in \Sigma}T(-1)$, this sequence implies $H^1(C_\Sigma(T))$ is torsion-free, and hence that $C_\Sigma(T)$ satisfies (ad$_4$).

To show $C_\Sigma(T)$ belongs to $D^{\rm s}(\mathfrak{A})$ it is thus enough to note that the long exact sequence of cohomology of the right hand column in (\ref{pre triangle}) combines with the vanishing of $H^3(C_\Sigma(W)^\vee[-3]) \cong H^0(C_\Sigma(W))^\vee$ and of the groups $H^2(R\Gamma_\Delta(k_v,T))$ for $v \in S_\infty$ to imply $H^3(C_\Sigma(T))$ vanishes.

The morphism $\theta'$ and complex $C_\Sigma(T)$ in (\ref{pre triangle}) therefore have the properties stated in Proposition \ref{admiss constructions}(iii).

To complete the proof we now fix a morphism $\phi$ as in claim (ii) and define $C_{\phi,\Sigma}(T)$ to be the mapping cone of the composite morphism
\[ C'_\Sigma(T) \xrightarrow{\theta'} R\Gamma_c(\co_{k,S},T) \to C_\phi(T)\]
where the second morphism is as constructed in claim (ii).

Then, since the exact sequence (\ref{les sigma}) identifies $H^3(R\Gamma_c(\co_{k,S},T))$ with $H^3(C_\phi(T))$, one checks easily that
 $H^3(C_{\phi,\Sigma}(T))$ vanishes and hence, by the above argument, that $C_{\phi,\Sigma}(T)$ belongs to $D^{\rm s}(\mathfrak{A})$.

In addition, claim (ii) implies that ${\rm Sel}(T^*(1))^\vee$ and $\sha(T)$ are isomorphic to subquotients of $H^2(C_{\phi,\Sigma}(T))$ since $H^2(C_\phi(T))$ is isomorphic to a submodule of $H^2(C_{\phi,\Sigma}(T))$.

This completes the proof of Proposition \ref{admiss constructions}.

\section{Tate motives}\label{tate section}  As a first illustration of the arithmetic interest of our approach, in this section we apply the results obtained above to the Galois representations $\ZZ_p(r)$ for various integers $r$.

In this way we are led to, amongst other things, extend the existing theory of refined Stark conjectures (see Remark \ref{etnc-fitSel-Gm}), address a problem explicitly raised by Washington in \cite[Remark after Th. 8.2]{washington} and by Lang in \cite[p. 260]{lang} (see \S\ref{Structure of Exterior Biduals}), extend and refine the conjectures that are formulated by Castillo and Jones in \cite{cj} and by Solomon in \cite{sol3} (see Remark \ref{general B}), refine the main conjecture formulated by Kurihara and the first and second authors in \cite{bks2-2} (see Remark \ref{bksremark}) and both  extend and correct the main result proved by El Boukhari in \cite{boukhari} (see Remark \ref{boukhari}).

We shall also in this section provide proofs for Theorems A and B, as stated in the Introduction.

In a subsequent article \cite{bmc2} of Macias Castillo and the first author the same methods are also applied to formulate and study refined versions of the Birch and Swinnerton-Dyer concerning the leading  terms of Artin-Hasse-Weil $L$-series.

\subsection{The leading terms} At the outset we fix a finite abelian extension $F/k$ of number fields and set $G := \Gal(F/k)$. We also fix a finite set of places $S$ of $k$ that contains all places that are either archimedean or ramify in $F/k$ and an auxiliary finite set of places $T$ in $k$ that is disjoint from $S$.

For any homomorphism $\chi$ in $\widehat{G} :=\Hom(G,\CC^\times)$ we consider the $S$-truncated $T$-modified $L$-function of a complex variable $s$ that is defined by setting
\[L_{k,S,T}(\chi,s) := \prod_{v\in T}(1-\chi(\Fr_v)Nv^{1-s})\prod_{v\notin S}(1-\chi(\Fr_v)Nv^{-s})^{-1},\]
where $\Fr_v$ denotes the (arithmetic) Frobenius in $G$ of any place of $F$ above $v$ and $Nv$ is the cardinality of the residue field at $v$.

We then define a corresponding $\CC[G]$-valued $L$-function by setting
\[\theta_{F/k,S,T}(s):=\sum_{\chi\in\widehat{G}}L_{k,S,T}(\chi^{-1},s)e_\chi .\]
%
This function admits a meromorphic continuation to $\CC$ and its leading term at an integer $r$ is the unit of $\RR[G]$ that is given by the sum
\[\theta^*_{F/k,S,T}(r):=\sum_{\chi\in\widehat{G}}L_{k,S,T}^*(\chi^{-1},r)e_\chi.\]
where $L_{k,S,T}^*(\chi^{-1},r)$ denotes the leading term in the Laurent expansion of $L_{k,S,T}(\chi^{-1},s)$ at $s=r$.

In the following sections we shall investigate results that are obtained by applying the general approach developed in earlier sections to the special elements that are constructed by using the leading terms $\theta^*_{F/k,S,T}(r)$.

\subsection{Weight zero}\label{weight 0 sect} In \cite{bks1} Kurihara and the first and second authors describe a canonical `compactly supported, $T$-modified, Weil-\'etale cohomology' complex $R\Gamma_{c,T}((\mathcal{O}_{F,S})_\mathcal{W},\ZZ)$ for the constant sheaf $\ZZ$.

In this section we discuss the special elements that are constructed by combining this complex with the leading term $\theta^*_{F/k,S,T}(0)$.

For any $G$-module $M$ we write $M^\ast$ and $M^\vee$ for the linear dual $\Hom_\ZZ(M,\ZZ)$ and Pontryagin dual $\Hom_\ZZ(M,\QQ/\ZZ)$, each endowed with the contragredient action of $G$.

\subsubsection{}\label{recall selmer}We first recall relevant properties of the complex $R\Gamma_{c,T}((\mathcal{O}_{F,S})_\mathcal{W},\ZZ)$.

To do this we recall that the `$S$-relative $T$-trivialized integral Selmer group' ${\rm Sel}_{S}^{T}(F)$ for the multiplicative group $\GG_m$ over $F$ is defined to be the cokernel of a canonical homomorphism of $G$-modules
\[{\prod}_{w}\ZZ \longrightarrow \Hom_\ZZ(F_T^{\times},\ZZ) \]
(see \cite[Def. 2.1]{bks1} where the notation $\mathcal{S}_{S,T}(\mathbb{G}_{m/F})$ is used). Here in the product $w$ runs over all places of $F$ that do not lie above places in $S\cup T$, $F_T^\times$ is the subgroup of $F^\times$ comprising elements $u$ for which $u-1$ has a strictly positive valuation at each place above $T$ and the unlabeled arrow sends each element $(x_w)_w$ to the map $(u \mapsto \sum_{w}{\rm ord}_w(u)x_w)$ with ${\rm ord}_w$ the normalised additive valuation at $w$.

Then the group ${\rm Sel}_{S}^{T}(F)$ is a natural analogue for $\mathbb{G}_m$ of the integral Selmer groups of abelian varieties that are defined by Mazur and Tate in \cite{mt}, lies in a canonical exact sequence of $G$-modules of the form
\begin{equation}\label{can ses original} 0 \to {\rm Cl}_{S}^{T}(F)^\vee \to {\rm Sel}_{S}^{T}(F) \to (\mathcal{O}^{\times}_{F,S,T})^\ast\to 0\end{equation}
and also has a subquotient canonically isomorphic to ${\rm Cl}(F)^\vee$ (for details see \cite[Prop. 2.2, Rem. 2.3 and Prop. 2.4(ii)]{bks1}). Here ${\rm Cl}^T_S(F)$ is the ray class group modulo the product of
places of $F$ above $T$ of the subring $\mathcal{O}_{F,S}$ of $F$ comprising elements that are integral at all places outside $S$, $\mathcal{O}^{\times}_{F,S,T}$ is the group $F_T^\times \cap \mathcal{O}_{F,S}^\times$ and ${\rm Cl}(F)$ is the class group of $F$ and all duals are endowed with the contragredient action of $G$.

We also write $Y_{F,S}$ for the free abelian group on the set $S_F$ of places of $F$ above $S$ and $X_{F,S}$ for the submodule of $Y_{F,S}$ comprising elements whose coefficients sum to zero. We note that $Y_{F,S}$ is endowed with a natural action of $G$ and that, with respect to this action,  $X_{F,S}$ is a submodule.


\begin{proposition}\label{weil-etale} If the group $F_T^\times$ is torsion-free, then $C_{F,S,T} := R\Gamma_{c,T}((\mathcal{O}_{F,S})_\mathcal{W},\ZZ)$ belongs to $D^{\rm s}(\ZZ[G])$ and there are canonical identifications $H^1(C_{F,S,T})=Y_{F,S}/\Delta_S(\ZZ)$ and $H^2(C_{F,S,T}) = {\rm Sel}_S^T(F)$, where $\Delta_S$ denotes the natural diagonal map $\ZZ\to Y_{F,S}$. \end{proposition}

\begin{proof} This follows immediately from the result of \cite[Prop. 2.4]{bks1}. \end{proof}

Following this result we will assume in the rest of \S\ref{weight 0 sect} that $T$ is chosen so that $F_T^\times$ is torsion-free.

We write
\begin{equation}\label{dirichlet} \lambda_{F,S}: \RR\cdot \mathcal{O}_{F,S}^\times \cong
 \RR\cdot X_{F,S}\end{equation}
for the `Dirichlet regulator' isomorphism of $\RR[G]$-modules that sends each $u$ in $\mathcal{O}_{F,S}^\times$ to $-\sum_{w\in S_F}\log|u|_ww$. We note that, since ${\rm Cl}_S^T(F)$ is finite, the exact sequence (\ref{can ses original}) combines with Proposition \ref{weil-etale} to imply that the linear dual of $\lambda_{F,S}$ gives a canonical isomorphism of $\RR[G]$-modules
\[ \lambda^\ast_{F,S} : \RR\cdot H^1(C_{F,S,T}) \xrightarrow{\sim} \RR \cdot H^2(C_{F,S,T}).\]

We can now recall the following result (due to Greither and the first author and to Flach). In this result we write $x \mapsto x^\#$ for the  involution of $\ZZ[G]$ that inverts elements of $G$

\begin{theorem}\label{etnc-ab} If $F$ is an abelian extension of $\QQ$, then the inverse of $\theta^*_{F/k,S,T}(0)^{\#}$ is a characteristic element for the pair $(C_{F,S,T},\lambda^\ast_{F,S})$.
\end{theorem}

\begin{proof} This follows immediately from \cite[Rem. 3.3, Prop. 3.4]{bks1}. \end{proof}

\begin{remark}\label{general etnc}{\em For general finite abelian extensions $F/k$, the argument of loc. cit. shows that the assertion of Theorem \ref{etnc-ab} is equivalent to the validity of the relevant case of the equivariant Tamagawa Number Conjecture.}
\end{remark}

\begin{remark}\label{old way}{\em Set $C_{F,S,T}^\ast := R\Hom_{\ZZ[G]}(C_{F,S,T},\ZZ[G][-3])$. Then, assuming that $F_T^\times$ is torsion-free, Proposition \ref{weil-etale} combines with Lemma \ref{dual preserve} to imply $C_{F,S,T}^\ast$ belongs to $D^{\rm s}(\ZZ[G])$ and that there are identifications of $H^1(C_{F,S,T}^\ast)$ and $H^2(C_{F,S,T}^\ast)$ with $\mathcal{O}_{F,S,T}^\times$ and the transpose integral Selmer module ${\rm Sel}_{S}^{T}(F)^{\rm tr}$ that is defined in \cite[Def. 2.6]{bks1} (where it is denoted $\mathcal{S}^{\rm tr}_{S,T}(\mathbb{G}_{m/F})$). In particular, the map $\lambda_{F,S}$ induces an isomorphism of $\RR[G]$-modules $\RR\cdot H^1(C_{F,S,T}^\ast)\cong \RR\cdot H^2(C_{F,S,T}^\ast)$ and the argument used in \cite[Prop. 3.4]{bks1} shows Theorem \ref{etnc-ab} is equivalent to asserting that the inverse of $\theta^*_{F/k,S,T}(0)$ is a characteristic element for the pair $(C^\ast_{F,S,T},\lambda_{F,S})$.}
\end{remark}

\subsubsection{}\label{HigherFitWeightZero}  We write $\bf 1$ for the trivial homomorphism $G \to \QQ^{c,\times}$.  For each non-negative integer $a$ we write $\widehat G'_{S,a}$ for the subset of $\widehat{G} \setminus \{{\bf 1}\}$ comprising homomorphisms $\psi$ for which the set $S_\psi:= \{v \in S: G_v\subseteq \ker(\psi)\}$ has cardinality $a$, where we write $G_v$ for the decomposition subgroup in $G$ of a place $v$ of $k$. We then set
\[ \widehat G_{S,a} := \begin{cases} \widehat G_{S,a}' \cup \{{\bf 1}\}, &\text{ if $a=|S|-1$},\\
    \widehat G_{S,a}', &\text{ if $a\not= |S|-1$.}\end{cases}\]
Write $\widehat G_{S,(a)}$ for the union $\bigcup_{a'\ge a}\widehat G_{S,a'}$ and define idempotents of $\QQ[G]$ by setting
\[ e_{S,a} := \sum_{\psi \in \widehat G_{S,a}} e_\psi\,\,\,\,\text{ and }\,\,\,\, e_{S,(a)} := \sum_{\psi \in \widehat G_{S,(a)}}e_\psi = \sum_{a'\ge a}e_{S,a'}.\]

Since the case $F=k$ is not of much interest we shall assume, for simplicity, in the following result that $G$ is not trivial.\\


\begin{theorem}\label{Fit-Selmer-Gm} Let $\mathcal{L}$ be a characteristic element for the pair $(C_{F,S,T}, \lambda^\ast_{F,S})$. Then for any non-negative integer $a$, any subset $S_a$ of $S$ that has cardinality at least $a+1$ and contains $\bigcup_{\psi \in \widehat G_{S,a}'}S_\psi$, any $x$ in $\ZZ[G]\cap \ZZ[G]e_{S,(a)}$ and any $\Phi$ in ${\bigwedge}_{\ZZ[G]}^a\Hom_{\ZZ[G]}(\mathcal{O}_{F,S,T}^\times,\ZZ[G])$ one has
\[\Phi( x\cdot\lambda_{F,S}^{-1}(e_{S,a}\cdot(\mathcal{L}^{-1})^\#\cdot {\bigwedge}^a_{\ZZ[G]}X_{F,S})) \subseteq {\rm Fit}_{\ZZ[G]}^a({\rm Sel}_S^T(F))^\# \cap
{\rm Ann}_{\ZZ[G]}({\rm Cl}_{S_a}^T(F)).\]
\end{theorem}

\begin{proof} Set $C:=C_{F,S,T}$. We prove first that the idempotent $e_{C,a}$ defined in \S \ref{hse def sect} is equal to $e_{S,a}$. To do this, we note that the isomorphism (\ref{dirichlet}) combines with \cite[Chap. I, Prop. 3.4]{tb} to imply that for each $\chi$ in $\widehat{G}$ one has
\begin{equation}\label{Tatedim}
\dim_\CC (e_\chi(\CC\cdot(\mathcal{O}^{\times}_{F,S})^\ast))=\begin{cases}
|S_\chi| &\mbox{ if } \chi\neq {\bf 1}\\
|S|-1 &\mbox{ if } \chi={\bf 1}\end{cases},
\end{equation}
and hence that $e_\chi\cdot e_{S,a}\neq 0$ if and only if $\dim_\CC (e_\chi(\CC\cdot(\mathcal{O}^{\times}_{F,S})^\ast))=a$.
Since $H^2(C)= {\rm Sel}_{S}^T(F)$, the exact sequence (\ref{can ses original}) therefore implies $e_{C,a}=e_{S,a}$, as required.

We next recall that \cite[Prop. 2.4(ii) and (iii)]{bks1} combine to give a natural surjective homomorphism of $G$-modules ${\rm Sel}_S^T(F)\rightarrow{\rm Sel}_{S_a}^T(F)$, that our choice of $S_a$ implies ${\rm Sel}_{S_a}^T(F)$ has rank at least $a$
in each simple component of $\ZZ[G]e_{S,(a)}$ and that the sequence (\ref{can ses original}) implies $\Hom_\ZZ({\rm Cl}_{S}^{T}(F),\QQ/\ZZ)$ is the torsion subgroup of ${\rm Sel}_{S_a}^{T}(F)$.

Finally we note that, as $\ZZ[G]$ is Gorenstein, the observation in Remark \ref{get rid of x and y} implies that we may take $y=1$ in the statement of Theorem \ref{char els}(i).

Thus, by putting all of the above observations together, we may deduce from the latter result that for any ordered subsets $\{\varphi_i\}_{1\le i\le a}$ of $(\mathcal{O}_{F,S,T}^{\times})^\ast$ and $\{x_j\}_{1\le j\le a}$ of $(Y_{F,S}/\Delta_S(\ZZ))^\ast = X_{F,S}$ and any element $x$ of  $\ZZ[G]\cap \ZZ[G]e_{S,(a)}$ one has
\begin{equation}\label{key1} ({\bigwedge}_{j=1}^{j=a}x_j)( x\cdot (\lambda_{F,S}^{\ast})^{-1}(e_{S,a}\cdot\mathcal{L}^{-1}\cdot {\bigwedge}_{i=1}^{i=a}\varphi_i)) \in
{\rm Fit}_{\ZZ[G]}^a({\rm Sel}_S^T(F)) \cap {\rm Ann}_{\ZZ[G]}({\rm Cl}_{S_a}^T(F)^\vee).\end{equation}

Now it is straightforward to check $e_{S,a} = e_{S,a}^\#$, that $\ZZ[G]\cap \ZZ[G]e_{S,(a)} = (\ZZ[G]\cap \ZZ[G]e_{S,(a)})^\#$, that ${\rm Ann}_{\ZZ[G]}({\rm Cl}_{S_a}^T(F)^\vee) = {\rm Ann}_{\ZZ[G]}({\rm Cl}_{S_a}^T(F))^\#$ and that
\[ (({\bigwedge}_{j=1}^{j=a}x_j)( (\lambda_{F,S}^{\ast})^{-1}({\bigwedge}_{i=1}^{i=a}\varphi_i)))^\# =
 ({\bigwedge}_{i=1}^{i=a}\varphi_i)( \lambda_{F,S}^{-1}({\bigwedge}_{j=1}^{j=a}x_j))\]
for all choices of $x_j$ and $\varphi_i$.

To deduce the claimed result from (\ref{key1}) it is thus enough to note that all elements of ${\bigwedge}_{\ZZ[G]}^a\Hom_{\ZZ[G]}(\mathcal{O}_{F,S,T}^\times,\ZZ[G])$ and ${\bigwedge}^a_{\ZZ[G]}X_{F,S}$ can be respectively obtained as integral linear combinations of the form ${\bigwedge}_{i=1}^{i=a}\varphi_i$ and ${\bigwedge}_{j=1}^{j=a}x_j$ as the subsets $\{\varphi_i\}_{1\le i\le a}$ and  $\{x_j\}_{1\le j\le a}$ that are chosen above vary. \end{proof}

\begin{remark}\label{etnc-fitSel-Gm}{\em If $F$ is abelian over $\QQ$ (in which case Theorem \ref{etnc-ab} can be used) or, following Remark \ref{general etnc}, one assumes the relevant case of the equivariant Tamagawa number conjecture to be valid, then the element $e_{S,a}\cdot(\mathcal{L}^{-1})^\#$ in Theorem \ref{Fit-Selmer-Gm} can be taken to be the value at $s=0$ of the series $s^{-a}\cdot\theta_{F/k,S,T}(s)$.

By using this observation it can be seen that the inclusion in Theorem \ref{Fit-Selmer-Gm} refines the existing theory of higher-order `abelian Stark conjectures'. To be a little more precise, we recall that there is a very extensive theory of such conjectures that are due to Stark \cite{stark}, to Rubin \cite{R} and to Popescu \cite{Pop} amongst others and are nicely surveyed, for example, by Valli\`{e}res in \cite[\S1]{vallieres} and that all of these conjectures were recently both extended and refined by Livingstone Boomla and the first author in \cite{dbalb}. However, even the inclusion that is conjectured in \cite{dbalb} deals only with certain distinguished elements of ${\bigwedge}^a_{\ZZ[G]}X_{F,S}$ and so is more restricted than that predicted by Theorem \ref{Fit-Selmer-Gm}.}
\end{remark}

\subsubsection{}\label{Structure of Exterior Biduals} In this section, we specialise Theorem \ref{str-bidual} to the setting of Theorem \ref{etnc-ab} in the  case $F=\QQ(\zeta_f)^+$ for some integer $f$.

To do this we set $S=\{\infty\}\cup\{\ell| f\}$ and let $T$ be an auxiliary non-empty subset of places of $\QQ$ disjoint from $S$. We use the  $T$-modified cyclotomic unit
 \[ c_{F,T}:=\delta_T\cdot (1-\zeta_f)  \]
of $\mathcal{O}_{F,S,T}^\times$ with $\delta_T:=\prod_{\ell\in T}(1-\sm_\ell^{-1} \cdot \ell)$, where $\sm_\ell$ is the Frobenius element of $\ell$.

In \cite[p. 260]{lang} Lang explicitly comments that the module structure of the quotient of the group of units by the group of cyclotomic units is a `mystery'. This issue is at least partly addressed by the following result.

\begin{theorem}\label{G-stru_isom} Set $F :=\QQ(\zeta_f)^+$ and $G :=\Gal(F/\QQ)$. Then there exists a canonical perfect $G$-bilinear pairing of finite modules
\[
(\mathcal{O}_{F,S,T}^\times/\langle c_{F,T}\rangle)_{\rm tor}\times (\ZZ[G]/\Fit^1_{G}({\rm Sel}^T_S(F)^{\rm tr}))_{\rm tor} \to \QQ[G]/\ZZ[G] .\]
%
\end{theorem}

\begin{proof} We set $C ^\ast:=C^\ast_{F,S,T}$ and recall from Remark \ref{old way} that $\theta^*_{F/\QQ,S,T}(0)^{-1}$ is a characteristic element for the pair $(C^\ast,\lambda_{F,S})$.

We now set $V=\{\infty\}\subsetneq S$ and fix an archimedean place $w$ of $F$ and a (non-archimedean) place $w_0$ of $F$ that lies above a rational prime in $S\setminus V$. Then the canonical exact sequence
\begin{equation}\label{last eq}
0\rightarrow {\rm Cl}_S^T(F)\rightarrow {\rm Sel}_S^T(F)^{\rm tr}\rightarrow X_{F,S}\rightarrow 0
\end{equation}
(from \cite[Rem. 2.7]{bks1}) allows us to regard $\mathcal{X}_V:=\{w-w_0\}$ as a subset of $H^2(C^\ast)_{\rm tf}$.

In this way the classical equality
\begin{equation*}\label{derivative}L_S'(\chi,0)=-\dfrac{1}{2}\sum_{\sm\in \Gal(\QQ(\zeta_f)/\QQ)}\log|(1-\zeta_f^{\sm})(1-\zeta_f^{-\sm})|_w \chi(\sm)
\end{equation*}
(taken from, for example, \cite[p. 79]{tb}) implies that the higher special element associated to the data $C^\ast, \lambda_{F,S},\theta^*_{F/\QQ,S,T}(0)^{-1}$ and $\mathcal{X}_V$ coincides with $c_{F,T}$.

Given this, the claimed isomorphism follows directly upon applying Theorem \ref{str-bidual} to the data $(C^\ast, \lambda_{F,S},\theta^*_{F/\QQ,S,T}(0)^{-1},\mathcal{X}_V)$.\end{proof}

\begin{remark}\label{washingtonremark}{\em \
 Assume now that $f=p^n$ for some prime $p$. In this case the quotient $\mathcal{O}_{F,S,T}^\times/\langle c_{F,T}\rangle$ is finite and the $G$-module $X_{F,S}$ is free of rank one so that $e_1=1$ and the exact sequence (\ref{last eq}) implies that $\Fit_G^0({\rm Cl}_S^T(F))=\Fit_G^1({\rm Sel}_S^T(F)^{\rm tr})$. Theorem \ref{G-stru_isom} therefore gives an isomorphism of (cyclic) $G$-modules $ (\mathcal{O}_{F,S,T}^\times/\langle c_{F,T}\rangle)^\vee\cong \ZZ[G]/\Fit_G^0({\rm Cl}_S^T(F))$.

This isomorphism resolves the problem explicitly raised by Washington in \cite[Rem. following Th. 8.2]{washington} of establishing a precise connection in this case between the Galois structures of $\mathcal{O}_{F,S,T}^\times/\langle c_{F,T}\rangle$ and ${\rm Cl}_S^T(F)$. }
\end{remark}

\begin{remark}\label{remark bks}{\em It is also possible to  deduce Theorem \ref{G-stru_isom} directly from the result of Kurihara and the first two authors in \cite[Th. 7.5]{bks1} rather than via specialization of the more general result of Theorem \ref{str-bidual}.}
\end{remark}

\subsubsection{}\label{proof of A} We can now prove Theorem A, as stated in the Introduction.

To do this we note that, since $G$ is abelian, for each $G$-module $M$ there exists a $G$-module $M^\iota$ that has the same underlying set as $M$ but upon which the $G$-action is obtained by composing the given action of $G$ on $M$ with the involution $x \mapsto x^\#$ of $\ZZ[G]$ that inverts elements of $G$.

In particular, the $G$-bilinearity of the pairing in Theorem \ref{G-stru_isom} implies that if one composes this pairing with the projection map $\QQ[G]/\ZZ[G] \to \QQ/\ZZ$ that is induced by sending each element of $\QQ[G]$ to its coefficient at the identity element of $G$, then one obtains a perfect $G$-invariant bilinear pairing of the form
\[ (\mathcal{O}_{F,S,T}^\times/\langle c_{F,T}\rangle)_{\rm tor}\times (\ZZ[G]/\Fit^1_{G}({\rm Sel}^T_S(F)^{\rm tr}))^\iota_{\rm tor} \to \QQ/\ZZ .\]

To derive the pairing of Theorem A in the Introduction from this it is enough to note that the assignment $x\mapsto x^\#$ induces a natural isomorphism of $G$-modules
\[ (\ZZ[G]/\Fit^1_{G}({\rm Sel}^T_S(F)^{\rm tr}))^\iota \cong \ZZ[G]/\Fit^1_{G}({\rm Sel}^T_S(F)^{\rm tr})^\#\]
and that \cite[Lem. 2.8]{bks1} implies $\Fit^1_{G}({\rm Sel}^T_S(F)^{\rm tr})^\#$ is equal to $\Fit^1_{G}({\rm Sel}^T_S(F))$.

This completes the proof of Theorem A.

\subsection{Weight minus two}\label{weight -2} In the next two sections, we assume $S$ also contains the $p$-adic places of $k$. For each integer $r$ we consider the complexes
\[ C(r) := R\Gamma_c(\mathcal{O}_{k,S},\ZZ_p(r)_F)\]
of $D^{\rm a}(\ZZ_p[G])$ defined in \S\ref{csc exam}. If we denote by $S_{\infty}(F)$ the set of infinite places of $F$, we also set
\[Y_F(r):=\bigoplus_{w\in S_\infty(F)}H^0(F_w, \ZZ_p(r)).\]  In this section we consider in detail the case when $r=1$

\subsubsection{}\label{first -2}We write $A_F$ for the Sylow $p$-subgroup of the ideal class group of $F$ and $M_S(F)$ for the maximal pro-$p$ abelian extension of $F$ unramified outside $S$. Also for each $p$-adic place $w$ of $F$ we write $U_w$ for the (pro-$p$) group of principal units in the completion of $F$ at $w$.

Then there are canonical identifications $H^2(C(1)) = \Gal(M_S(F)/F), H^3(C(1))= \ZZ_p,$ $ \sha(\ZZ_p(1)_F) = A_F$ and for each $p$-adic place $v$ of $k$ also $H^1_f(k_v,\ZZ_p(1)_F) = \prod_{w\mid v}U_w$ where the product is over all places of $F$ above $v$.

Thus, if we set $Y_F := Y_F(1)$, then for any homomorphism
\begin{equation}\label{phi}
\phi: Y_F \to \prod_{w\mid p}U_w
\end{equation}
of $\ZZ_p[G]$-modules the construction of Proposition \ref{admiss constructions}(ii) gives an exact triangle in $D(\ZZ_p[G])$
\begin{equation}\label{canonical exact tri} Y_F[-1]\oplus Y_F[-2] \to C(1) \to C_\phi(1) \to Y_F[0]\oplus Y_F[-1]
\end{equation}
for a complex $C_\phi(1)$ in $D^{\rm a}(\ZZ_p[G])$. In addition, in this case one finds that $\sha(\ZZ_p(1)_F) = A_F$ is a quotient (rather than merely a subquotient) of $H^2(C_\phi(1))$.

We set $T_{\bf 1} := \sum_{g \in G}g$ and write $e_{\bf 1}$ for the idempotent $|G|^{-1}T_{\bf 1}$ of $\QQ_p[G]$. We set $\mathfrak{B} := \ZZ_p[G](1-e_{\bf 1})$ and also $B := \QQ_p[G](1-e_{\bf 1})$. We note that the natural short exact sequence of $\ZZ_p[G]$-modules
\begin{equation}\label{tauto ses}0 \to (T_{\bf 1}) \to \ZZ_p[G] \xrightarrow{1\mapsto 1-e_{\bf 1}} \mathfrak{B} \to 0\end{equation}
implies that for a $\ZZ_p[G]$-module $N$ there is a natural action of $\mathfrak{B}$ on the quotient $N/H^0(G,N)$.

The complex $C_\phi(1)_0 := \mathfrak{B}\otimes^{\mathbb{L}}_{\ZZ_p[G]}C_\phi(1)$ belongs to $D^{\rm a}(\mathfrak{B})$ and, since $H^3(C_\phi(1)) = H^3(C(1))$ identifies with $\ZZ_p$, the group $H^3(C_\phi(1)_0)= H^3(C(1)_0)$ is finite and the $B$-modules $\QQ_p\otimes_{\ZZ_p}H^1(C_\phi(1)_0)$ and $\QQ_p\otimes_{\ZZ_p}H^2(C_\phi(1)_0)$, and $\QQ_p\otimes_{\ZZ_p}H^1(C(1)_0)$ and $\QQ_p\otimes_{\ZZ_p}H^2(C(1)_0)$, are isomorphic (as a consequence of the respective conditions ({\rm ad}$_2$) and ({\rm ad}$_3$)).

In the sequel we write $e_*$ for the sum of all primitive idempotents $e$ in $B$ with the property that $e(\QQ_p\otimes_{\ZZ_p}H^2(C_\phi(1)_0))$, and hence also $e(\QQ_p\otimes_{\ZZ_p}H^1(C_\phi(1)_0))$, vanishes.

\begin{proposition}\label{weight -2 prop} The following claims are valid.
\begin{itemize}
\item[(i)] One has $e_*(\QQ_p\otimes_{\ZZ_p}Y_F) = e_*(\QQ_p\otimes_{\ZZ_p}H^1(C(1)_0))$ and the map
\[ \phi_{(e_*)}: e_*(\QQ_p\otimes_{\ZZ_p}H^1(C(1)_0)) \to e_*(\QQ_p\otimes_{\ZZ_p}H^2(C(1)_0))\]
induced by the composite of $\phi$ and the reciprocity map $\prod_{w\mid p}U_w \to \Gal(M_S(F)/F)$ is bijective.

\item[(ii)] Let $\mathcal{L}$ be a characteristic element of $C(1)_0$ with respect to an isomorphism $\psi$ of $B_{\CC_p}$-modules $\CC_p\otimes_{\ZZ_p}H^1(C(1)_0) \to \CC_p\otimes_{\ZZ_p}H^2(C(1)_0)$ that agrees with $\CC_p\otimes_{\QQ_p}\phi_{(e_*)}$ on $e_*(\CC_p\otimes_{\ZZ_p}H^1(C(1)_0))$.

Then the product $e_*\cdot \mathcal{L}^{-1}$ belongs to $\mathfrak{B}$ and annihilates the module $A_F/H^0(G,A_F)$.
\end{itemize}
\end{proposition}

\begin{proof} Claim (i) follows directly from our definition of $e_*$ and the long exact sequence of cohomology of the triangle (\ref{canonical exact tri}).

To prove claim (ii) we use the notation of \S\ref{control section} with $C= C_\phi(1)$, $\mathfrak{A} = \ZZ_p[G]$ and $\wrp =p\ZZ_p$.

Then, since $H^2(C_\phi(1)) = \ZZ_p$, one has $e_0 = 1-e_{\bf 1}$ and ${\rm Fit}_{\ZZ_p[G]}(H^2(C_\phi(1)))$ identifies with the augmentation ideal $I_p(G):= \ZZ_p[G]\cap B$ of $\ZZ_p[G]$.

Thus, for some $x$ in $I_p(G)\cap B^\times$, Proposition \ref{reduction} implies the existence of a complex $C_x =  C_\phi(1)_x$ in $D^{\rm s}(\mathfrak{B})$ with the property that $H^2(C_x)$ contains $H^2(C_\phi(1)_0)$ and $x\cdot{\rm Det}_{\mathfrak{B}}(C_x) = e_0\cdot{\rm Det}_{\ZZ_p[G]}(C_\phi(1)) = {\rm Det}_{\mathfrak{B}}(C_\phi(1)_0)$.

In addition, since the complex $Be_*\otimes_{\mathfrak{B}}C_x$ is acyclic one can identify $e_*\cdot{\rm Det}_{\mathfrak{B}}(C_x)^{-1}$ as a sublattice of $B$   and, with respect to this identification, the equality $e_* = e_*e_0$ implies that
\begin{multline}\label{multi inclusion} e_*\cdot {\rm Det}_{\ZZ_p[G]}(C_\phi(1))^{-1} = e_*\cdot {\rm Det}_{\mathfrak{B}}(C_\phi(1)_0)^{-1} = xe_*\cdot {\rm Det}_{\mathfrak{B}}(C_x)^{-1}\\ \subseteq x\cdot {\rm Ann}_{\mathfrak{B}}(H^2(C_x))\subseteq x\cdot {\rm Ann}_{\mathfrak{B}}(H^2(C_\phi(1)_0)) \subseteq {\rm Ann}_{\mathfrak{B}}(H^0(G,\mathfrak{B}\otimes_{\ZZ_p}H^2(C_\phi(1)))). \end{multline}
Here the first inclusion follows from the argument of Proposition \ref{almost there}, the second is obvious and the third is true because of the exact sequence (see Lemma \ref{spec seq exact}(ii))
\begin{equation*} H^2(C_\phi(1)_0) \to
 H^0(G,\mathfrak{B}\otimes_{\ZZ_p}H^2(C_\phi(1))) \to
 H^2(G,\mathfrak{B}\otimes_{\ZZ_p}H^1(C_\phi(1)))
 \end{equation*}
and the short exact sequence (\ref{tauto ses}) can be used to show  $H^2(G,\mathfrak{B}\otimes_{\ZZ_p}H^1(C_\phi(1)))$ is isomorphic to $H^3(G,(T_{\bf 1})\otimes_{\ZZ_p}H^1(C_\phi(1)))$ and so is annihilated by $x$.

Next, we note that the functor $-\otimes_{\ZZ_p}H^2(C_\phi(1))$ preserves the exactness of the sequence (\ref{tauto ses}) and hence, upon taking $G$-invariants, gives an exact sequence
\[ 0 \to H^0(G,(T_{\bf 1}\otimes_{\ZZ_p}H^2(C_\phi(1)))) \to H^0(G,\ZZ_p[G]\otimes_{\ZZ_p}H^2(C_\phi(1))) \to H^0(G,\mathfrak{B}\otimes_{\ZZ_p}H^2(C_\phi(1))).\]

Since the third module in this sequence identifies with $H^2(C_\phi(1))$ in such a way that the image of the second arrow is equal to $H^0(G,H^2(C_\phi(1)))$, we deduce from (\ref{multi inclusion}) that any element in $e_*\cdot {\rm Det}_{\mathfrak{B}}(C_\phi(1)_0)^{-1}$ belongs to $\mathfrak{B}$ and annihilates $H^2(C_\phi(1))/H^0(G,H^2(C_\phi(1)))$.

Hence, since the fact that $A_F$ is isomorphic to a quotient of $H^2(C_\phi(1))$ implies that $A_F/H^0(G,A_F)$ is isomorphic to a quotient of $H^2(C_\phi(1))/H^0(G,H^2(C_\phi(1)))$, we deduce that any such element annihilates $A_F/H^0(G,A_F)$.

It is thus enough to show that for any element $\mathcal{L}$ as in claim (ii) one has $e_*\cdot \mathcal{L} \in e_*\cdot {\rm Det}_{\mathfrak{B}}(C_\phi(1)_0)$.

To do this we combine the standard identification ${\rm Det}_{\ZZ_p[G]}(Y_F[-1]\oplus Y_F[-2]) = (\ZZ_p[G],0)$ with the exact triangle in $D(\mathfrak{B})$ that is obtained by applying $\mathfrak{B}\otimes_{\ZZ_p[G]}-$ to (\ref{canonical exact tri}) to obtain an identification of $\mathfrak{B}$-modules ${\rm Det}_{\mathfrak{B}}(C_\phi(1)_0) = {\rm Det}_{\mathfrak{B}}(C(1)_0)$ that has the following property: the restriction of $\vartheta_{\psi}$ to $e_*\cdot {\rm Det}_{\mathfrak{B}}(C(1)_0)$ coincides with the composite $e_*\cdot {\rm Det}_{\mathfrak{B}}(C(1)_0) = e_*\cdot {\rm Det}_{\mathfrak{B}}(C_\phi(1)_0) \subset e_*\cdot B$, where the inclusion is as used in (\ref{multi inclusion}).

This implies, in particular, that $e_*\cdot \vartheta_{\psi}^{-1}(\mathcal{L}) = e_*\cdot \mathcal{L}$ belongs to $e_*\cdot {\rm Det}_{\mathfrak{B}}(C_\phi(1)_0)$, as required. \end{proof}

\subsubsection{}\label{proof of B}To prove Theorem B (as stated in the introduction) it is enough to combine Proposition \ref{weight -2 prop} with previous results of Macias Castillo and the first author.

To explain this we assume $F$ is totally real. In this case the module $Y_F$ vanishes and so the complex $C(1)_\phi$ constructed in \S\ref{first -2} coincides with $C(1)$. The idempotent $e_*$ in Proposition \ref{weight -2 prop} therefore coincides with the idempotent $e_{F/k,1}$ that occurs in \cite[Th. 5.4]{omac}.

Given this observation, and the stated hypothesis that either $\mu_p(F)$ vanishes or $p$ does not divide $[F:k]$, the latter result implies  the existence of a characteristic element $\mathcal{L}$ of $C(1)_0$ with the property that
\[ e_*\cdot \mathcal{L}^{-1} = e_*\cdot \theta^p_{F/k,S}(1),\]
where $\theta^p_{F/k,S}(1)$ is the element of $B_{\CC_p}$ defined by the interpolation property (\ref{interpolation}).

To deduce Theorem B from the final assertion of Proposition \ref{weight -2 prop} it is thus enough to show that $e_*\cdot \theta^p_{F/k,S}(1)$ is equal to $\theta^p_{F/k,S}(1)$.

This is in turn equivalent to proving that $L_{p,S}(1,\rho)$ vanishes for any non-trivial homomorphism $\rho: G \to \QQ_p^{c,\times}$ for which the space $e_\rho(\QQ_p^c\otimes_{\ZZ_p}H^2(C(1)))$ does not vanish and this follows directly from the result of \cite[Th. 2.4]{dbmc2} (with $m = 1$ and $\rho$ non-trivial).

This completes the proof of Theorem B.

\begin{remark}\label{general B}{\em Without any restriction on the abelian extension $F/k$, and for any choice of homomorphism $\phi$ as in (\ref{phi}), the equivariant Tamagawa number conjecture for the pair $(h^0({\rm Spec}(F))(1),\ZZ_p[G])$ leads to a precise (conjectural) description of an element $\mathcal{L}$ of the form required by Proposition \ref{weight -2 prop}(ii). The inverse of this element is the product of the leading term $\theta^*_{F/k,S}(1)$ by a combination of natural archimedean periods and regulators, a $p$-adic regulator and a `logarithmic resolvent' associated to $\phi$. In this way, Proposition \ref{weight -2 prop} leads one to the formulation, and in the case that $F$ is an abelian extension of $\QQ$ the proof, of conjectures that both extend and refine those formulated by Castillo and Jones in \cite{cj} and by Solomon in \cite{sol3}, both of which deal only with the `minus part' of CM extensions. Further details of this aspect of the theory are given in the PhD thesis \cite{thesis} of the third author.}\end{remark}

\subsection{Other weights}\label{other-weight} Throughout this section we assume that $j\notin\{0,1\}$ and that $p$ is an odd prime. We fix an idempotent $\varepsilon$ of $\ZZ_p[G]$ an a finite set of places $T$ of $k$ that is disjoint from $S$.

\subsubsection{}Following the construction of \cite{bks2-2}, for any integer $j$, define $R\Gamma_T(\OO_{F,S}, \ZZ_p(j))$ to be the complex that lies in an exact triangle in $D(\ZZ_p[G])$ of the form
\begin{equation}\label{T-mod tri}
R\Gamma_T(\OO_{F,S}, \ZZ_p(j))\rightarrow R\Gamma(\OO_{F,S}, \ZZ_p(j))\rightarrow\bigoplus_{w\in T_F}R\Gamma (\kappa_w, \ZZ_p(j))
\end{equation}
and in each degree $i$ the `$T$-modified' \'etale cohomology of $\ZZ_p(j)$ is then defined by setting $H^i_T(\OO_{F,S},\ZZ_p(j)):=H^i(R\Gamma_T(\OO_{F,S}, \ZZ_p(j))).$ Now we define an object of $D(\ZZ_p[G]\varepsilon)$ by setting
\[C^\varepsilon_{F,S,T}(j):=\ZZ_p[G]\varepsilon \otimes^{\mathbb{L}}_{\ZZ_p[G]}(R\Gamma_T(\OO_{F,S}, \ZZ_p(1-j))\oplus Y_F(-j)[-2]).\]

We recall from \cite{bks2-2} the properties of this complex.

\begin{proposition}\label{bks22complex} If $j\notin\{0,1\}$ and $\varepsilon H^1_T(\mathcal{O}_{F,S}, \ZZ_p(1-j))$ is $\ZZ_p$-torsion-free, then $C_{F,S,T}^\varepsilon(j)$ is an object in the category $D^{\rm s}(\ZZ_p[G]\varepsilon)$ such that
\[H^i(C_{F,S,T}^\varepsilon(j)) = \begin{cases}  \varepsilon H^1_T(\mathcal{O}_{F,S}, \ZZ_p(1-j)), & \mbox{ if } i=1,\\
\varepsilon H^2_T(\mathcal{O}_{F,S},\ZZ_p(1-j))\oplus \varepsilon Y_F(-j), &\mbox{ if } i=2.\end{cases}\]
Furthermore, there exists a (non-canonical) isomorphism of $\QQ_p[G]$-modules
\[\QQ_p \otimes_{\ZZ_p}H^1(C_{F,S,T}^\varepsilon(j))\cong\QQ_p\otimes_{\ZZ_p} H^2(C_{F,S,T}^\varepsilon(j)).\]
\end{proposition}
\begin{proof}
This is proven in \cite[Lem. 4.1]{bks2-2}.
\end{proof}

Following this result, we assume to be given an isomorphism of $E[G]$-modules
\[ \lambda: E \otimes_{\ZZ_p}H^1(C_{F,S,T}^\varepsilon(j)) \cong E \otimes_{\ZZ_p}H^2(C_{F,S,T}^\varepsilon(j))\]
for some extension field $E$ of $\QQ_p$.

\begin{theorem}\label{otherweights} Fix an integer $j$ with $j\notin\{0,1\}$ and assume that $\varepsilon H^1_T(\mathcal{O}_{F,S}, \ZZ_p(1-j))$ is torsion-free. Let $\mathcal{L}$ be a characteristic element for $(C^\varepsilon_{F,S,T}(j), \lambda)$ and write $r$ for the rank of the (free) $\ZZ_p[G]\varepsilon$-module $\varepsilon\cdot Y_F(-j)$.

Then for any separable subset $\mathcal{X}$ of $H^2(C^\varepsilon_{F,S,T}(j))_{\rm tf}$ with $|\mathcal{X}|=a$, there exists an isomorphism of $\ZZ_p[G]$-modules
\[\left(\dfrac{\bigcap^a_{\ZZ_p[G]}  H^1_T(\mathcal{O}_{F,S}, \ZZ_p(1-j))  }{\ZZ_p[G]\cdot\eta_\mathcal{X}}\right)_{\rm tor}^\vee\cong \left(\dfrac{\ZZ_p[G]    }{\Fit_{\ZZ_p[G]}^{a-r}(  H^2_T(\mathcal{O}_{F,S},\ZZ_p(1-j)))^{\varepsilon_a}}\right)_{\rm tor},\]
where $\eta_\mathcal{X}$ is the higher special element of $(C_{F,S,T}^\varepsilon(j),\lambda,\mathcal{L},\mathcal{X})$ and we set $\varepsilon_a:=\varepsilon\cdot e_{C,a}$.
\end{theorem}
\begin{proof} We abbreviate $C^\varepsilon_{F,S,T}(j)$ to $C$, $e_{C,a}$ to $e_a$ and $H^i_T(\mathcal{O}_{F,S}, \ZZ_p(1-j))$ to $H^i_T$ for $i=1,2$.
To prove the claimed isomorphism, we first specialise Theorem \ref{str-bidual} to the data $(C,\lambda,\mathcal{L},\mathcal{X})$ and obtain an isomorphism of $\ZZ_p[G]$-modules
\begin{equation}\label{intermediate isom}\left(\dfrac{\bigcap^a_{\ZZ_p[G]} H^1(C)}{\ZZ_p[G]\cdot\eta_\mathcal{X}}\right)_{\rm tor}^\vee\cong \left(\dfrac{\ZZ_p[G] \varepsilon}{\Fit_{\ZZ_p[G]}^{a}(  H^2(C))^{e_a}}\right)_{\rm tor}.\end{equation}

By Proposition \ref{bks22complex} one has $H^1(C)=\varepsilon H^1_T$ and $H^2(C)=\varepsilon H^2_T\oplus \varepsilon Y_F(-j)$. In addition, Lemma \ref{torsion} implies the quotients $(\bigcap^a_{\ZZ_p[G]} \varepsilon H^1_T)/(\ZZ_p[G]\cdot\eta_\mathcal{X})$ and  $\ZZ_p[G]\varepsilon/\Fit_{\ZZ_p[G]}^{a}(  H^2(C))^{e_a}$ are equal to the torsion subgroups of $(\bigcap^a_{\ZZ_p[G]} H^1_T)/ (\ZZ_p[G]\cdot\eta_\mathcal{X})$ and $\ZZ_p[G]/\Fit_{\ZZ_p[G]}^{a}(  H^2(C))^{e_a}$ respectively.

Furthermore, since $\varepsilon\cdot Y_F(-j)$ is a free $\ZZ_p[G]\varepsilon$-module of rank $r$, a standard property of Fitting ideals implies that
\[ \Fit_{\ZZ_p[G]\varepsilon}^a(H^2(C))= \Fit_{\ZZ_p[G]\varepsilon}^{a-r}(\varepsilon H^2_T)=\varepsilon\cdot\Fit_{\ZZ_p[G]}^{a-r}(H^2_T).\]

The claimed isomorphism now follows upon combining these observations with the isomorphism in (\ref{intermediate isom}).
\end{proof}

\begin{remark}\label{bksremark}{\em In view of Theorem \ref{str-bidual}, the above isomorphism is equivalent to an equality
\begin{equation}\label{bks-conjecture}I(\eta_\mathcal{X})=\varepsilon_a\cdot\Fit_{\ZZ_p[G]}^{a-r}(  H^2_T(\mathcal{O}_{F,S},\ZZ_p(1-j)).
\end{equation} Suppose now that $a=r$, as would be implied by the validity of Schneider's conjecture \cite{ps}. Then the equivariant Tamagawa number conjecture implies that $\varepsilon\cdot\theta_{F/k,S,T}^*(j)^{-1}$ is a characteristic element for the pair ($C^\varepsilon_{F,S,T}(j)$, $\lambda_j$), where $\lambda_j$ is the period-regulator isomorphism at $j$, as defined by Kurihara, the first and second authors in \cite[\S 2.2]{bks2-2}. In particular, with $\mathcal{X}$ chosen to be the canonical basis of $\varepsilon Y_F(-j)$ that is fixed in \cite[Lem. 2.1]{bks2-2}, the higher special element $\eta_\mathcal{X}$ coincides with the Stark element of rank $r$ and weight $-2j$ defined in \cite[Def. 2.7]{bks2-2} and the equality (\ref{bks-conjecture}) recovers the prediction of \cite[Conj. 3.5]{bks2-2}.}
\end{remark}

\subsubsection{}\label{Higher algebraic $K$-groups} In this last section, we set $\ZZ'=\ZZ[1/2]$ and for each abelian group write $M'$ in place of $\ZZ'\otimes_\Z M$. We also fix an integer $m$ with $m > 1$.

We fix an integer $f$ with $f\not\equiv 2\pmod{4}$ and write $F$ for the field generated by a primitive $f$-th root of unity in $\CC$. We write  $\Sigma$ for the set of places of $\QQ$ comprising $\infty$ and all prime divisors of $f$ and set $\varepsilon^-_m:=(1-(-1)^m\tau)/2$ where $\tau$ is the complex conjugation in $G = \Gal(F/\QQ)$.

We write $\epsilon_m(\zeta_f)$ for Beilinson's `cyclotomic element' in $\QQ\otimes_\ZZ K_{2m-1}(\mathcal{O}_F)$, as described by Neukirch in
\cite[Part II, \S1]{neuk}, and then set $c_F(m) :=  2^{-1}(m-1)!f^{m-1}\cdot \epsilon_m(\zeta_f)$.

\begin{theorem}\label{G-stru_isom-neg} The element $c_F(m)$ belongs to $\varepsilon_{m}^-\cdot K_{2m-1}(\mathcal{O}_{F})'$ and there exists a canonical isomorphism of $\ZZ'[G]$-modules
\[\left(\dfrac{\varepsilon_{m}^-\cdot K_{2m-1}(\mathcal{O}_{F})'}{\ZZ'[G]\cdot c_F(m)}\right)^\vee \cong\dfrac{\ZZ'[G]\varepsilon_{m}^-}{\varepsilon_{m}^-\cdot\Fit^0_{\ZZ'[G]}(K_{2m-2}(\mathcal{O}_{F,\Sigma})')}.\]
\end{theorem}

\begin{proof} It suffices to prove the claimed isomorphism after tensoring with $\ZZ_p$ for each odd prime $p$. We fix an odd prime $p$ and write $S$ for $\Sigma\cup \{p\}$. Then the known validity of the Quillen-Lichtenbaum conjecture implies that there exists for both $k=1$ and $k=2$ a canonical Chern character isomorphism $\ZZ_p\otimes_\ZZ K_{2m-k}(\mathcal{O}_{F,\Sigma}) \cong H^k(\mathcal{O}_{F,S}, \ZZ_p(m))$.

In addition, in the case $k=1$ one has $K_{2m-k}(\mathcal{O}_{F,\Sigma}) = K_{2m-k}(\mathcal{O}_{F})$ and, by Huber and Wildeshaus \cite[Cor. 9.7]{hw}, the image of $1\otimes c_F(m)$ under the Chern character isomorphism coincides with the cyclotomic element of Deligne and Soul\'e in $H^1(\mathcal{O}_{F,S}, \ZZ_p(m))$.

Since this is true for all odd primes $p$ it implies, in particular, that the element $c_F(m)$ belongs to $\varepsilon_{m}^-\cdot K_{2m-1}(\mathcal{O}_{F})'$, as claimed.

Since the module $\varepsilon_m^-H^1(\mathcal{O}_{F,S},\ZZ_p(m))$ is $\ZZ_p$-torsion-free, Proposition \ref{bks22complex} implies that $C_{F/\QQ,S,\emptyset}^{\varepsilon_m^-}$ is an object in $D^{\rm s}(\ZZ_p[G]\varepsilon_m^-)$. In this case, $\varepsilon_m^-\theta_{F/\QQ,S,\emptyset}^*(1-m)^{-1}$ is a characteristic element for the pair $(C_{F/\QQ,S,\emptyset}^{\varepsilon_m^-},\lambda_{1-m})$ (see, for example, \cite[Cor. 4.4]{bks2-2}). Let $\mathcal{X}$ be the canonical basis of $\varepsilon_m^- Y_F(m-1)$ chosen as in \cite[Lem. 2.1]{bks2-2} and let $\eta_\mathcal{X}$ be the higher special element associated with the data $(C_{F/\QQ,S,\emptyset}^{\varepsilon_m^-},\lambda_{1-m}, \varepsilon_m^-\theta_{F/\QQ,S,\emptyset}^*(1-m)^{-1}, \mathcal{X})$. Note that the rank of $\varepsilon_m^- Y_F(m-1)$ over $\ZZ_p[G]\varepsilon_m^-$ equals to $1$ so Theorem \ref{otherweights} specialises to the above data to imply the isomorphism
\[\left(\dfrac{ \varepsilon_m^-\cdot H^1_T(\mathcal{O}_{F,S}, \ZZ_p(m))}{\ZZ_p[G]\cdot \eta_\mathcal{X}}\right)^\vee\cong \dfrac{\ZZ_p[G]\varepsilon_m^-}{\Fit_{\ZZ_p[G]}^{0}( \varepsilon_m^- \cdot H^2_T(\mathcal{O}_{F,S},\ZZ_p(m)))}.\]
Hence it suffices to show that the higher special element $\eta_\mathcal{X}$ in this context coincides with Deligne-Soul\'e's  cyclotomic element. This is proved in \cite[\S 5.1]{bks2-2}.
\end{proof}

\begin{remark}\label{boukhari}{\em The result of Theorem \ref{G-stru_isom-neg} both extends, and clarifies, the main result of El Boukhari in \cite{boukhari}. To explain this, we fix an odd prime $p$ and write $H$ for the maximal subgroup of $G$ of order prime to $p$ and $I_\ell$ for the inertia subgroup in $G$ of each rational prime $\ell$. We fix a homomorphism $\chi: H\to \mathbb{Q}_p^{c\times}$ that is non-trivial on $I_\ell\cap H$ for all primes $\ell\not= p$ that ramify in $F^H/\QQ$ and also such that $\chi(\varepsilon_m^-) = 1$. Then, setting $R_\chi := \mathbb{Z}_p[G]^\chi$, Theorem \ref{G-stru_isom-neg} implies that
\begin{multline*}  {\rm Fit}^0_{R_\chi}\left(\left(\left(\dfrac{ K_{2m-1}(\mathcal{O}_{F})\otimes_\ZZ\ZZ_p}{\ZZ_p[G]\cdot c_F(m)}\right)^\chi\right)^\vee\right)\\ = {\rm Fit}^0_{R_\chi}((K_{2m-2}(\mathcal{O}_{F})\otimes_\ZZ \ZZ_p)^\chi)     \cdot \prod_{\ell\in \Sigma\setminus \{\infty\}}(1- {\rm Fr}^{-1}_\ell\cdot \ell^{m-1})e_{I_\ell},\end{multline*}
where ${\rm Fr}_\ell$ is the Frobenius automorphism of $\ell$ in $G/I_\ell$ and $e_{I_\ell}$ the idempotent $(\#I_\ell)^{-1}\sum_{g \in I_\ell}g$.

The equality in the main result (Theorem 6.5) of \cite{boukhari} is of precisely this form but is obtained under the additional assumptions that $m$ is odd and $\chi$ is non-trivial on the intersection of $H$ with the full decomposition subgroup in $G$ of each prime that ramifies in $F$. In addition, the factor `$Q(0)$' that occurs in the statement of \cite[Th. 6.5]{boukhari} is incorrectly defined in loc. cit. and the above equality provides a corrected version. For more details of this aspect of the theory see \cite[\S 5.3.3]{thesis}.}
\end{remark}



\end{document}